\documentclass[10pt,letterpaper]{amsart}
\usepackage{fullpage}
\usepackage[toc,page,title,titletoc,header]{appendix}
\usepackage{amsthm,amsmath,amsfonts,amssymb,euscript,hyperref,graphics,color,slashed,mathrsfs,graphicx,mathtools,enumerate,cite}
\hypersetup{colorlinks, citecolor=black, urlcolor=black, linkcolor=black}
\usepackage[english]{babel}
\usepackage{bm}

\allowdisplaybreaks

\numberwithin{equation}{section}

\newtheorem{theorem}{Theorem}[section]
\newtheorem{lemma}[theorem]{Lemma}
\newtheorem{proposition}[theorem]{Proposition}
\newtheorem{corollary}[theorem]{Corollary}

\newtheorem{claim}[theorem]{Claim}
\newtheorem{remark}[theorem]{Remark}
\newtheorem{convention}[theorem]{Convention}

\title[Inverse scattering of 3D Alfv\'en waves]{Ideal MHD. Part III:\\ Inverse scattering of Alfv\'en waves in three dimensional\\ ideal magnetohydrodynamics}

\author[Mengni Li]{Mengni Li}
\address{School of Mathematics, Southeast University, No.2 SEU Road, Nanjing 211189, P.R. China}
\email{krisymengni@163.com,  lmn@seu.edu.cn}

\begin{document}
\begin{abstract}

The purpose of this paper is to solve the inverse scattering problem of nonlinear Alfv\'en waves governed by the three dimensional ideal incompressible MHD system. Bridging together geometric methods and weighted energy estimates, we establish a couple of scattering isomorphisms to substantially strengthen our previous rigidity results. 
This answer is consistent with the physical intuition that Alfv\'en waves behave exactly in the same manner as their scattering fields detected by the faraway observers.  
The novelty of the present work is twofold:  for one thing, the relationship between Alfv\'en waves emanating from the plasma and their scattering fields at infinities is explored to the best; for another thing, 
the null structure inherent in MHD equations is thoroughly examined, especially when we estimate the pressure term.

\medskip

\noindent
\textbf{Running title}:  Inverse scattering of 3D Alfv\'en waves
\end{abstract}

\maketitle
\tableofcontents

\section{Introduction}

The study of magnetohydrodynamics (MHD) concerns mutual interactions between electromagnetic fields and electrically conducting fluids, and our discussion is restricted to the ideal incompressible case. 
Due to the fact that Alfv\'en waves can propagate and return, we are interested in an inverse scattering topic to recover initial data emanating from the plasma when given scattering fields at infinities, namely the faraway traces of solutions to the MHD system. Recently, we have approached this issue in \cite{Li-Yu} by concentrating on the rigidity aspect that Alfv\'en waves must vanish if their scattering fields vanish at infinities. This current study is intended to further give a satisfactory answer to the above inverse scattering problem: the scattering operator can uniquely determine solutions to the MHD system and accordingly Alfv\'en waves can be reconstructed from knowledge of their scattering fields at infinities.

\subsection{The MHD system}

Let us consider the ideal incompressible MHD system in $\mathbb{R}^{1+3}$: 
\begin{equation}\label{MHD general}
\begin{cases}
	&\partial_t  v+ v\cdot \nabla v = -\nabla p + (\nabla\times b)\times b, \\
	&\partial_t b + v\cdot \nabla b =  b \cdot \nabla v,\\
	&\operatorname{div} v =0,\\
	&\operatorname{div} b =0,
\end{cases}\end{equation}
where $b(t,x):\mathbb{R}\times\mathbb{R}^3\to\mathbb{R}^3$ is the magnetic field, $v(t,x):\mathbb{R}\times\mathbb{R}^3\to\mathbb{R}^3$ is the fluid velocity, and $p(t,x):\mathbb{R}\times\mathbb{R}^3 \to\mathbb{R}$ is the fluid pressure (this is a scalar function). 

\smallskip

We first proceed by decomposing the Lorentz force term $(\nabla\times b)\times b$ into 
\[\underbrace{(\nabla\times b)\times b}_{\text{Lorentz force}}=\underbrace{\nabla(-\tfrac{1}{2}| b|^2)}_{\sim \text{ magnetic pressure}}+\underbrace{\ b \cdot \nabla  b\ }_{\sim \text{ magnetic tension}}\]
and using $p$ again in the place of the total pressure $p+\tfrac{1}{2}| b|^2$. This provides convenience for rephrasing the momentum equation (the first equation in \eqref{MHD general}) as 
\begin{equation*}
   \partial_t  v+ v\cdot \nabla v = -\nabla p+ b \cdot \nabla b.
\end{equation*}
Particular emphasis should be given to the magnetic tension term $b \cdot \nabla b$, which is the only restoring force to produce the Alfv\'en waves discovered by the Swedish Nobel laureate Hannes Alfv\'en \cite{Alfven}. As exhaustively discussed in \cite{Davidson}, the phenomena of the Alfv\'en waves have been successfully employed to investigate a host of topics in astrophysics, such as star formation, sunspots, solar flares, solar winds and so on. We also refer the interested readers to \cite{Davidson,Priest-Forbes} for example to obtain a more complete presentation. 

\smallskip
Throughout the rest of this paper, we shall follow the convention that $B_0=(0,0,1)$ often represents a strong background magnetic field of the system.  We focus on the most interesting situation where a strong background magnetic field $B_0$ presents and thus a small initial perturbation will generate Alfv\'en waves which propagate along $B_0$. To be specific, let us first diagonalize \eqref{MHD general} in terms of the Els\"{a}sser variables $Z_+ = v +b$ and $Z_- = v-b$. 
This leads to 
\begin{equation}\label{MHD in Elsasser}\begin{cases}
	&\partial_t  Z_+ +Z_- \cdot \nabla Z_+ = -\nabla p, \\
	&\partial_t  Z_- +Z_+ \cdot \nabla Z_- = -\nabla p,\\
	&\operatorname{div} Z_+ =0,\\
	&\operatorname{div} Z_- =0.
\end{cases}\end{equation}
Then the fluctuations $z_+=Z_+-B_0$ and $z_-=Z_-+B_0$ propagate along $B_0$ in opposite directions and obey the following system: 
\begin{equation}\label{eq:MHD}\begin{cases}
	&\partial_{t}z_{+}+Z_-\cdot \nabla z_{+} =-\nabla p,\\
	&\partial_{t}z_{-}+Z_+\cdot \nabla z_{-} =-\nabla p,\\
	&\operatorname{div} z_{+}=0,\\
	&\operatorname{div} z_{-}=0.
\end{cases}\end{equation}

Recall that the curl of a vector field $f$ on $\mathbb{R}^3$  is defined as 
$\operatorname{curl} f = \big(\varepsilon_{ijk}\partial_if^j\big) \partial_k$, where $\varepsilon_{ijk}$ is a totally anti-symmetric symbol associated to the volume form of $\mathbb{R}^3$ and repeated indices are understood as summations. Now we define $j_+=\operatorname{curl} z_+$ and $j_-=\operatorname{curl} z_-$. 
Taking the curl of \eqref{eq:MHD} subsequently yields the following system of equations for $(j_{+},j_{-})$:
\begin{equation}\label{eq:curlMHD-2}\begin{cases}
	&\partial_{t}j_{+}+Z_-\cdot \nabla j_{+} =-\nabla z_{-}\wedge\nabla z_{+},\\
	&\partial_{t}j_{-}+Z_+\cdot \nabla j_{-} =-\nabla z_{+}\wedge\nabla z_{-},\\
	&\operatorname{div} j_{+}=0,\\
	&\operatorname{div} j_{-}=0,
\end{cases}\end{equation}
where the nonlinear terms on the right hand side can be explicitly written as 
\begin{equation*}
\begin{cases}
	&\nabla z_- \wedge \nabla z_+ =\big(\varepsilon_{ijk}\partial_i z_-^l\partial_l z_+^j\big)\partial_k, \\ 
	&\nabla z_+ \wedge \nabla z_- =\big(\varepsilon_{ijk}\partial_i z_+^l\partial_l z_-^j\big)\partial_k.
	\end{cases}
\end{equation*}

\medskip

\subsection{Motivation of the problem and Overview}

Much work has been done in the past concerning the small-data-global-existence of incompressible MHD with strong magnetic backgrounds. In particular, the space dimension three scenario, which nature gives preference, is not only the most important but also the most challenging. 
Contributions on this subject can be traced back at least to \cite{Bardos-Sulem-Sulem}, where Bardos, Sulem and Sulem obtained global solutions in the H\"older space for the ideal case by means of the convolution with fundamental solutions. For the case with strong fluid viscosity, global existence results were achieved by Lin, Xu and Zhang through Fourier method in the seminal works such as \cite{Lin-Zhang,Xu-Zhang}  in the sense that the smallness of the data relies on the viscosity.  Soon thereafter, He, Xu and Yu synthesized the idea of wave equations and the method of weighted energy estimates to prove the global nonlinear stability for both the ideal case and the case with small fluid viscosity in \cite{He-Xu-Yu}, where the size of initial data does not depend on viscosity. 
Their proof is inspired by the work \cite{Christodoulou-Klainerman} on the nonlinear stability of Minkowski spacetime while alternative proofs on the similar result can be found in \cite{Cai-Lei,Wei-Zhang-2017,Wei-Zhang-2018}. 
The aforementioned works pave the way for many new developments on global well-posedness and long time behavior of MHD; see  \cite{Abidi-Zhang,Li-Wu-Xu,Li-Yu,Xu} for instance.

\smallskip

This paper is still bound up with the global existence for MHD and on this basis advances the approach to investigating the scattering behavior of global solution. 
Let us mention that the scattering theory (leading to rigidity) for Alfv\'en waves governed by the MHD system is the central issue of our previous paper \cite{Li-Yu}, 
which came as a surprise to most researchers since there were no analogous constructions before. 
Precisely, we  proved that the scattering fields of Alfv\'en waves are well-defined by the traces of the solution at characteristic infinities, and more importantly, a couple of rigidity from infinity theorems can be constructed as follows: 
\begin{align*}
&\textit{the Alfv\'en waves must vanish everywhere}\\
&\textit{if their scattering fields vanish at infinities.}
\end{align*}
This statement is rather striking due to its consistency with the following physical intuition: 
\begin{align*}
&\textit{there are no Alfv\'en waves at all emanating from the plasma}\\
&\textit{if no waves are detected by the faraway observers.}
\end{align*}
Essentially speaking, our rigidity results indeed reflect a sense of uniqueness that we can recover the vanishing initial Alfv\'en waves from the vanishing scattering fields. 
The natural extension to this rigidity phenomenon is whether we can recover the initial Alfv\'en waves from  whatever the scattering fields are.       
However, the proof in \cite{Li-Yu} has not yielded this kind of promising inverse scattering results. 

\smallskip

Luckily, the scattering theory in the case of wave equations is tractable and has been studied extensively in the literature \cite{Friedlander80,Lax-Phillips,Isakov} together with their cited references and citing references. 
A number of instances are formulated therein, among the simplest and the most classical of which are free waves, i.e. smooth solutions to the linear wave equations $\Box\phi=0$ in $\mathbb{R}^{1+n}$. 
The standard approach to understand the corresponding scattering theory is to use the  Radon transform inversion formula. 
It is then proved that free waves can be reconstructed from the explicit expression of scattering fields. 
This remarkable result serves as inspiration for the developments in the nonlinear setting, especially including  \cite{Li-2021} where we manage to construct an inverse scattering theorem for one-dimensional semilinear wave equations verifying the null conditions.  
It essentially says that 
\begin{align*}
&\textit{the one-dimensional semilinear
waves behave exactly in the same manner as }\\
&\textit{their scattering fields detected
by the far-away observers.}
\end{align*}
Compared with the scattering theory in earlier works, our contribution reflects the following two highlights. 
One is that reasoning therein is merely based on the one-dimension geometric constructions and the weighted energy estimates (rather than microlocal analysis techniques and Carleman estimates, when compared with most studies on scattering). 
The other is tightly related to the one-dimension null structure, namely left-traveling waves are coupled only with right-traveling waves: these two kinds of waves propagate in opposite directions with the increasing spatial distance so that the decay of null form in time is obtainable.

\smallskip

We note that the dynamic behavior of nonlinear Alfv\'en waves is similar to that of one-dimensional waves, and the MHD system in strong magnetic backgrounds contains nonlinear terms with null structure. As such, it is of great interest to know if our previous construction proposed for the model  $\Box\phi=Q(\partial\phi,\partial\phi)$ in $\mathbb{R}^{1+1}$ can help us solve the inverse scattering problem of Alfv\'en waves more naturally. 
The ideas above motivate at least attempts to work with a suitable modification of the framework used in \cite{Li-2021}.  
Nevertheless, the estimates required are considerably more involved and present novel technical challenges due to the quasilinear nature of the MHD system. 
Furthermore, it is reasonable to conjecture that
\begin{equation}\label{conjecture}\begin{split}
&\textit{the Alfv\'en waves behave exactly in the same manner as }\\
&\textit{their scattering fields detected
by the far-away observers.}
\end{split}\end{equation}
This ultimate conjecture will be elaborately stated in Theorem \ref{thm:inversetheorem}. 
Though there are two recent works pointing in this direction (i.e. \cite{He-Xu-Yu} for the construction of scattering operator, and \cite{Li-Yu} in the sense of rigidity or uniqueness), completely satisfactory results are still missing as witnessed. 
Thereby, the principal goal of this paper is to further study this issue and give an affirmative answer to thoroughly explore the relationship between the scattering fields and the initial data of Alfv\'en waves.

\bigskip
\noindent
\textbf{Structure of the paper.}
In Section \ref{sec:mainresults}, we introduce basic notations and present the main results. Our results consist of three main building blocks: 
global existence and weighted energy estimates, construction of scattering scattering fields and their weighted Sobolev spaces,  scattering isomorphisms between scattering fields and their initial data. 
Section \ref{sec:preliminary} includes bootstrap assumptions to start the proof and several preliminary lemmas used throughout this paper. 
Here we remark that the pressure estimates, as the main technical ingredient and the most involved part of this paper, are independently dealt with and fully collected in this section. 
The bootstrap argument  is then completed in Section \ref{sec:energy}, which allows us to construct the weighted energy estimates and the global solution. 
The last two sections are devoted to the construction of scattering fields and in particular to the construction of scattering isomorphisms. 
In addition,  Section \ref{sec:inverse} contains a detailed proof for the main inverse scattering theorem with all four cases discussed. Finally,  we end this paper with some remarks on our inverse scattering theorem.

\bigskip

\noindent
\textbf{Acknowledgement.}
The author started on this project in Spring 2019 and would like to thank Professor Pin Yu for many enlightening discussions as well as for unfailing support over the past years.  
This work was supported in part by the Natural Science Foundation of Jiangsu Province (Grant No. BK20220792); in part by the National Natural Science Foundation of China (Grant Nos. 12171267 and 12201107); and in part by the Zhishan Youth Scholar Program of Southeast University.

\bigskip
\section{The core set-up and Main results}\label{sec:mainresults}

\subsection{Notations and Global solutions}
We will follow in this work most of the conventions and notations used by \cite{He-Xu-Yu,Li-Yu}: 
\begin{enumerate}[(i)]
\smallskip
\item 
We have two  characteristic (space-time)
vector fields $L_{+}$ and $L_{-}$ as 
\begin{equation}\label{vector}
L_{+}=\partial_t+Z_+\cdot\nabla,\ \ \ \  L_{-}=\partial_t+Z_-\cdot\nabla.
\end{equation}
Let $\widetilde{\operatorname{div}}$ denote the divergence of $\mathbb{R}^4$ with respect to the standard Euclidean metric. 
As a result of $\operatorname{div}Z_\pm=0$, there holds $\widetilde{\operatorname{div}}L_\pm=0$.

\medskip

\item 
We have two characteristic functions $u_+=u_+(t,x)$ and $u_-=u_-(t,x)$ as 
\begin{equation}\label{definitionforu}
\begin{cases}
&L_+ u_+ = 0,\\
&u_+\big|_{\Sigma_0} = x_3,
\end{cases}\ \ \ \ \ \ 
\begin{cases}
	&L_- u_- = 0,\\
	&u_-\big|_{\Sigma_0} = x_3.
\end{cases}
\end{equation}
Similarly, we also make the following definition for $x_i^\pm=x_i^\pm(t,x)$ with $i=1,2$:
\begin{equation}\label{definitionforxi}
	\begin{cases}
		&L_+ x_i^+ = 0,\\
		&x_i^+\big|_{\Sigma_0} = x_i,
	\end{cases}\ \ \ \ \ \ 
	\begin{cases}
		&L_- x_i^- = 0,\\
		&x_i^-\big|_{\Sigma_0} = x_i.
	\end{cases}
\end{equation}
It is clear that $x_3^\pm=u_\pm$ if we let $i=3$ in \eqref{definitionforxi}. 

\medskip

\item 
We shall use $C^{+}_{u_+}$ and $C^{-}_{u_-}$ to denote the level sets of the characteristic functions $u_+$ and $u_-$. Precisely, for given real numbers $u_{+,0}$, $u_{-,0}$ and $t_0$, we define 
the following characteristic hypersurfaces: 
\begin{align*}
	&C^{+}_{u_{+,0}}=\big\{(t,x)\in \mathbb{R}\times \mathbb{R}^3\big|u_+(t,x) =u_{+,0}\big\},\\
	&C^{-}_{u_{-,0}}=\big\{(t,x)\in \mathbb{R}\times \mathbb{R}^3\big|u_-(t,x) =u_{-,0}\big\},\\
	&C^{+,t_0}_{u_{+,0}}=\big\{(t,x)\in [0,t_0]\times \mathbb{R}^3\big|u_+(t,x) =u_{+,0}\big\},\\
	&C^{-,t_0}_{u_{-,0}}=\big\{(t,x)\in [0,t_0]\times \mathbb{R}^3\big|u_-(t,x) =u_{-,0}\big\}.
\end{align*}
By these constructions, we note that the vector fields $L_+$ and $L_-$ are the normals of (also parallel to) $C^{+}_{u_+}$ and $C^{-}_{u_-}$, respectively. Moreover, the foliations $\big\{C^{+}_{u_+}\,\big|u_+\in\mathbb{R}\,\big\}$ and $\big\{C^{-}_{u_-}\,\big|u_-\in\mathbb{R}\,\big\}$
indeed form a double characteristic foliation of the space-time $\mathbb{R}^{1+3}$. As we will see in the next sections, energy fluxes through these two characteristic foliations can help us handle the inadequacy of usual energies associated to the natural time foliation $\big\{\Sigma_t\,\big|\,t\in\mathbb{R}\big\}$ in some situations.

\medskip

\item 
Two space-time weight functions are defined as
\begin{equation}\label{weight}
\langle u_+\rangle=(R^2+|u_+|^2)^{\frac{1}{2}},\ \ \ \ \ \ \langle u_-\rangle=(R^2+|u_-|^2)^{\frac{1}{2}},
\end{equation}
where $R$ is a positive constant to be determined later on (e.g. $R=100$ suffices). 
It is easy to check that 
$L_\pm \langle u_\pm\rangle =0$. We fix a small number $\delta>0$ once for all in this paper (e.g. $\delta=0.1$ suffices) and let $\omega=1+\delta$.  We define the basic energy norm through $\Sigma_t$ and the basic flux norm through $C^{\pm,t}_{u_{\pm}}$ as 
\begin{equation*}
	E_{\pm}(t)=\int_{\Sigma_t}\langle u_{\mp}\rangle^{2\omega}| z_{\pm}|^2,\ \ \ \ \ \ F_{\pm}(t)=\sup_{u_{\pm}\in\mathbb{R}}\int_{C^{\pm,t}_{u_{\pm}}}\langle u_{\mp}\rangle^{2\omega}|z_{\pm}|^2,
\end{equation*}
where the integrals should be understood as 
\begin{equation*}
\int_{\Sigma_{t}}f :=\int_{\Sigma_{t}}f(t,x)dx,\ \ \ \ \ \ \int_{C^{\pm,t}_{u_{\pm}}}f:=\int_{C^{\pm,t}_{u_{\pm}}}f(t,x)d\sigma_{\pm}.
\end{equation*}
Given a multi-index $\alpha=(\alpha_1,\alpha_2,\alpha_3)$ of order $|\alpha|=\alpha_1+\alpha_2+\alpha_3$ with $\alpha_i \in \mathbb{Z}_{\geqslant 0}$ ($i=1,2,3$), we regard $\nabla^\alpha$ as  $\partial^{\alpha_1}_{x_1}\partial^{\alpha_2}_{x_2}\partial^{\alpha_3}_{x_3}$ and denote $f^{(\alpha)}:=\nabla^\alpha f$. 
Then the higher order energy and flux are defined as 
\begin{equation*}
E^{(\alpha)}_{\pm}(t)=\int_{\Sigma_t}\langle u_{\mp}\rangle^{2\omega}\big|\nabla z^{(\alpha)}_{\pm}\big|^2,\ \ \ \ \ \  F^{(\alpha)}_{\pm}(t)=\sup_{u_{\pm}\in\mathbb{R}}\int_{C^{\pm,t}_{u_{\pm}}}\langle u_{\mp}\rangle^{2\omega}\big|j^{(\alpha)}_{\pm}\big|^2.
\end{equation*}
Let $t^*$ be the lifespan of solutions to \eqref{eq:MHD}. We also define the total energy norms and total flux norms indexed by a number $k\in \mathbb{Z}_{\geqslant 0}$ as follows: 
\begin{align*}
	&E_{\pm} =\displaystyle\sup_{0\leqslant t\leqslant t^{*}}E_{\pm}(t),\ \ \ \ \ \ \ \ \ \ \ \ \ \ F_{\pm} =\displaystyle F_{\pm}(t^*), \\
	&E^{k}_{\pm}=\sup_{0\leqslant t\leqslant t^{*}}\sum_{|\alpha|=k}E^{(\alpha)}_{\pm}(t),\ \ \ \ \ \  F^{k}_{\pm}=\sum_{|\alpha|=k}F^{(\alpha)}_{\pm}(t^*).
\end{align*}
\end{enumerate}

\medskip

With these notations in hand, a slightly modified version of Theorem 1.3 in \cite{He-Xu-Yu}, which constructs the global solution to \eqref{eq:MHD}, can be stated as follows. One may refer to Section \ref{sec:energy} for a more concise proof.

\begin{theorem}[Global existence for ideal MHD with main \textit{a priori} estimate]\label{thm:global existence}
	Let $R\geqslant 100$, $\delta \in(0,\frac{2}{3})$ and $N_* \in \mathbb{Z}_{\geqslant 5}$. There exists a universal constant $\varepsilon_0\in(0,1)$ such that for all constants $0<\varepsilon\leqslant \varepsilon_0$, if the initial data $\big(z_+(0,x),z_-(0,x)\big)$ of \eqref{eq:MHD} satisfy
	\begin{equation}\label{eq:initial energy}
		\mathcal{E}^{N_*}(0) =\sum_{+,-}\sum_{k=0}^{N_*+1}\big\|(R^2+|x_3|^2)^{\frac{1+\delta}{2}}\nabla^{k} z_{\pm}(0,x)\big\|_{L^2(\mathbb{R}^3)}^2\leqslant\varepsilon^2,
	\end{equation}
	then the ideal MHD system \eqref{eq:MHD}  admits a unique global solution $\big(z_+(t,x),z_-(t,x)\big)$. 
	Moreover, there is a universal constant $C$ such that the solution $\big(z_+(t,x),z_-(t,x)\big)$ enjoys the following energy estimate:
	\begin{equation*}
	\sum_{+,-}\bigg(E_{\pm}+F_{\pm}+\sum_{k=0}^{N_{*}}E_{\pm}^k+\sum_{k=0}^{N_{*}}F_{\pm}^k\bigg)\leqslant C \mathcal{E}^{N_*}(0).
	\end{equation*}
\end{theorem}

\medskip

\subsection{Infinities and Scattering fields}

Now we come to the main subject of this paper by defining the scattering fields (radiation fields) associated to the solution $\big(z_+(t,x),z_-(t,x)\big)$. 
The starting point is their required geometric objects, i.e. the characteristic lines and the characteristic infinities. 
The geometric constructions can be read off easily from Figure \ref{fig:infinities}. 
\begin{figure}[ht]
	\centering
	\includegraphics[width=3.3in]{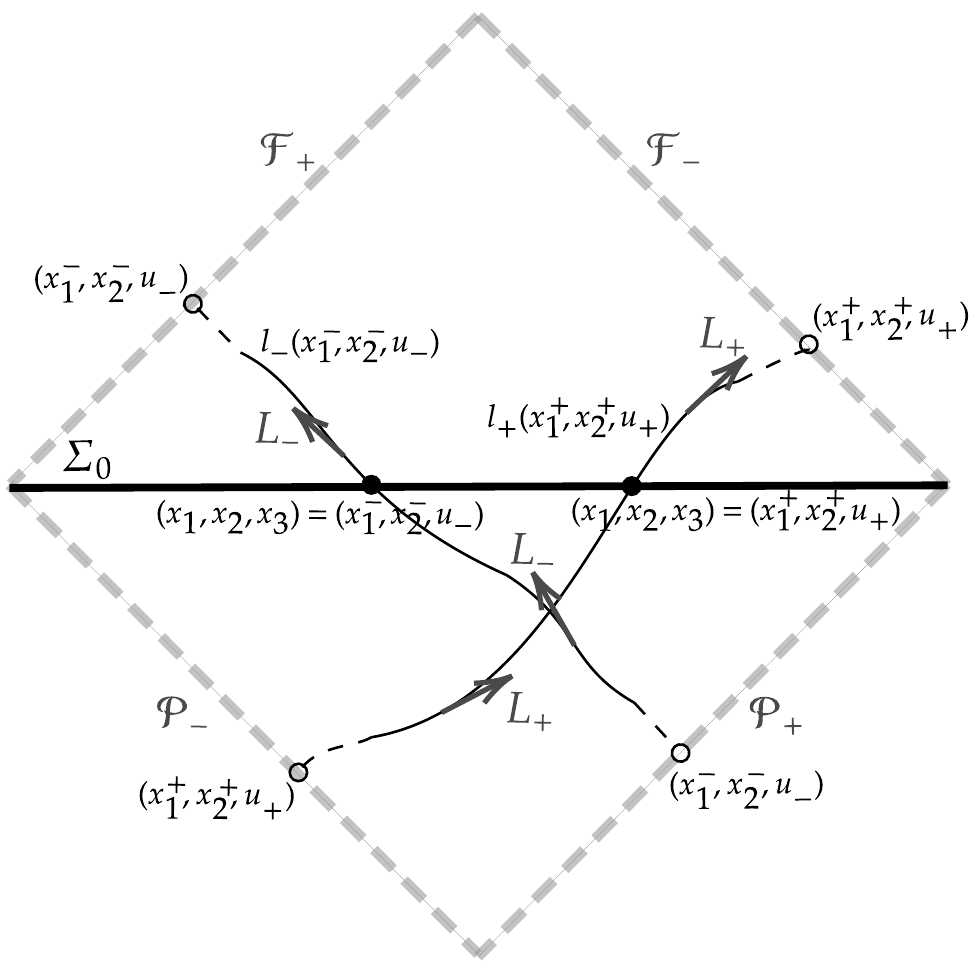}
	\caption{Traveling lines and infinities}
	\label{fig:infinities}
\end{figure}
\begin{enumerate}[(i)]
\medskip
\item 
Given a point $(x_1,x_2,x_3) \in \Sigma_0$, there is a unique characteristic line $l_-$  
(left-traveling to the future and right-traveling to the past)  along the tangent vector field $L_-$,  which can be parameterized by 
$l_-: \mathbb{R}\rightarrow \mathbb{R}\times \mathbb{R}^3$, $ t\mapsto (t,x_1^-,x_2^-,u_-)$
and satisfies $(x_1^-,x_2^-,u_-)\big|_{l_-}\equiv (x_1,x_2,x_3)$.  We also denote this line by $l_-(x_1^-,x_2^-,u_-)$. It follows that  $l_-(x_1^-,x_2^-,u_-)\subset C^-_{u_-}$. 
Similarly, we can define the characteristic line $l_+(x_1^+,x_2^+,u_+)$ (right-traveling to the future and left-traveling to the past).  

\medskip

\item 
Let us denote the collection of all characteristic lines by 
\begin{align*}
&\mathcal{F}_+=\big\{l_-(x_1^-,x_2^-,u_-)\,\big| \,t=+\infty,\,(x_1^-,x_2^-,u_-)\in \mathbb{R}^3\big\},\\
&\mathcal{F}_-=\big\{l_+(x_1^+,x_2^+,u_+)\,\big| \,t=+\infty,\,(x_1^+,x_2^+,u_+)\in \mathbb{R}^3\big\},\\
&\mathcal{P}_+=\big\{l_-(x_1^-,x_2^-,u_-)\,\big| \,t=-\infty,\,(x_1^-,x_2^-,u_-)\in \mathbb{R}^3\big\},\\ &\mathcal{P}_-=\big\{l_+(x_1^+,x_2^+,u_+)\,\big| \,t=-\infty,\,(x_1^+,x_2^+,u_+)\in \mathbb{R}^3\big\}.
\end{align*}
We call them as the left future characteristic infinity, the right future characteristic infinity, the right past characteristic infinity and the left past characteristic infinity respectively. 
We remark here that $\mathcal{F}_+$ and $\mathcal{P}_+$ can be regarded as differentiable manifolds if we use $(x_1^-,x_2^-,u_-)$ as a fixed global coordinate system on them, and so as $\mathcal{F}_-$ and $\mathcal{P}_-$ if we use $(x_1^+,x_2^+,u_+)$. Moreover, these four infinities are exactly the spaces where the scattering fields live.

\medskip

\item As depicted in the following Figure \ref{fig:integration},  
\begin{figure}[ht]
	\centering
	\includegraphics[width=5.6in]{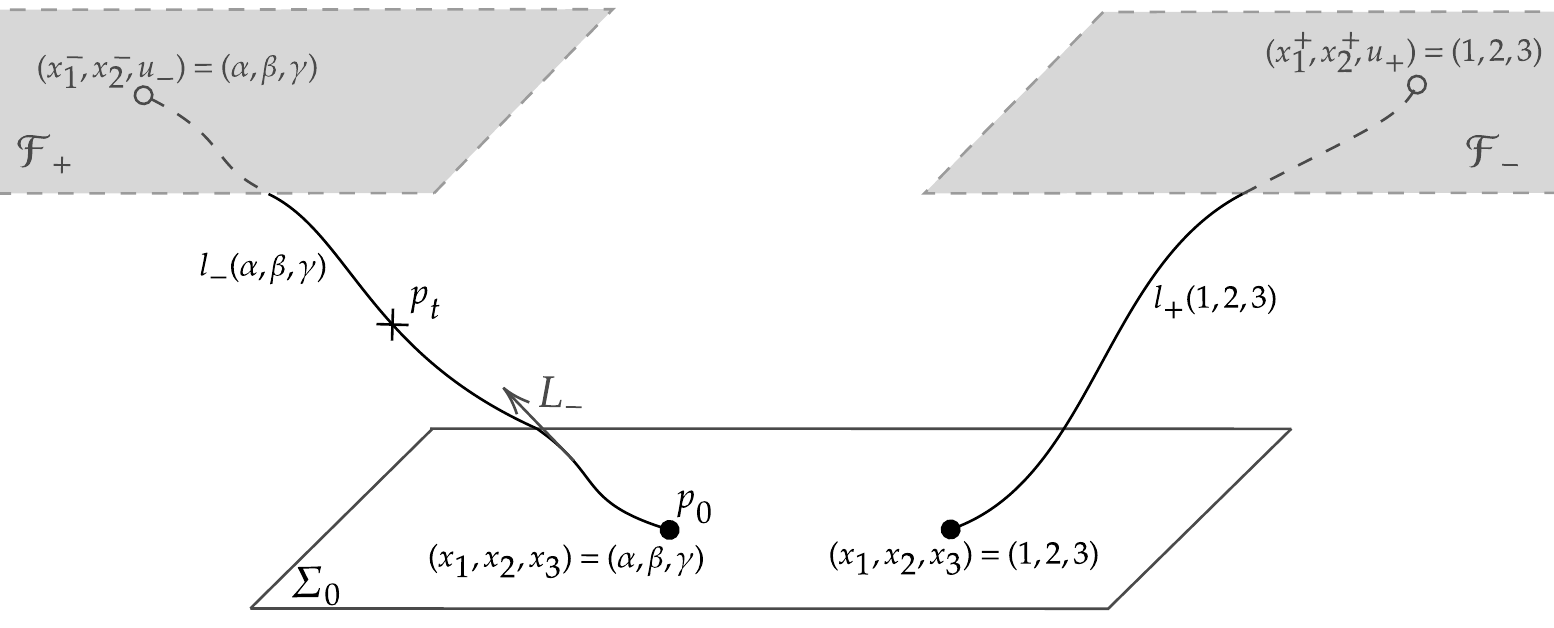}
	\caption{Integration along characteristic lines}
	\label{fig:integration}
\end{figure}
for a fixed point $p_0=(\alpha,\beta,\gamma)$ on $\Sigma_0$, we see that the characteristic line $l_-(\alpha,\beta,\gamma)$ passes through it and also hits $\mathcal{F}_+$ at the point $(x_1^-,x_2^-,u_-)=(\alpha,\beta,\gamma)$. 
Consider a point $p_{t}=(t,\alpha,\beta,\gamma)$ on $l_-(\alpha,\beta,\gamma)$ with $t\geqslant 0$. According to the first equation of \eqref{eq:MHD}, we have
\begin{equation*}
\nabla_{L_-}\big(z_{+}(\tau,\alpha,\beta,\gamma)\big) =-\nabla p(\tau,\alpha,\beta,\gamma).
\end{equation*}
Integrating this equation along the segment	of $l_-(\alpha,\beta,\gamma)$ between $p_0$ and $p_t$ then gives the expression    
\begin{equation*}
	z_+(t,\alpha,\beta,\gamma)=z_+(0,\alpha,\beta,\gamma)-\int_0^{t}\nabla p(\tau,\alpha,\beta,\gamma)\sqrt{1+|z_--B_0|^2}
	(\tau,\alpha,\beta,\gamma)d\tau.
\end{equation*}
Here we have used the fact that the measure of  $l_-(\alpha,\beta,\gamma)$ with respect to time $t$ (parameter of curve) can be written as 
 $|L_-|=\sqrt{1+|Z_-|^2}=\sqrt{1+|z_--B_0|^2}$ formally. Indeed, all the vector fields ($z_+$ and $z_-$ with various initial points and/or at various times) have the corresponding integral expressions. 
\end{enumerate}

\medskip

It is tempting to think that these integral expressions converge as $t\to\pm\infty$ and naively lead to explicit formulas of the scattering fields on
$\mathcal{F}_+$, $\mathcal{F}_-$, $\mathcal{P}_+$ and $\mathcal{P}_-$ respectively:
\begin{equation}\label{scatteringfields}
	\begin{cases}
		&\displaystyle z_+(+\infty;x_1^-,x_2^-,u_-):=z_+(0,x_1^-,x_2^-,u_-)-\int_0^{+\infty} \big(
		\nabla p\cdot\sqrt{1+|z_--B_0|^2}\big)(\tau,x_1^-,x_2^-,u_-)d\tau,\\
		&\displaystyle z_-(+\infty;x_1^+,x_2^+,u_+):=z_-(0,x_1^+,x_2^+,u_+)-\int_0^{+\infty} \big(
		\nabla p\cdot\sqrt{1+|z_++B_0|^2}\big)
		(\tau,x_1^+,x_2^+,u_+)d\tau,\\
		&\displaystyle z_+(-\infty;x_1^-,x_2^-,u_-):=z_+(0,x_1^-,x_2^-,u_-)-\int_0^{-\infty} \big(
		\nabla p\cdot\sqrt{1+|z_--B_0|^2}\big)(\tau,x_1^-,x_2^-,u_-)d\tau,\\
		&\displaystyle z_-(-\infty;x_1^+,x_2^+,u_+):=z_-(0,x_1^+,x_2^+,u_+)-\int_0^{-\infty} \big(
		\nabla p\cdot\sqrt{1+|z_++B_0|^2}\big)
		(\tau,x_1^+,x_2^+,u_+)d\tau.
	\end{cases}
\end{equation}
\smallskip
This expectation can be fulfilled in the following result. 
\begin{theorem}[Existence of scattering fields on infinities]\label{thm:existence} For the solution $\big(z_+(t,x),z_-(t,x)\big)$ constructed in Theorem \ref{thm:global existence}, all the  integrals in formulas of 
\eqref{scatteringfields}
converge. Thus, the vector fields $z_+(+\infty;x_1^-,x_2^-,u_-)$, $z_-(+\infty;x_1^+,x_2^+,u_+)$, $z_+(-\infty;x_1^-,x_2^-,u_-)$ and  $z_-(-\infty;x_1^+,x_2^+,u_+)$ in \eqref{scatteringfields} are well-defined, i.e. point-wisely defined, and we call them as the left future scattering field, the right future scattering field, the right past scattering field and the left past scattering field respectively. 
\end{theorem}

\smallskip

The previous constructions of scattering fields on infinities are shown in the following Figure \ref{fig:inverse}. 
\begin{figure}[ht]
	\centering
	\includegraphics[width=3.5in]{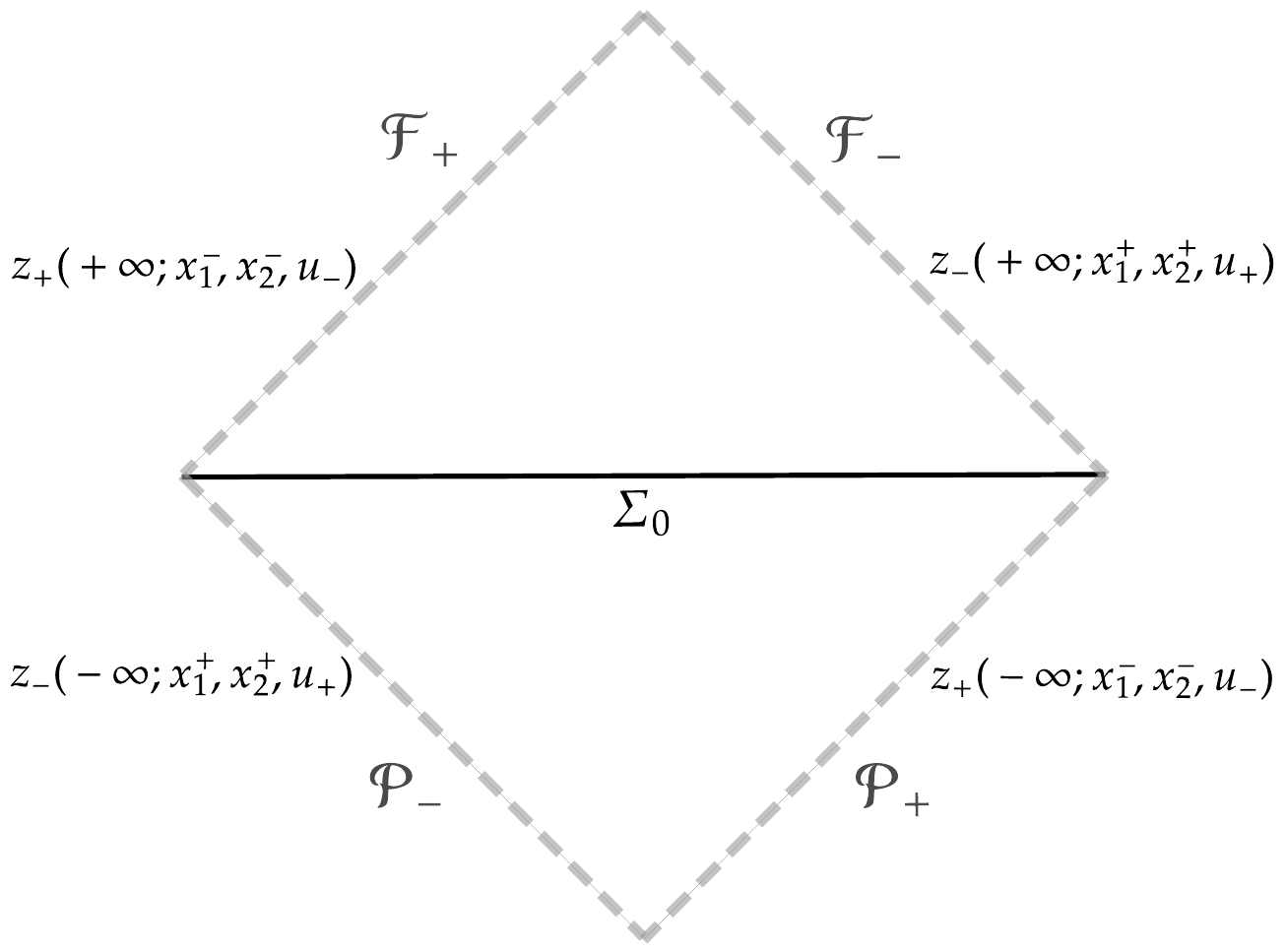}
	\caption{Scattering fields at infinities}
	\label{fig:inverse}
\end{figure}

Using the coordinate systems on infinities, we can assume that the measures on $\mathcal{F}_+$ and $\mathcal{P}_+$ are $d\mu_-:=dx_1^-dx_2^-du_-$ while the measures on $\mathcal{F}_-$ and $\mathcal{P}_-$ are $d\mu_+:=dx_1^+dx_2^+du_+$ (all up to universal constants). 
Based on this observation, we are able to formulate the following two theorems which both analyze the information of the weighted Sobolev spaces where the scattering fields live at infinities. 

\begin{theorem}[Weighted Sobolev spaces for scattering fields]\label{thm:weightSobolev}
The scattering fields constructed in Theorem \ref{thm:existence} live in the weighted Sobolev spaces as follows:
\begin{equation}\label{spaces}
	\begin{cases}
		&\displaystyle z_+(+\infty;x_1^-,x_2^-,u_-)\in H^{N_*+1}(\mathcal{F}_+,\langle u_-\rangle^{2\omega}d\mu_-),\\
		&\displaystyle z_-(+\infty;x_1^+,x_2^+,u_+)\in H^{N_*+1}(\mathcal{F}_-,\langle u_+\rangle^{2\omega}d\mu_+),\\
		&\displaystyle z_+(-\infty;x_1^-,x_2^-,u_-)\in H^{N_*+1}(\mathcal{P}_+,\langle u_-\rangle^{2\omega}d\mu_-),\\
		&\displaystyle z_-(-\infty;x_1^+,x_2^+,u_+)\in H^{N_*+1}(\mathcal{P}_-,\langle u_+\rangle^{2\omega}d\mu_+).
	\end{cases}
\end{equation}
\end{theorem}

\begin{theorem}[Deviation of  scattering fields from linear propagating fields] \label{thm:deviation}
Let the initial data of \eqref{eq:MHD} satisfy \[\|z_{\pm}(0,x_1^\mp,x_2^\mp,u_\mp)\|_{H^{N_*+1}(\Sigma_0,\langle u_\mp\rangle^{2\omega}d\mu_\mp)} \leqslant \varepsilon,\] where $N_*\geqslant 5$ and $0<\varepsilon\leqslant\varepsilon_0$ with $\varepsilon_0$ determined in Theorem \ref{thm:global existence}.  Then it holds for the scattering fields constructed in Theorem \ref{thm:existence} and their corresponding linear propagating fields that 
\begin{equation}\label{eqdeviation}
	\begin{cases}
		&\displaystyle \big\|z_+(+\infty;x_1^-,x_2^-,u_-)-z_+(0,x_1^-,x_2^-,u_-)\big\|_{H^{N_*+1}(\mathcal{F}_+,\langle u_-\rangle^{2\omega}d\mu_-)}=O(\varepsilon^2),\\
		&\displaystyle \big\|z_-(+\infty;x_1^+,x_2^+,u_+)-z_-(0,x_1^+,x_2^+,u_+)\big\|_{H^{N_*+1}(\mathcal{F}_-,\langle u_+\rangle^{2\omega}d\mu_+)}=O(\varepsilon^2),\\
		&\displaystyle \big\|z_+(-\infty;x_1^-,x_2^-,u_-)-z_+(0,x_1^-,x_2^-,u_-)\big\|_{H^{N_*+1}(\mathcal{P}_+,\langle u_-\rangle^{2\omega}d\mu_-)}=O(\varepsilon^2),\\
		&\displaystyle \big\|z_-(-\infty;x_1^+,x_2^+,u_+)-z_-(0,x_1^+,x_2^+,u_+)\big\|_{H^{N_*+1}(\mathcal{P}_-,\langle u_+\rangle^{2\omega}d\mu_+)}=O(\varepsilon^2).
	\end{cases}
\end{equation}
\end{theorem}

We point out that we have recovered all the derivatives of scattering fields at infinities in these two theorems, which close the corresponding possible gaps existing in \cite{He-Xu-Yu,Li-Yu}. In particular, Theorem \ref{thm:deviation} also provides a careful description for the asymptotic behavior of Alfv\'en waves at infinities. 

\medskip

\subsection{Scattering isomorphisms and Inverse scattering}
Let us identify $\Sigma_0$ with infinities by the coordinates $(x_1,x_2,x_3)\mapsto (x_1^\mp,x_2^\mp,u_\mp)$ where $u_\mp=x_3^\mp$, and consider the scattering fields associated to the initial data  
$\big(z_+(0,x_1^-,x_2^-,u_-),z_-(0,x_1^+,x_2^+,u_+)\big)$. 

\smallskip
At the moment, we start to pursue the scattering operators by collecting information of Alfv\'en waves (both $z_+$ and $z_-$ are required) from two adjacent infinities ($\mathcal{F}_+\cup \mathcal{F}_-$ or $\mathcal{F}_-\cup \mathcal{P}_+$ or $\mathcal{P}_+\cup \mathcal{P}_-$ or $\mathcal{P}_-\cup \mathcal{F}_+$).  
For completeness and clarity in the presentation, we make four sets of  definitions $(a)-(d)$ according to these four possibilities. In each set, we use $\mathcal{L}^{(\#)}$ to denote the scattering operator in the linear setting or the linear propagating operator (without loss of generality, consider $z_\pm$ as the solution to  $\partial_{t}z_{\pm}\mp B_0\cdot \nabla z_{\pm} =-\nabla p$ or $\partial_{t}z_{\pm}\mp B_0\cdot \nabla z_{\pm} \approx 0$, i.e.  $z_\pm$ propagate along the straight vector fields $\partial_t\mp B_0$ with the traces as a linear version of \cite{Li-Yu}), while the notation $\mathcal{N}^{(\#)}$ represents the nonlinear scattering operator governed by \eqref{eq:MHD} (i.e.  $z_\pm$ propagate along the (curved) characteristic vector fields $L_\mp$ defined in \eqref{vector}). 

\smallskip
A more precise formulation is as follows: 
\smallskip
\begin{enumerate}[($a$)]
\item $H^{N_*+1}(\Sigma_0,\langle u_-\rangle^{2\omega}d\mu_-)\times H^{N_*+1}(\Sigma_0,\langle u_+\rangle^{2\omega}d\mu_+)\rightarrow H^{N_*+1}(\mathcal{F}_+,\langle u_-\rangle^{2\omega}d\mu_-) \times  H^{N_*+1}(\mathcal{F}_-,\langle u_+\rangle^{2\omega}d\mu_+),$
\begin{align*}
\ \ \mathcal{L}^{(a)}:\  \big(z_+(0,x_1^-,x_2^-,u_-),z_-(0,x_1^+,x_2^+,u_+)\big)&\mapsto \big(z_+(0,x_1^-,x_2^-,u_-),z_-(0,x_1^+,x_2^+,u_+)\big),\\
\ \ \mathcal{N}^{(a)}:\  \big(z_+(0,x_1^-,x_2^-,u_-),z_-(0,x_1^+,x_2^+,u_+)\big)&\mapsto \big(z_+(+\infty;x_1^-,x_2^-,u_-),z_-(+\infty;x_1^+,x_2^+,u_+)\big);
\end{align*}
\item $H^{N_*+1}(\Sigma_0,\langle u_-\rangle^{2\omega}d\mu_-)\times H^{N_*+1}(\Sigma_0,\langle u_+\rangle^{2\omega}d\mu_+)\rightarrow H^{N_*+1}(\mathcal{P}_+,\langle u_-\rangle^{2\omega}d\mu_-) \times  H^{N_*+1}(\mathcal{F}_-,\langle u_+\rangle^{2\omega}d\mu_+)$,
\begin{align*}
	\ \ \mathcal{L}^{(b)}:\  \big(z_+(0,x_1^-,x_2^-,u_-),z_-(0,x_1^+,x_2^+,u_+)\big)&\mapsto \big(z_+(0,x_1^-,x_2^-,u_-),z_-(0,x_1^+,x_2^+,u_+)\big),\\
	\ \ \mathcal{N}^{(b)}:\  \big(z_+(0,x_1^-,x_2^-,u_-),z_-(0,x_1^+,x_2^+,u_+)\big)&\mapsto \big(z_+(-\infty;x_1^-,x_2^-,u_-),z_-(+\infty;x_1^+,x_2^+,u_+)\big);
\end{align*}
\item $H^{N_*+1}(\Sigma_0,\langle u_-\rangle^{2\omega}d\mu_-)\times H^{N_*+1}(\Sigma_0,\langle u_+\rangle^{2\omega}d\mu_+)\rightarrow H^{N_*+1}(\mathcal{P}_+,\langle u_-\rangle^{2\omega}d\mu_-) \times  H^{N_*+1}(\mathcal{P}_-,\langle u_+\rangle^{2\omega}d\mu_+)$,
\begin{align*}
	\ \ \mathcal{L}^{(c)}:\  \big(z_+(0,x_1^-,x_2^-,u_-),z_-(0,x_1^+,x_2^+,u_+)\big)&\mapsto \big(z_+(0,x_1^-,x_2^-,u_-),z_-(0,x_1^+,x_2^+,u_+)\big),\\
	\ \ \mathcal{N}^{(c)}:\  \big(z_+(0,x_1^-,x_2^-,u_-),z_-(0,x_1^+,x_2^+,u_+)\big)&\mapsto \big(z_+(-\infty;x_1^-,x_2^-,u_-),z_-(-\infty;x_1^+,x_2^+,u_+)\big);
\end{align*}
\item $H^{N_*+1}(\Sigma_0,\langle u_-\rangle^{2\omega}d\mu_-)\times H^{N_*+1}(\Sigma_0,\langle u_+\rangle^{2\omega}d\mu_+)\rightarrow H^{N_*+1}(\mathcal{F}_+,\langle u_-\rangle^{2\omega}d\mu_-) \times  H^{N_*+1}(\mathcal{P}_-,\langle u_+\rangle^{2\omega}d\mu_+)$,
\begin{align*}
	\ \ \mathcal{L}^{(d)}:\  \big(z_+(0,x_1^-,x_2^-,u_-),z_-(0,x_1^+,x_2^+,u_+)\big)&\mapsto \big(z_+(0,x_1^-,x_2^-,u_-),z_-(0,x_1^+,x_2^+,u_+)\big),\\
	\ \ \mathcal{N}^{(d)}:\  \big(z_+(0,x_1^-,x_2^-,u_-),z_-(0,x_1^+,x_2^+,u_+)\big)&\mapsto \big(z_+(+\infty;x_1^-,x_2^-,u_-),z_-(-\infty;x_1^+,x_2^+,u_+)\big).
\end{align*}
\end{enumerate}

\medskip

With the above motivations in mind, our ultimate conjecture \eqref{conjecture} begins to be explicitly identified by the deviation result in Theorem \ref{thm:deviation}.
We are now ready to state the main inverse scattering result. 

\begin{theorem}[Inverse scattering theorem for Alfv\'en waves]\label{thm:inversetheorem}
Let the initial data of \eqref{eq:MHD} satisfy \[\|z_{\pm}(0,x_1^\mp,x_2^\mp,u_\mp)\|_{H^{N_*+1}(\Sigma_0,\langle u_\mp\rangle^{2\omega}d\mu_\mp)} \leqslant \varepsilon,\] where $N_*\geqslant 5$ and $0<\varepsilon\leqslant\varepsilon_0$ with $\varepsilon_0$ determined in Theorem \ref{thm:global existence}. It follows that 
\begin{align*}
d \,\mathcal{N}^{(\#)}	\big|_{\mathbf{0}}=\mathcal{L}^{(\#)},\ \ \text{where }  \#=a,b,c,d,
\end{align*}
that is, the differential (or the linearization) of $\mathcal{N}^{(\#)}$ at $\mathbf{0} \in H^{N_*+1}(\Sigma_0,\langle u_-\rangle^{2\omega}d\mu_-)\times H^{N_*+1}(\Sigma_0,\langle u_+\rangle^{2\omega}d\mu_+)$ is equal to $\mathcal{L}^{(\#)}$ with $(\#)$ ranging over all situations $(a)$-$(d)$. Therefore, $\mathcal{N}^{(*)}$ is a local diffeomorphism at $\mathbf{0} \in H^{N_*+1}(\Sigma_0,\langle u_-\rangle^{2\omega}d\mu_-)\times H^{N_*+1}(\Sigma_0,\langle u_+\rangle^{2\omega}d\mu_+)$.

\end{theorem}

\smallskip

This inverse scattering result explores the relationship between the nonlinear scattering theory and  the linear propagating theory of \eqref{eq:MHD}, and eventually leads to the relationship between the scattering fields and the initial data of Alfv\'en waves. It follows that the initial data of Alfv\'en waves and even the Alfv\'en waves themselves can be uniquely determined  by their scattering fields on infinities. They indeed share the same dynamical wave 
behavior. 
As a consequence, 
our ultimate conjecture \eqref{conjecture} on the inverse scattering problem of three dimensional Alfv\'en waves can be addressed.  
In our view, the main value of our argument lies in the answer to this conjecture and its accounting for the physical interpretation that one can reconstruct Alfv\'en waves emanating from the plasma from knowledge of their detecting waves received
and measured by far-away observers. 

\medskip

Our task of the rest paper has thus been reduced to prove the five theorems above, and we will complete this task in next sections. 
The approach combines and further develops the  important ideas from \cite{He-Xu-Yu,Li-Yu,Li-2021}:  the MHD equations in strong magnetic backgrounds contain nonlinear terms with null structure; for that reason, we are able to synthesize the idea of wave equations and the weighted energy estimates.
A crucial idea to always keep in mind is to  translate the point-wise properties of the scattering fields at infinities to the weighted energy conditions for solutions at a large finite time.

\bigskip

\section{Preparing the bootstrap and Pressure estimates}\label{sec:preliminary}
Let us devote this section to making abundant preparations for the energy method and necessary estimates. 
To keep our discussion as simple as possible, we shall proceed to study the issue concerning global solutions by merely taking the future time (i.e. $t\geqslant 0$) into account. 
We shall remark in passing that this simplification is based on the symmetry of time and will be adopted throughout the rest of this paper.

\smallskip

\subsection{Bootstrap ansatz}

We first fix a positive integer $N_* \geqslant 5$.
To use the method of continuity, we assume that there exists a $t^*>0$ such that the following holds:
\begin{enumerate}[\bf{Ansatz} 1]
\item (On the amplitude of fluctuation)
Initially we assume 
\begin{equation}\label{Bootstrap on fluctuation}
\|z_\pm\|_{L^\infty}\leqslant \frac{1}{2}.
\end{equation}

\item
 (On the underlying geometry) There exists a universal constant $C_0$ such that 
\begin{equation}\label{Bootstrap on geometry}
	\Big|\Big(\frac{\partial x_i^\pm}{\partial x_j}\Big)-\mathrm{I}\Big|\leqslant 2 C_0 \varepsilon,\ \ \ \ 
	 \Big|\nabla\Big(\frac{\partial x_i^\pm}{\partial x_j}\Big)\Big|\leqslant 2 C_0 \varepsilon,
\end{equation}
where $\mathrm{I}$ is the $3\times 3$ identity matrix.

\item
 (On the total energy bound) There exists a universal constant $C_1$ such that 
\begin{equation}\label{Bootstrap on energy}
	\sum_{+,-}\bigg(E_{\pm}+F_{\pm}+\sum_{k=0}^{N_{*}}E_{\pm}^k+\sum_{k=0}^{N_{*}}F_{\pm}^k\bigg)\leqslant 2\big(C_1\big)^2\varepsilon^2.
\end{equation}
\end{enumerate}
It should be noted that the latter two sets of bootstrap assumptions are legitimate. This is because \eqref{Bootstrap on geometry}-\eqref{Bootstrap on energy} hold for the initial data, which further allows their correctness for at least a short time interval $[0,t^*]$. 

\medskip
To close the continuity argument, we need to show that there exists a universal constant $\varepsilon_0\in(0,1)$ such that for all constants $0<\varepsilon\leqslant \varepsilon_0$, the constant $\tfrac{1}{2}$ in \eqref{Bootstrap on fluctuation} can be improved to its half $\frac{1}{4}$ and the constant $2$ in \eqref{Bootstrap on geometry}-\eqref{Bootstrap on energy} all can be improved to its half $1$, i.e. under assumptions \eqref{Bootstrap on fluctuation}-\eqref{Bootstrap on energy}, better bounds can be obtained:
\begin{enumerate}[\bf{Goal} 1]
	\item (On the amplitude of fluctuation) Improving the amplitude ansatz \eqref{Bootstrap on fluctuation} to 
	\begin{equation}\label{Bootstrap on fluctuation improved}
		\|z_\pm\|_{L^\infty}\leqslant \frac{1}{4}.
	\end{equation}
	\item  (On the underlying geometry) 
	Improving the geometry ansatz \eqref{Bootstrap on geometry} to 
	\begin{equation}\label{Bootstrap on geometry improved}
		\Big|\Big(\frac{\partial x_i^\pm}{\partial x_j}\Big)-\text{I}\Big|\leqslant  C_0 \varepsilon,\ \ \ \  \Big|\nabla\Big(\frac{\partial x_i^\pm}{\partial x_j}\Big)\Big|\leqslant  C_0 \varepsilon.
	\end{equation}
	
	\item  (On the total energy bound) 
	Improving the energy ansatz \eqref{Bootstrap on energy} to 
	\begin{equation}\label{Bootstrap on energy improved}
		\sum_{+,-}\bigg(E_{\pm}+F_{\pm}+\sum_{k=0}^{N_{*}}E_{\pm}^k+\sum_{k=0}^{N_{*}}F_{\pm}^k\bigg)\leqslant \big(C_1\big)^2\varepsilon^2.
	\end{equation}
\end{enumerate} 

\smallskip
We emphasize that the aforementioned constants $C_0$, $C_1$ and $\varepsilon_0$ are independent of the lifespan $[0,t^*]$. Therefore, the assumptions \eqref{Bootstrap on geometry}-\eqref{Bootstrap on energy} will never be saturated so that we can continue $t^*$ to $+\infty$. To put it in another way, the global existence of solutions to \eqref{eq:MHD} will follow from the method of continuity once we accomplish the above three goals. 
Very nicely, the question only remains to prove \eqref{Bootstrap on fluctuation improved}-\eqref{Bootstrap on energy improved}  under assumptions \eqref{Bootstrap on fluctuation}-\eqref{Bootstrap on energy}.

\medskip

\subsection{Technical preliminaries}
This subsection is a collection of conventions and lemmas in subsequent considerations.

\begin{convention}
The notation $A\lesssim B$ means that there is a universal constant $C$ such that $A\leqslant CB$.
\end{convention}

Among the preliminary lemmas, the following weighted div-curl lemma features as a fundamental tool to control the gradient of vectors by their divergence and curl. 
We remark that the readers can consult Lemma 2.6 in \cite{He-Xu-Yu} and Lemma 2.3 in \cite{Li-Yu} for a divergence-free version of this lemma. 
\begin{lemma}[Weighted div-curl lemma]\label{lemma:divcurl}
	Let $\lambda(x)\geqslant 1$ 
	be a smooth positive function on $\mathbb{R}^3$ with the additional property $|\nabla\lambda|\lesssim \lambda$. For any smooth vector field $v(x)\in H^1(\mathbb{R}^3)$, we have 
	\begin{equation}\label{eq:d-c1}
		\big\|\sqrt\lambda\nabla v\big\|_{L^2(\mathbb{R}^3)}^2 \lesssim\big\|\sqrt\lambda\operatorname{div }v\big\|_{L^2(\mathbb{R}^3)}^2+ \big\|\sqrt\lambda\operatorname{curl }v\big\|_{L^2(\mathbb{R}^3)}^2+ \big\|\sqrt\lambda v\big\|_{L^2(\mathbb{R}^3)}^2,
	\end{equation}
	provided $\sqrt{\lambda}v\in L^2({\mathbb{R}^3})$ and $\sqrt{\lambda}\nabla v\in L^2({\mathbb{R}^3})$.
\end{lemma}
\begin{proof}
Based on the vector calculus identity 
\begin{equation*}
	-\Delta v=-\nabla(\operatorname{div}v)+\operatorname{curl}\operatorname{curl}v,
\end{equation*}
we begin by multiplying both sides by $\lambda v$ and then integrating over $\mathbb{R}^3$ to derive
\begin{equation}\label{multiply}
	-\int_{\mathbb{R}^3}\lambda v\cdot \Delta vdx=-\int_{\mathbb{R}^3}\lambda v\cdot \nabla(\operatorname{div}v)dx+\int_{\mathbb{R}^3}\lambda v\cdot \operatorname{curl}\operatorname{curl}vdx.
\end{equation}
After integration by parts, these three terms in \eqref{multiply} can be written respectively as:
	\begin{align*}
		&-\int_{\mathbb{R}^3}\lambda v\cdot \Delta vdx		
		=\int_{\mathbb{R}^3} \lambda|\nabla v|^2dx+\int_{\mathbb{R}^3}\nabla \lambda\cdot v\cdot \nabla vdx,\\
		&-\int_{\mathbb{R}^3}\lambda v\cdot \nabla(\operatorname{div}v)dx
		=\int_{\mathbb{R}^3} \lambda|\operatorname{div} v|^2dx+\int_{\mathbb{R}^3}\nabla \lambda\cdot  v\cdot \operatorname{div} vdx,\\
		&\int_{\mathbb{R}^3}\lambda v\cdot \operatorname{curl}\operatorname{curl}vdx
		=-\int_{\mathbb{R}^3} \lambda|\operatorname{curl} v|^2dx-\int_{\mathbb{R}^3}(\nabla \lambda\wedge v)\cdot \operatorname{curl} vdx.
	\end{align*}
	Plugging them back into \eqref{multiply}, we get 
	\begin{align*}
		\int_{\mathbb{R}^3} \lambda|\nabla v|^2dx&\leqslant\int_{\mathbb{R}^3} \lambda|\operatorname{div} v|^2dx
		+\int_{\mathbb{R}^3} \lambda|\operatorname{curl} v|^2dx\\
		&\ \ \ \ +\int_{\mathbb{R}^3}|\nabla \lambda||v||\nabla v|dx
		 +\int_{\mathbb{R}^3}|\nabla \lambda| |v||\operatorname{div} v|dx+\int_{\mathbb{R}^3}|\nabla \lambda||v||\operatorname{curl} v|dx.
	\end{align*}
Using Cauchy-Schwarz inequality on the right hand side further yields
\[\int_{\mathbb{R}^3} \lambda|\nabla v|^2dx\leqslant\frac{3}{2}\int_{\mathbb{R}^3} \lambda|\operatorname{div} v|^2dx
+\frac{3}{2}\int_{\mathbb{R}^3} \lambda|\operatorname{curl} v|^2dx
+\frac{3}{2}\int_{\mathbb{R}^3}\frac{|\nabla\lambda|^2}{\lambda}|v|^2dx
+\frac{1}{2}\int_{\mathbb{R}^3}\lambda|\nabla v|^2dx.\]
Because of the fact that 
$|\nabla\lambda|\lesssim\lambda$ gives $\frac{|\nabla\lambda|}{\sqrt{\lambda}}\lesssim\sqrt{\lambda}$, we arrive at the conclusion that 
	\begin{equation*}
		\int_{\mathbb{R}^3} \lambda|\nabla v|^2dx\lesssim\int_{\mathbb{R}^3} \lambda|\operatorname{div} v|^2dx
		+\int_{\mathbb{R}^3} \lambda|\operatorname{curl} v|^2dx
		+\int_{\mathbb{R}^3}\lambda|v|^2dx.
	\end{equation*}
This proves the lemma. 
\end{proof}

\smallskip
As a quick corollary of Lemma \ref{lemma:divcurl}, we obtain:

\begin{corollary}\label{coro:divcurl}
	Let $\lambda(x)\geqslant 1$ 
	be a smooth positive function on $\mathbb{R}^3$ with the additional property $|\nabla\lambda|\lesssim \lambda$. For any smooth vector field $v(x)\in H^m(\mathbb{R}^3)$ $(m\in \mathbb{Z}_{\geqslant 1})$, we have 
	\begin{equation}\label{eq:d-c2}
		\big\|\sqrt\lambda\nabla^k v\big\|_{L^2(\mathbb{R}^3)}^2 \lesssim\sum_{l=0}^{k-1}\big\|\sqrt\lambda\operatorname{div }\nabla^lv\big\|_{L^2(\mathbb{R}^3)}^2+\sum_{l=0}^{k-1} \big\|\sqrt{\lambda}\operatorname{curl }\nabla^lv\big\|_{L^2(\mathbb{R}^3)}^2+ \big\|\sqrt\lambda v\big\|_{L^2(\mathbb{R}^3)}^2,
	\end{equation}
	provided $\sqrt{\lambda}v\in L^2({\mathbb{R}^3})$ and $\sqrt{\lambda}\nabla^k v\in L^2({\mathbb{R}^3})$, where $1\leqslant k\leqslant m$.
\end{corollary}
\begin{proof}
	For $1\leqslant k\leqslant m$, it is easy to see that 
	$\nabla^{k-1}v\in H^{m-k+1}(\mathbb{R}^3)\subset  H^1(\mathbb{R}^3)$ enables us to invoke \eqref{eq:d-c1}:
	\[\big\|\sqrt\lambda\nabla^k v\big\|_{L^2(\mathbb{R}^3)}^2 \lesssim\big\|\sqrt\lambda\operatorname{div }\nabla^{k-1}v\big\|_{L^2(\mathbb{R}^3)}^2+ \big\|\sqrt\lambda\operatorname{curl }\nabla^{k-1}v\big\|_{L^2(\mathbb{R}^3)}^2+ \big\|\sqrt\lambda\nabla^{k-1} v\big\|_{L^2(\mathbb{R}^3)}^2.\]
	By induction on $k$, we can infer 
	\eqref{eq:d-c2} immediately. 
\end{proof}

\begin{remark}\label{rmk:divcurl}
	In particular, if we take $v=z_\pm$ (this is a divergence free vector field) and $m=N_*$ in Corollary \ref{coro:divcurl}, then it holds for $1\leqslant k\leqslant N_*$ that 
		\begin{equation}\label{eq:z to j inequality}
		\big\|\sqrt\lambda\nabla z_\pm^{(k-1)}\big\|_{L^2(\mathbb{R}^3)}^2 \lesssim \big\|\sqrt\lambda
		z_\pm\big\|_{L^2(\mathbb{R}^3)}^2+\sum_{l=0}^{k-1} \big\|\sqrt\lambda
		j_\pm^{( l)}\big\|_{L^2(\mathbb{R}^3)}^2.
	\end{equation}
\end{remark}

\medskip

Throughout the rest of this paper, the weight function $\lambda$ will be constructed from $\langle u_\pm \rangle=(R^2+|u_\pm|^2)^{\frac{1}{2}}$ and their combining forms.  
We collect some technical observations on weight functions as follows:
\begin{lemma}[Bound of weight]\label{lemma:weights}
	For  $R\geqslant 100$ and $\omega=1+\delta$, we have the following inequalities:
	\begin{enumerate}[(i)]
		\item Concerning the weights $\lambda=\langle u_{\pm}\rangle^{\omega}$, $\frac{\langle u_{\pm}\rangle^{\omega}}{\langle u_{\mp}\rangle^{\frac{\omega}{2}}}$ and $\langle u_{\mp}\rangle^{\omega}\langle u_{\pm}\rangle^{\frac{\omega}{2}}$, for $l=1,2$, there holds
		\begin{equation}\label{differentiate weights coro}
			\begin{split}
				&|\nabla^l \langle u_\pm\rangle^\omega|\lesssim \langle u_\pm\rangle^\omega, \\
				&\Big|\nabla^l\Big(\frac{\langle u_{\pm}\rangle^{\omega}}{\langle u_{\mp}\rangle^{\frac{\omega}{2}}}\Big)\Big|\lesssim \frac{\langle u_{\pm}\rangle^{\omega}}{\langle u_{\mp}\rangle^{\frac{\omega}{2}}},\\
				&\big|\nabla^l\big(\langle u_{\pm}\rangle^{\omega}\langle u_{\mp}\rangle^{\frac{\omega}{2}}\big)\big|\lesssim \langle u_{\pm}\rangle^{\omega}\langle u_{\mp}\rangle^{\frac{\omega}{2}}.
			\end{split}
		\end{equation}
		
		\item Concerning the weight $\lambda=\langle u_{\mp}\rangle^{\omega}\langle u_{\pm}\rangle^{\frac{\omega}{2}}$, there holds 
\begin{equation}\label{eq:xgeqleq}
(\langle u_{\mp}\rangle^{\omega}\langle u_{\pm}\rangle^{\frac{\omega}{2}})(\tau,x)\lesssim\begin{cases}
(\langle u_{\mp}\rangle^{\omega}\langle u_{\pm}\rangle^{\frac{\omega}{2}})(\tau,x')+|x-x'|^{\frac{3\omega}{2}},&\ \ \ \ \text{if}\ \  |x-x'|\geqslant 1\\
(\langle u_{\mp}\rangle^{\omega}\langle u_{\pm}\rangle^{\frac{\omega}{2}})(\tau,x'),&\ \ \ \ \text{if}\ \  |x-x'|\leqslant 2.
\end{cases}
\end{equation}
We remark here that both the inequalities in \eqref{eq:xgeqleq} hold when $1\leqslant |x-x'|\leqslant 2$.
	\end{enumerate}	
\end{lemma}
\begin{proof}
For (i), differentiation on $\langle u_\pm\rangle$  gives the bound $|\nabla^l\langle u_\pm\rangle|\lesssim 1$ due to  \eqref{Bootstrap on geometry},  and then directly implies \eqref{differentiate weights coro}. 
For (ii), we first apply the mean value theorem on $\langle u_\pm\rangle$ to get 
\begin{equation*}
	\langle u_\pm\rangle(\tau,x) \leqslant \langle u_\pm\rangle(\tau,x')+\big\|\nabla\langle u_\pm\rangle\big\|_{L^\infty(\mathbb{R}^3)}|x-x'|\lesssim \langle u_\pm\rangle(\tau,x')+|x-x'|.
\end{equation*}
On one hand, for $|x-x'|\geqslant 1$, \eqref{eq:xgeqleq} follows immediately. 
On the other hand, for $|x-x'|\leqslant 2$, since $R\geqslant 100$, this yields 
\begin{equation*}
	\langle u_\pm\rangle(\tau,x) \lesssim (R^2+|u_\pm|^2)^{\frac{1}{2}}(\tau,x')+1\leqslant (2R^2+|u_\pm|^2)^{\frac{1}{2}}(\tau,x')
	\lesssim\langle u_\pm\rangle(\tau,x'),
\end{equation*}
and hence \eqref{eq:xgeqleq} is an immediate consequence. This completes the proof.
\end{proof}

\medskip

By virtue of Lemma \ref{lemma:divcurl} and Lemma \ref{lemma:weights}, to derive the following point-wise estimates from the standard Sobolev inequality and the energy ansatz \eqref{Bootstrap on energy} is  more or less standard. We omit its proof here and refer readers to Lemma 2.5 in \cite{Li-Yu}.
\begin{lemma}[Weighted Sobolev inequality]\label{lemma:Sobolev}
	For all $k\leqslant N_*-2$ and multi-indices $\alpha$ with $\left|\alpha\right|=k$, we have 
	\begin{equation}\label{eq:Sobolev}\begin{split}
			&\langle u_{\pm}\rangle^\omega|z_{\mp}|\lesssim (E_{\mp}+E^0_{\mp}+E^1_{\mp})^{\frac{1}{2}}\lesssim C_1\varepsilon,\\
			&\langle u_{\pm}\rangle^\omega|\nabla z_{\mp}^{(\alpha)}|\lesssim (E^k_{\mp}+E^{k+1}_{\mp}+E^{k+2}_{\mp})^{\frac{1}{2}}\lesssim C_1\varepsilon. 
	\end{split}\end{equation}
\end{lemma}

\medskip

Before proceeding further, we elaborate a little more on elementary properties of weight functions, especially the separation property. Let $\psi_{\pm}(t,y)=(\psi^1_{\pm}(t,y), \psi^2_{\pm}(t,y), \psi^3_{\pm}(t,y))$ (mapping from $\Sigma_0$ to $\Sigma_t$) be the flow generated by $Z_{\pm}$, i.e. for all $t\in \mathbb{R}$ and $y\in \mathbb{R}^3$:
\begin{equation}\label{flow definition}
	\frac{d}{dt}\psi_{\pm}(t,y)=Z_{\pm}(t,\psi_{\pm}(t,y)), \ \ \psi_{\pm}(0,y)=y.
\end{equation}
From now on, we consider $y$ as the initial label and $x$ as the present label when using the flow map. Since the fluctuations are given by $z_\pm = Z_\pm \mp B_0$, then integrating \eqref{flow definition} gives rise to
\begin{equation}\label{flow in integration}
	\psi_{\pm}(t,y)=y+\int_0^tZ_{\pm}(\tau,\psi_{\pm}(\tau,y))d\tau =y \pm t B_0+\int_0^tz_{\pm}(\tau,\psi_{\pm}(\tau,y))d\tau.
\end{equation}
\begin{lemma}[Separation property of weight]\label{lemma:separation weight}
	For the product of $\langle u_{+}\rangle$ and $\langle u_{-}\rangle$, there holds
	\begin{equation}\label{eq:separation weight}
		\langle u_{+}\rangle\langle u_{-}\rangle\gtrsim R+t.
	\end{equation} 
\end{lemma}
\begin{proof}
	Thanks to \eqref{flow in integration}, the $x_3$-coordinate component of the flow $\psi_\pm$ is given by 
	\begin{equation*}
		\psi^3_{\pm}(t,y)=y_3 \pm t +\int_0^tz^3_{\pm}(\tau,\psi_{\pm}(\tau,y))d\tau.
	\end{equation*}
	Since $u_\pm=x_3^\pm$ and $x^\pm(t,\psi_\pm(t,y))=y$, we obtain 
	\begin{equation*}
		u_\pm(t,\psi_\pm(t,y))=x_3^\pm(t,\psi_\pm(t,y))=y_3=\psi_{\pm}^3(t,y)\mp t-\int_0^tz_{\pm}^3(\tau,\psi_{\pm}(\tau,y))d\tau.
	\end{equation*}
	As a result, we infer that 
	\begin{equation*}
		u_\pm(t,x)=x_3\mp t-\int_0^tz_{\pm}^3(\tau,\psi_{\pm}(\tau,\psi^{-1}_{\pm}(t,x)))d\tau.
	\end{equation*}
	Together with the amplitude ansatz \eqref{Bootstrap on fluctuation}, this leads to
	\begin{equation*}
		\big|(u_--u_+)-2t\big|\leqslant\int_0^t\big(\|z_+^3\|_{L^\infty}+\|z_-^3\|_{L^\infty}\big)d\tau\leqslant t,
	\end{equation*}
	which then yields the estimates $	t\leqslant |u_+-u_-|\leqslant 3t$,  $|u_+|+|u_-|\geqslant t$, and at least one of the inequalities $|u_+|\geqslant\frac{t}{2}$ and  $|u_-|\geqslant\frac{t}{2}$ hold. Thus we arrive at the estimate that 
	\begin{equation*}
		\langle u_{+}\rangle\langle u_{-}\rangle=(R^2+|u_+|^2)^{\frac{1}{2}}(R^2+|u_-|^2)^{\frac{1}{2}}
		\geqslant R(R^2+\tfrac{t^2}{4})^{\frac{1}{2}}\geqslant \tfrac{R}{2}(R^2+t^2)^{\frac{1}{2}},
	\end{equation*}
	which gives \eqref{eq:separation weight} immediately.
\end{proof}

\bigskip

As an application of Lemma \ref{lemma:Sobolev} and Lemma \ref{lemma:separation weight}, we can measure the separation of fluctuations $z_\pm$ in terms of decay in time. The proof is straightforward and so is omitted. 
\begin{lemma}[Separation estimate of fluctuation]\label{lemma:separation estimate}
For all $\alpha$ and $\beta$ with $|\alpha|,|\beta|\leqslant N_*-1$, we have 
\begin{equation}
\big|z_+^{(\alpha)}(t,x)z_-^{(\beta)}(t,x)\big|\lesssim\frac{\big(C_1\big)^2\varepsilon^2}{\langle u_{+}\rangle^\omega\langle u_{-}\rangle^\omega}\lesssim\frac{\big(C_1\big)^2\varepsilon^2}{(R+t)^\omega}.
\end{equation}
\end{lemma}

\bigskip

Observe that the flux is a robust tool to investigate the decay of waves.
We now use it to bound a set of spacetime integrals about $z_\pm$ and their derivatives, which will be useful for energy estimates.
\begin{lemma}[Spacetime estimate by using flux]\label{lemma:flux} For all $0\leqslant k\leqslant N_*$, we have
	\begin{equation}\label{eq:flux}
		\begin{split}
			&\int_{0}^{t^*}\int_{\Sigma_{\tau}}\frac{\langle u_{\mp}\rangle^{2\omega}}{\langle u_{\pm}\rangle^{\omega}}|z_{\pm}|^2\lesssim \big(C_1\big)^2\varepsilon^2,\\
			&\int_{0}^{t^*}\int_{\Sigma_{\tau}}\frac{\langle u_{\mp}\rangle^{2\omega}}{\langle u_{\pm}\rangle^{\omega}}\big|j^{(k)}_{\pm}\big|^2\lesssim \big(C_1\big)^2 \varepsilon^2,\\
			&\int_{0}^{t^*}\int_{\Sigma_{\tau}}\frac{\langle u_{\mp}\rangle^{2\omega}}{\langle u_{\pm}\rangle^{\omega}}\big|\nabla z^{(k)}_{\pm}\big|^2 \lesssim \big(C_1\big)^2\varepsilon^2.
		\end{split}
	\end{equation} 
\end{lemma}
\begin{proof}
	Denote $x_h=(x_1,x_2)$. If we parameterize $C_{u_\pm}^{\pm,t}$ by $(x_h,t)$, the surface measure $d\sigma_\pm$ can be written as 
	\[d\sigma_\pm=\sqrt{1+|\nabla_{t,x_h}u_\pm|^2}dx_hdt=\sqrt{1+|\partial_{t}u_\pm|^2+|\nabla_{x_h}u_\pm|^2}dx_hdt=\sqrt{1+|Z_\pm\cdot\nabla u_\pm|^2+|\nabla_{x_h}u_\pm|^2}dx_hdt,\]
	where we have used $\partial_tu_\pm + Z_\pm\cdot\nabla u_\pm=L_\pm u_\pm=0$. By \eqref{Bootstrap on fluctuation} and \eqref{Bootstrap on geometry}, for sufficiently small $\varepsilon$, there exists some constant $C$ such that  $|z_\pm| \leqslant C\varepsilon$ and $|\nabla u_\pm- e_3| \leqslant C\varepsilon$. It follows that  $(1-C\varepsilon)^2\leqslant|Z_\pm\cdot\nabla u_\pm|\leqslant (1+C\varepsilon)^2$ and $|\nabla_{x_h} u_\pm| \leqslant |\nabla u_\pm- e_3| \leqslant C\varepsilon$. Consequently, there exists some constant $C$ (may be different from the constant above) such that
	$$\sqrt{2}-C\varepsilon\leqslant\sqrt{1+|Z_\pm\cdot\nabla u_\pm|^2+|\nabla_{x_h}u_\pm|^2}\leqslant\sqrt{2}+C\varepsilon.$$
	Thus we can summarize that 
	\begin{equation}\label{surface measure}
		d\sigma_\pm=(\sqrt{2}+O(\varepsilon))d{x_h}dt.
	\end{equation}
	
	Using \eqref{Bootstrap on energy},  \eqref{surface measure} and the fact that $\displaystyle \int_{\mathbb{R}}\frac{1}{\langle u_{\pm}\rangle^\omega}du_{\pm}$ can be bounded by a universal constant, we are now ready to derive estimates for the spacetime integrals in \eqref{eq:flux}:
	\begin{align*}
		\int_{0}^{t^*}\int_{\Sigma_{\tau}}\frac{\langle u_{\mp}\rangle^{2\omega}}{\langle u_{\pm}\rangle^{\omega}}|z_{\pm}|^2dxd\tau
		&\lesssim\int_{u_{\pm}}\Big(\int_{C^{\pm,t^*}_{u_{\pm}}}\frac{\langle u_{\mp}\rangle^{2\omega}}{\langle u_{\pm}\rangle^{\omega}}|z_{\pm}|^2 d\sigma_{\pm}\Big)du_{\pm}\\
		&\lesssim \sup_{u_{\pm}\in \mathbb{R}}\Big[\int_{C^{\pm,t^*}_{u_{\pm}}}\langle u_{\mp}\rangle^{2\omega}|z_{\pm}|^2d\sigma_{\pm}\Big]\int_{\mathbb{R}}\frac{1}{\langle u_{\pm}\rangle^\omega}du_{\pm}\\
		&\lesssim F_{\pm}\lesssim \big(C_1\big)^2\varepsilon^2.
	\end{align*}
	The bound on $\displaystyle \int_{0}^{t^*}\int_{\Sigma_{\tau}}\frac{\langle u_{\mp}\rangle^{2\omega}}{\langle u_{\pm}\rangle^{\omega}}\big|j^{(k)}_{\pm}\big|^2dxd\tau$ can be obtained exactly in the same way. Due to \eqref{differentiate weights coro}, we know that the weight functions used above satisfy the conditions of the div-curl lemma (Lemma \ref{lemma:divcurl}, Corollary \ref{coro:divcurl} and especially Remark \ref{rmk:divcurl}) and hence we can invoke Remark \ref{rmk:divcurl}. Together with \eqref{eq:z to j inequality}, the previous two estimates yield
	\begin{align*}
		\int_{0}^{t^*}\int_{\Sigma_{\tau}}\frac{\langle u_{\mp}\rangle^{2\omega}}{\langle u_{\pm}\rangle^{\omega}}\big|\nabla z^{(k)}_{\pm}\big|^2dxd\tau
		&\lesssim
		\int_{0}^{t^*}\bigg(\Big\|\frac{\langle u_{\mp}\rangle^{\omega}}{\langle u_{\pm}\rangle^{\frac{\omega}{2}}}z_{\pm}\Big\|_{L^2(\mathbb{R}^3)}^2
		+\sum_{l=0}^k\Big\|\frac{\langle u_{\mp}\rangle^{\omega}}{\langle u_{\pm}\rangle^{\frac{\omega}{2}}}j^{(l)}_{\pm}\Big\|_{L^2(\mathbb{R}^3)}^2\bigg)d\tau\\
		&\lesssim \big(C_1\big)^2\varepsilon^2.
	\end{align*}
	We have thus proved the lemma.
\end{proof}

\smallskip

Let us end this subsection by recalling a classical energy estimate for the following linear system:
\begin{equation}\label{eq:linear eq}\begin{cases}
		& \partial_{t}f_{+}+Z_-\cdot \nabla f_{+}  =\rho_{+},\\
		& \partial_{t}f_{-}+Z_+\cdot \nabla f_{-}  =\rho_{-},
\end{cases}\end{equation}
where $f_\pm$ and $\rho_\pm$ are smooth vector fields defined on $\mathbb{R}\times \mathbb{R}^3$ with sufficiently fast decay in $x$-variables. We also refer the interested readers to  \cite{He-Xu-Yu,Li-Yu} for its proof.
\smallskip
\begin{lemma}[Linear energy estimate]\label{lemma:linear EE}
	For all weight functions $\lambda_{\pm}$ defined on $[0,t^*]\times\mathbb{R}^3$
	with the properties $L_{\pm}\lambda_{\mp}=0$, where $L_{\pm}=\partial_t+ Z_\pm\cdot\nabla$, for all $t\in [0,t^*]$, we have 
	\begin{equation}\label{eq:linear EE}
	\sum_{+,-}	\int_{\Sigma_{t}}\lambda_{\pm}|f_{\pm}|^2 dx +\sum_{+,-}\sup_{u_{\pm}\in\mathbb{R}}\int_{C_{u_{\pm}}^{\pm,t}}\lambda_{\pm}|f_{\pm}|^2 d\sigma_{\pm}\leqslant \sum_{+,-}\int_{\Sigma_{0}}\lambda_{\pm}|f_{\pm}|^2 dx +2\sum_{+,-}\int_{0}^{t}\int_{\Sigma_{\tau}}\lambda_{\pm}\big|f_{\pm}\big|\big|\rho_{\pm}\big|dxd\tau.
	\end{equation}
In particular, except for the coefficients of the first terms on the both sides of \eqref{eq:linear EE}, the exactly numerical constants are irrelevant to this paper.
\end{lemma}

\smallskip

\subsection{Estimates on the pressure term}
This subsection is where the pressure estimates are set-up. 

\smallskip

Due to $\operatorname{div} z_{\pm}=0$,  taking divergence of the first equation in \eqref{eq:MHD} gives 
\begin{equation*}
	-\Delta p=\partial_{i}z^{j}_{-}\partial_{j}z^{i}_{+}.
\end{equation*}
Using the Newtonian potential, we obtain the following expression for the pressure:
\begin{equation}\label{eq:pressure}
	p(\tau,x)=
	\frac{1}{4\pi} \int_{\mathbb{R}^3}\frac{1}{|x-x'|}\cdot\big(\partial_{i}z^{j}_{-}\partial_{j}z^{i}_{+}\big)(\tau,x')dx'.
\end{equation}

\smallskip

Before proceeding further, we first collect several some facts concerning integration of $\frac{1}{|x|^\gamma}$:
\begin{equation}\label{observe integration}\begin{split}
&\frac{1}{|x|^\gamma}\text{\Large$\text{\Large$\chi$}$}_{|x|\leqslant 1}\in L^1(\mathbb{R}^3)\iff \gamma<3,\\
&\frac{1}{|x|^\gamma}\text{\Large$\text{\Large$\chi$}$}_{|x|\geqslant 1}\in L^1(\mathbb{R}^3)\iff \gamma>3,\\
&\frac{1}{|x|^\gamma}\text{\Large$\text{\Large$\chi$}$}_{|x|\leqslant 1}\in L^2(\mathbb{R}^3)\iff \gamma<\frac{3}{2},\\
&\frac{1}{|x|^\gamma}\text{\Large$\text{\Large$\chi$}$}_{|x|\geqslant 1}\in L^2(\mathbb{R}^3)\iff \gamma>\frac{3}{2},
\end{split}
\end{equation}
where $\text{\Large$\text{\Large$\chi$}$}_S$ represents the characteristic function of the set $S$. 
Indeed, the facts in \eqref{observe integration} will help us avoid the non-integrable singularity $\frac{1}{|x-x'|}$ in pressure estimates and energy estimates.

\smallskip

We are now in a position to work on \eqref{eq:pressure} and derive bounds on pressure. 

\begin{lemma}\label{lemma:pressure}
Let $\theta(r)$ be a smooth cut-off function  so that $\theta(r)=1$ for $r\leqslant 1$ and $\theta(r)=0$ for $r\geqslant 2$.	
For $l=1,2,3$, for all $(\tau,x)\in \mathbb{R}\times \mathbb{R}^3$, there holds
	\begin{equation}\label{eq:pressure derivative}
		\nabla^l p(\tau,x)=\mathbf{A}_l+\mathbf{B}_l,
	\end{equation}
	where
	\begin{align*}
		&\mathbf{A}_l:=\frac{1}{4\pi}\int_{\mathbb{R}^3}\nabla_x\frac{1}{|x-x'|}\cdot\nabla_{x'}^{l-1}\Big(\theta(|x-x'|)\cdot\big(\partial_{i}z^{j}_{-}\partial_{j}z^{i}_{+}\big)(\tau,x')\Big)dx',\\
		&\mathbf{B}_l:=(-1)^{l-1}\frac{1}{4\pi}\int_{\mathbb{R}^3}\partial_{i}^{1-\lfloor\frac{l}{2}\rfloor}\partial_{j}^{1-\lfloor\frac{l-1}{2}\rfloor}\Big(\nabla^l_x\frac{1}{|x-x'|}\cdot \big(1-\theta(|x-x'|)\big)\Big)\cdot\big(\partial_{i}^{\lfloor\frac{l}{2}\rfloor}z^{j}_{-}\partial_{j}^{\lfloor\frac{l-1}{2}\rfloor}z^{i}_{+}\big)(\tau,x')dx'.
	\end{align*}
\end{lemma}
\begin{proof}
	According to \eqref{eq:pressure}, there holds the following decomposition on each time slice $\Sigma_{\tau}$:
	\begin{align*}
		\nabla^l p(\tau,x)	
		&=
		\frac{1}{4\pi}\int_{\mathbb{R}^3}\nabla^l_x\frac{1}{|x-x'|}\cdot\big(\partial_{i}z^{j}_{-}\partial_{j}z^{i}_{+}\big)(\tau,x')dx'\\
		&=\underbrace{\frac{1}{4\pi}\int_{\mathbb{R}^3}\nabla^l_x\frac{1}{|x-x'|}\cdot \theta(|x-x'|)\cdot \big(\partial_{i}z^{j}_{-}\partial_{j}z^{i}_{+}\big)(\tau,x')dx'}_{\mathbf{A}_l'}\\
		&\ \ \ \   +\underbrace{\frac{1}{4\pi}\int_{\mathbb{R}^3}\nabla^l_x\frac{1}{|x-x'|}\cdot \big(1-\theta(|x-x'|)\big)\cdot\big(\partial_{i}z^{j}_{-}\partial_{j}z^{i}_{+}\big)(\tau,x')dx'}_{\mathbf{B}_l'}.
	\end{align*}

	For $\mathbf{A}_l'$, based on the observation that
	\[\nabla_x\frac{1}{|x-x'|}=-\nabla_{x'}\frac{1}{|x-x'|},\]
	we acquire
	\begin{align*}
		\mathbf{A}_l'&=(-1)^{l-1}\frac{1}{4\pi}\int_{\mathbb{R}^3}\nabla^{l-1}_{x'}\nabla_x\frac{1}{|x-x'|}\cdot \theta(|x-x'|)\cdot \big(\partial_{i}z^{j}_{-}\partial_{j}z^{i}_{+}\big)(\tau,x')dx'\\
		&=\frac{1}{4\pi}\int_{\mathbb{R}^3}\nabla_x\frac{1}{|x-x'|}\cdot\nabla^{l-1}_{x'}\Big( \theta(|x-x'|)\cdot \big(\partial_{i}z^{j}_{-}\partial_{j}z^{i}_{+}\big)(\tau,x')\Big)dx'\\
		&=\mathbf{A}_l.
	\end{align*}
	
	For $\mathbf{B}_l'$, in order to avoid the non-integrable singularity $\frac{1}{|x-x'|}$ in the energy estimates, we invoke the facts in \eqref{observe integration} and use integration by parts to infer that 
	\begin{align*}
		\mathbf{B}_1'&=\frac{1}{4\pi}\int_{\mathbb{R}^3}\partial_{j}\partial_{i}\Big(\nabla_x\frac{1}{|x-x'|} \cdot\big(1-\theta(|x-x'|)\big)\Big)\cdot\big(z^{j}_{-}z^{i}_{+}\big)(\tau,x')dx',\\
		\mathbf{B}_2'&=
		-\frac{1}{4\pi}\int_{\mathbb{R}^3}\partial_{j}\Big(\nabla^2_x\frac{1}{|x-x'|}\cdot \big(1-\theta(|x-x'|)\big)\Big)\cdot\big(\partial_{i}z^{j}_{-}z^{i}_{+}\big)(\tau,x')dx',\\
		\mathbf{B}_3'&=\frac{1}{4\pi}\int_{\mathbb{R}^3}\nabla^3_x\frac{1}{|x-x'|}\cdot \big(1-\theta(|x-x'|)\big)\cdot\big(\partial_{i}z^{j}_{-}\partial_{j}z^{i}_{+}\big)(\tau,x')dx'.
	\end{align*}
Arranging these three cases in a unified form, we conclude that 
	\begin{align*}
		\mathbf{B}_l'&=(-1)^{l-1}\frac{1}{4\pi}\int_{\mathbb{R}^3}\partial_{i}^{1-\lfloor\frac{l}{2}\rfloor}\partial_{j}^{1-\lfloor\frac{l-1}{2}\rfloor}\Big(\nabla^l_x\frac{1}{|x-x'|}\cdot \big(1-\theta(|x-x'|)\big)\Big)\cdot\big(\partial_{i}^{\lfloor\frac{l}{2}\rfloor}z^{j}_{-}\partial_{j}^{\lfloor\frac{l-1}{2}\rfloor}z^{i}_{+}\big)(\tau,x')dx'\\
		&=\mathbf{B}_l.
	\end{align*}

Hence the decomposition \eqref{eq:pressure derivative} is proved. 
\end{proof}

\smallskip

Lemma \ref{lemma:pressure} enables us to derive the following two lemmas for the pressure up to three order derivatives, which concern their point-wise bounds and space-time estimates respectively.

\begin{lemma}\label{lemma:pressure derivative bound}
	For $l=1,2,3$, for all $(\tau,x)\in \mathbb{R}\times \mathbb{R}^3$, we have the following point-wise bounds on pressure:
	\begin{equation}\label{eq:pressure derivative bound}
		|\nabla^l p(\tau,x)|\lesssim\frac{\big(C_1\big)^2\varepsilon^2}{(R+\tau)^\omega}.
	\end{equation}
\end{lemma}
\begin{proof}
Thanks to \eqref{eq:pressure derivative}, we have
	\begin{equation}\label{eq:nablalp}
		|\nabla^l p(\tau,x)|\leqslant|\mathbf{A}_l|+|\mathbf{B}_l|,
	\end{equation}
	where 
	\begin{align*}
		&	|\mathbf{A}_l|
		\lesssim  \sum^l_{l_1,l_2=1}     \int_{|x-x'|\leqslant 2}\frac{|(\nabla^{l_1} z_{-}\cdot\nabla^{l_2} z_{+})(\tau,x')|}{|x-x'|^2}dx',\stepcounter{equation}\tag{\theequation}\label{eq:Al}\\
		&	|\mathbf{B}_l|\lesssim\int_{|x-x'|\geqslant 1}\frac{|(\nabla^{\lfloor\frac{l}{2}\rfloor}z_{-}\cdot\nabla^{\lfloor\frac{l-1}{2}\rfloor}  z_{+})(\tau,x')|}{|x-x'|^4}dx'.\stepcounter{equation}\tag{\theequation}\label{eq:Bl}
	\end{align*}

Together with \eqref{eq:separation weight}, \eqref{eq:Sobolev} and \eqref{observe integration}, the estimates \eqref{eq:Al} and \eqref{eq:Bl} lead us to 
	\begin{align*}
		\big(R+\tau\big)^{{\omega}}|\mathbf{A}_l|
		&\stackrel{\eqref{eq:separation weight}}{\lesssim}\sum^l_{l_1,l_2=1} 
		\int_{|x-x'|\leqslant 2}\frac{\big(\langle u_+\rangle^\omega\langle u_-\rangle^{ \omega}\big)(\tau,x')|(\nabla^{l_1} z_{-}\cdot\nabla^{l_2} z_{+})(\tau,x')|}{|x-x'|^2}dx'\\
		&\lesssim\sum^l_{l_1,l_2=1} \big\|\langle u_+\rangle^\omega\nabla^{l_1}z_-\big\|_{L^\infty(\mathbb{R}^3)}\big\|\langle u_-\rangle^\omega\nabla^{l_2}z_+\big\|_{L^\infty(\mathbb{R}^3)}
		\int_{|x-x'|\leqslant 2}\frac{1}{|x-x'|^2}dx'\\
		&\stackrel{\eqref{eq:Sobolev}}{\lesssim}\big(C_1\big)^2\varepsilon^2	\int_{|x-x'|\leqslant 2}\frac{1}{|x-x'|^2}dx'\\
		&\stackrel{\eqref{observe integration}}{\lesssim} \big(C_1\big)^2\varepsilon^2,\\
		\big(R+\tau\big)^{{\omega}}|\mathbf{B}_l|
		&\stackrel{\eqref{eq:separation weight}}{\lesssim}\int_{|x-x'|\geqslant 1}\frac{\big(\langle u_+\rangle^\omega\langle u_-\rangle^{ \omega}\big)(\tau,x')|(\nabla^{\lfloor\frac{l}{2}\rfloor}z_{-}\cdot\nabla^{\lfloor\frac{l-1}{2}\rfloor}  z_{+})(\tau,x')|}{|x-x'|^4}dx'\\
		&\lesssim \big\|\langle u_+\rangle^\omega\nabla^{\lfloor\frac{l}{2}\rfloor}z_-\big\|_{L^\infty(\mathbb{R}^3)}\big\|\langle u_-\rangle^\omega\nabla^{\lfloor\frac{l-1}{2}\rfloor}z_+\big\|_{L^\infty(\mathbb{R}^3)}
		\int_{|x-x'|\geqslant 1}\frac{1}{|x-x'|^4}dx'\\
		&\stackrel{\eqref{eq:Sobolev}}{\lesssim}\big(C_1\big)^2\varepsilon^2\int_{|x-x'|\geqslant 1}\frac{1}{|x-x'|^4}dx'\\
		&\stackrel{\eqref{observe integration}}{\lesssim} \big(C_1\big)^2\varepsilon^2.
	\end{align*}
	
	Putting these two estimates together into \eqref{eq:nablalp}, we thus obtain the result \eqref{eq:pressure derivative bound}.
\end{proof}

\smallskip

\begin{lemma}\label{lemma:pressureL2L2} For $l=1,2,3$, we have the following space-time estimates on pressure:
	\begin{equation}\label{eq:pressureL2L2}
		\int_{\mathbb{R}\times \mathbb{R}^3} \langle u_{\mp}\rangle^{2\omega}\langle u_{\pm}\rangle^\omega |\nabla^l p |^2\lesssim\big(C_1\big)^4\varepsilon^4.
	\end{equation}
\end{lemma}
\begin{proof}
	It follows from the estimate \eqref{eq:nablalp} that 
	\begin{equation}\label{eq:I}
		\int_{\mathbb{R}\times \mathbb{R}^3}\langle u_{\mp}\rangle^{2\omega}\langle u_{\pm}\rangle^\omega |\nabla^l p|^2\lesssim 
		\underbrace{\int_{0}^{t}\int_{\Sigma_{\tau}}\langle u_{\mp}\rangle^{2\omega}\langle u_{\pm}\rangle^\omega |\mathbf{A}_l|^2}_{\mathbf{I}_{\mathbf{A}_l}}
		+\underbrace{\int_{0}^{t}\int_{\Sigma_{\tau}}\langle u_{\mp}\rangle^{2\omega}\langle u_{\pm}\rangle^\omega |\mathbf{B}_l|^2}_{\mathbf{I}_{\mathbf{B}_l}}.
	\end{equation}
	
	For $\mathbf{I}_{\mathbf{A}_l}$, according to \eqref{eq:Al} and Young's inequality, we obtain 
	\begin{align*}
		\mathbf{I}_{\mathbf{A}_l}
		&\lesssim\sum^l_{l_1,l_2=1} 
		\int_{0}^{t}\Big\|\langle u_{\mp}\rangle^{\omega}(\tau,x)\langle u_{\pm}\rangle^{\frac{\omega}{2}}(\tau,x)\int_{|x-x'|\leqslant 2}\frac{|(\nabla^{l_1} z_{\mp}\cdot\nabla^{l_2} z_{\pm})(\tau,x')|}{|x-x'|^2}dx'\Big\|_{L^2(\mathbb{R}^3)}^2\\
		&\stackrel{\eqref{eq:xgeqleq}}{\lesssim}  \sum^l_{l_1,l_2=1} \int_{0}^{t}\Big\|\int_{|x-x'|\leqslant 2} \frac{\big(\langle u_{\mp}\rangle^{\omega}\langle u_{\pm}\rangle^{\frac{\omega}{2}}\big)(\tau,x')|(\nabla^{l_1} z_{\mp}\cdot\nabla^{l_2} z_{\pm})(\tau,x')|}{|x-x'|^2}dx'\Big\|_{L^2(\mathbb{R}^3)}^2\\
		&\stackrel{\text{Young's}}{\lesssim}\sum^l_{l_1,l_2=1} 
		\int_{0}^{t}\Big\|\frac{1}{|x|^2}\text{\Large$\chi$}_{|x|\leqslant 2}\Big\|_{L^1(\mathbb{R}^3)}^2 \big\|\langle u_{\pm}\rangle^{\omega}\nabla^{l_1}z_{\mp}\big\|_{L^\infty(\mathbb{R}^3)}^2\Big\|\frac{\langle u_{\mp}\rangle^{\omega}}{\langle u_{\pm}\rangle^{\frac{\omega}{2}}}\nabla^{l_2}
		z_{\pm}\Big\|_{L^2(\mathbb{R}^3)}^2 \\
		&\stackrel{\eqref{eq:Sobolev},\eqref{observe integration}}{\lesssim}\sum^l_{l_2=1}\big(C_1\big)^2\varepsilon^2\Big\|\frac{\langle u_{\mp}\rangle^{\omega}}{\langle u_{\pm}\rangle^{\frac{\omega}{2}}}
		\nabla^{l_2} z_{\pm}\Big\|^2_{L^2(\mathbb{R}^3)}\\
		&\stackrel{\eqref{eq:flux}}{\lesssim} \big(C_1\big)^4\varepsilon^4.\stepcounter{equation}\tag{\theequation}\label{eq:IAl}
	\end{align*}

	For $\mathbf{I}_{\mathbf{B}_l}$, by virtue of \eqref{eq:Bl}, we can derive 
	\begin{align*}
		\mathbf{I}_{\mathbf{B}_l}
		&\lesssim
		\int_{0}^{t}\Big\|\langle u_{\mp}\rangle^{\omega}(\tau,x)\langle u_{\pm}\rangle^{\frac{\omega}{2}}(\tau,x)\int_{|x-x'|\geqslant 1}\frac{|(\nabla^{\lfloor\frac{l}{2}\rfloor}z_{\mp}\cdot\nabla^{\lfloor\frac{l-1}{2}\rfloor}  z_{\pm})(\tau,x')|}{|x-x'|^4}dx'\Big\|_{L^2(\mathbb{R}^3)}^2\\
		&\stackrel{\eqref{eq:xgeqleq}}{\lesssim} \int_{0}^{t}\Big\|\int_{|x-x'|\geqslant 1}\Big(\big(\langle u_{\mp}\rangle^{\omega}\langle u_{\pm}\rangle^{\frac{\omega}{2}}\big)(\tau,x')+|x-x'|^{\frac{3\omega}{2}}\Big)\frac{|(\nabla^{\lfloor\frac{l}{2}\rfloor}z_{\mp}\cdot\nabla^{\lfloor\frac{l-1}{2}\rfloor}  z_{\pm})(\tau,x')|}{|x-x'|^4}dx'\Big\|_{L^2(\mathbb{R}^3)}^2\\
		&\lesssim
		\underbrace{\int_{0}^{t}\Big\|\int_{|x-x'|\geqslant 1}\big(\langle u_{\mp}\rangle^{\omega}\langle u_{\pm}\rangle^{\frac{\omega}{2}}\big)(\tau,x')\frac{|(\nabla^{\lfloor\frac{l}{2}\rfloor}z_{\mp}\cdot\nabla^{\lfloor\frac{l-1}{2}\rfloor}  z_{\pm})(\tau,x')|}{|x-x'|^4}dx'\Big\|_{L^2(\mathbb{R}^3)}^2}_{\mathbf{I}_{\mathbf{B}_l}'}\\
		&\ \ \ \ +\underbrace{\int_{0}^{t}\Big\|\int_{|x-x'|\geqslant 1}\frac{|(\nabla^{\lfloor\frac{l}{2}\rfloor}z_{\mp}\cdot\nabla^{\lfloor\frac{l-1}{2}\rfloor}  z_{\pm})(\tau,x')|}{|x-x'|^{4-\frac{3\omega}{2}}}dx'\Big\|_{L^2(\mathbb{R}^3)}^2}_{\mathbf{I}_{\mathbf{B}_l}''}.
	\end{align*}
In exactly the same manner as in \eqref{eq:IAl},  $\mathbf{I}_{\mathbf{B}_l}'$ can be bounded as follows:
	\begin{align*}
		\mathbf{I}_{\mathbf{B}_l}'&
		\stackrel{\text{Young's}}{\lesssim}
		\int_{0}^{t}\Big\|\frac{1}{|x|^4}\text{\Large$\chi$}_{|x|\geqslant 1}\Big\|_{L^1(\mathbb{R}^3)}^2 \big\|\langle u_{\pm}\rangle^{\omega}\nabla^{\lfloor\frac{l}{2}\rfloor}z_{\mp}\big\|_{L^\infty(\mathbb{R}^3)}^2\Big\|\frac{\langle u_{\mp}\rangle^{\omega}}{\langle u_{\pm}\rangle^{\frac{\omega}{2}}}\nabla^{\lfloor\frac{l-1}{2}\rfloor}
		z_{\pm}\Big\|_{L^2(\mathbb{R}^3)}^2 \\
		&\stackrel{\eqref{eq:Sobolev},\eqref{observe integration}}{\lesssim}\big(C_1\big)^2\varepsilon^2\int_{0}^{t}\Big\|\frac{\langle u_{\mp}\rangle^{\omega}}{\langle u_{\pm}\rangle^{\frac{\omega}{2}}}\nabla^{\lfloor\frac{l-1}{2}\rfloor}z_{\pm}\Big\|_{L^2(\mathbb{R}^3)}^2\\
		&\stackrel{\eqref{eq:flux}}{\lesssim} \big(C_1\big)^4\varepsilon^4.
	\end{align*}
	It remains to bound $\mathbf{I}_{\mathbf{B}_l}''$. Notice that \eqref{observe integration} and $\omega=1+\delta$ give us  $\dfrac{1}{|x|^{4-\frac{3\omega}{2}}}\text{\Large$\chi$}_{|x|\geqslant 1}\in L^2(\mathbb{R}^3)$ when $\omega\in(1,\frac{5}{3})$ (this is the place where we need constraints $\delta \in(0,\frac{2}{3})$). We can also infer 
	\begin{align*}
		\mathbf{I}_{\mathbf{B}_l}''&
		\stackrel{\text{Young's}}{\lesssim}
		\int_{0}^{t}\Big\|\frac{1}{|x|^{4-\frac{3\omega}{2}}}\text{\Large$\chi$}_{|x|\geqslant 1}\Big\|_{L^2(\mathbb{R}^3)}^2\big\|\nabla^{\lfloor\frac{l}{2}\rfloor}z_{-}\cdot\nabla^{\lfloor\frac{l-1}{2}\rfloor}  z_{+}\big\|_{L^1(\mathbb{R}^3)}^2\\
		&\stackrel{\eqref{observe integration}}{\lesssim} \int_{0}^{t}\big\|\nabla^{\lfloor\frac{l}{2}\rfloor}z_{-}\cdot\nabla^{\lfloor\frac{l-1}{2}\rfloor}  z_{+}\big\|_{L^1(\mathbb{R}^3)}^2\\
		&\stackrel{\eqref{eq:separation weight}}{\lesssim}\int_{0}^{t}\frac{1}{(R+\tau)^{2\omega}}\Big\|\langle u_{+}\rangle^{\omega}\nabla^{\lfloor\frac{l}{2}\rfloor} z_{-}\langle u_{-}\rangle^{\omega}\nabla^{\lfloor\frac{l-1}{2}\rfloor}  z_{+}\Big\|_{L^1(\mathbb{R}^3)}^2\\
		&\stackrel{\text{H\"older}}{\lesssim} \int_{0}^{t}\frac{1}{(R+\tau)^{2\omega}}\big\|\langle u_{+}\rangle^{\omega}\nabla^{\lfloor\frac{l}{2}\rfloor} z_{-}\big\|_{L^2(\mathbb{R}^3)}^2\big\|\langle u_{-}\rangle^{\omega}\nabla^{\lfloor\frac{l-1}{2}\rfloor} z_{+}\big\|_{L^2(\mathbb{R}^3)}^2\\
		&\stackrel{\eqref{Bootstrap on energy}}{\lesssim} \big(C_1\big)^2\varepsilon^2\big(C_1\big)^2\varepsilon^2\int_{0}^{t}\frac{1}{(R+\tau)^{2\omega}}d\tau\lesssim \big(C_1\big)^4\varepsilon^4. 
	\end{align*}
Thus we conclude that 
 \begin{equation}\label{eq:IBl}
 \mathbf{I}_{\mathbf{B}_l}\lesssim \big(C_1\big)^4\varepsilon^4. 
 \end{equation}
 
Finally we obtain \eqref{eq:pressureL2L2} by
substituting \eqref{eq:IAl} and \eqref{eq:IBl} into \eqref{eq:I}. This is what we wanted to prove.
\end{proof}

\smallskip

As a direct consequence of Lemma \ref{lemma:pressureL2L2}, we have the following $L_\tau^2L_x^\infty$ estimate for $\nabla p$: 
\begin{lemma}\label{lemma:pressureL2Linfty} There holds 
	\begin{equation}\label{eq:pressureL2Linfty}
		\int_{\mathbb{R}}\big\|\langle u_\mp\rangle^{\omega}\langle u_\pm\rangle^{\frac{\omega}{2}}\nabla p\big\|^2_{L^\infty_{x}}d\tau\lesssim\big(C_1\big)^4\varepsilon^4.
	\end{equation}
\end{lemma}
\begin{proof}
In fact, we can use the standard Sobolev inequality on $\mathbb{R}^3$ (for $x$) to get 
	\begin{align*}
		\int_{\mathbb{R}}\big\|\langle u_\mp\rangle^{\omega}\langle u_\pm\rangle^{\frac{\omega}{2}}\nabla p\big\|^2_{L^\infty_{x}}d\tau&\lesssim\int_{\mathbb{R}}\bigg(\big\|\langle u_\mp\rangle^{\omega}\langle u_\pm\rangle^{\frac{\omega}{2}}\nabla p\big\|^2_{L^2_{x}}+\big\|\nabla^2\big(\langle u_\mp\rangle^{\omega}\langle u_\pm\rangle^{\frac{\omega}{2}}\nabla p\big)\big\|^2_{L^2_{x}}\bigg)d\tau\\
		&\stackrel{\eqref{differentiate weights coro}}{\lesssim}\int_{\mathbb{R}}\bigg(\big\|\langle u_\mp\rangle^{\omega}\langle u_\pm\rangle^{\frac{\omega}{2}}\nabla p\big\|^2_{L^2(\mathbb{R}^3)}+\big\| \langle u_\mp\rangle^{\omega}\langle u_\pm\rangle^{\frac{\omega}{2}}\nabla^3 p\big\|^2_{L^2(\mathbb{R}^3)}\bigg)d\tau\\
		&=\int_{\mathbb{R}\times \mathbb{R}^3}\langle u_{\mp}\rangle^{2\omega}\langle u_{\pm}\rangle^\omega |\nabla p|^2+\int_{\mathbb{R}\times \mathbb{R}^3}\langle u_{\mp}\rangle^{2\omega}\langle u_{\pm}\rangle^\omega |\nabla^3 p|^2\\
		&\stackrel{\eqref{eq:pressureL2L2}}{\lesssim}\big(C_1\big)^4\varepsilon^4,
	\end{align*}
which gives the desired result of this lemma. 
\end{proof}

\smallskip

Moreover, we can also derive the space-time estimates for higher order derivatives of the pressure.

\begin{lemma}\label{lemma:pressureL2L2-H} 
	For $k=1,2,...,N_*+1$, we have the following space-time estimates on pressure:
	\begin{equation}\label{eq:pressureL2L2-H}
		\int_{\mathbb{R}\times \mathbb{R}^3} \langle u_{\mp}\rangle^{2\omega}\langle u_{\pm}\rangle^\omega |\nabla^k\nabla p |^2\lesssim\big(C_1\big)^4\varepsilon^4.
	\end{equation}
\end{lemma}
\begin{proof}
Observing that  $\operatorname{div}\nabla p=\Delta p=-\partial_iz_-^j\partial_jz_+^i$ and $\operatorname{curl}\nabla p=0$, 	
	we first take $v=\nabla p$ in Lemma \ref{lemma:divcurl} and apply Corollary \ref{coro:divcurl} to derive  
	\begin{align*}
		&\ \ \ \ \big\|\langle u_\mp\rangle^{\omega}\langle u_\pm\rangle^{\frac{\omega}{2}}\nabla^k \nabla p\big\|_{L^2(\mathbb{R}^3)}^2\\
		& \lesssim\sum_{l=0}^{k-1}\big\|\langle u_\mp\rangle^{\omega}\langle u_\pm\rangle^{\frac{\omega}{2}}\operatorname{div }\nabla^l\nabla p\big\|_{L^2(\mathbb{R}^3)}^2+\sum_{l=0}^{k-1} \big\|\langle u_\mp\rangle^{\omega}\langle u_\pm\rangle^{\frac{\omega}{2}}\operatorname{curl }\nabla^l\nabla p\big\|_{L^2(\mathbb{R}^3)}^2+ \big\|\langle u_\mp\rangle^{\omega}\langle u_\pm\rangle^{\frac{\omega}{2}} \nabla p\big\|_{L^2(\mathbb{R}^3)}^2\\
		&\lesssim\sum_{l=0}^{k-1}\big\|\langle u_\mp\rangle^{\omega}\langle u_\pm\rangle^{\frac{\omega}{2}}\nabla^l(\nabla z_\mp\cdot \nabla z_\pm)\big\|_{L^2(\mathbb{R}^3)}^2+\big\|\langle u_\mp\rangle^{\omega}\langle u_\pm\rangle^{\frac{\omega}{2}}\nabla p\big\|_{L^2(\mathbb{R}^3)}^2\\
		&\lesssim\sum_{l=0}^{k-1}\sum_{s=0}^l\big\|\langle u_\mp\rangle^{\omega}\langle u_\pm\rangle^{\frac{\omega}{2}}\big(\nabla z_\mp^{(s)}\cdot \nabla z_\pm^{(l-s)}\big)\big\|_{L^2(\mathbb{R}^3)}^2+\big\|\langle u_\mp\rangle^{\omega}\langle u_\pm\rangle^{\frac{\omega}{2}}\nabla p\big\|_{L^2(\mathbb{R}^3)}^2.
	\end{align*}
It is now obvious that 
	\begin{equation}\label{eq:I1sumI2}
	\mathbf{I}:=	\int_{\mathbb{R}\times \mathbb{R}^3} \langle u_{\mp}\rangle^{2\omega}\langle u_{\pm}\rangle^\omega \big|\nabla^k\nabla p \big|^2\lesssim \underbrace{\sum_{l=0}^{k-1}\sum_{s=0}^l\int_{\mathbb{R}\times \mathbb{R}^3}\langle u_\mp\rangle^{2\omega}\langle u_\pm\rangle^{\omega}\big|\nabla z_\mp^{(s)}\big|^2\big|\nabla z_\pm^{(l-s)}\big|^2}_{\mathbf{I_1}}+\underbrace{\int_{\mathbb{R}\times \mathbb{R}^3}\langle u_\mp\rangle^{2\omega}\langle u_\pm\rangle^{\omega}|\nabla p|^2}_{\mathbf{I_2}}.
	\end{equation}
Thanks to Lemma \ref{lemma:pressureL2L2}, we have 
\begin{equation}\label{eq:I2}
	\mathbf{I_2}\lesssim\big(C_1\big)^4\varepsilon^4.
\end{equation}
Hence it remains to bound $\mathbf{I_1}$.

\smallskip

For this purpose, we distinguish the following two cases according to the order of derivatives:\\
\textbf{Case 1}: $1\leqslant k\leqslant N_*-1$. In this case, we can use Lemma \ref{lemma:Sobolev} and Lemma \ref{lemma:flux} to get 
	\begin{align*}
		\mathbf{I_1}
		&\lesssim \sum_{l=0}^{k-1}\sum_{s=0}^l\int_{\mathbb{R}}\big\|\langle u_\pm\rangle^{\omega}\nabla z_\mp^{(s)}\big\|^2_{L^\infty(\mathbb{R}^3)}\int_{\mathbb{R}^3}\frac{\langle u_\mp\rangle^{2\omega}}{\langle u_\pm\rangle^{\omega}}\big|\nabla z_\pm^{(l-s)}\big|^2\\
		&\stackrel{\eqref{eq:Sobolev},\eqref{eq:flux}}{\lesssim}\big(C_1\big)^4\varepsilon^4.\stepcounter{equation}\tag{\theequation}\label{eq:I1case1}
\end{align*}
\textbf{Case 2}:  $k=N_*$ or $N_*+1$. In such a case, we can rewrite $\mathbf{I_1}$ as
\begin{align*}
	\mathbf{I_1}
		&\lesssim \sum_{l=0}^{N_*-2}\sum_{s=0}^l\int_{\mathbb{R}\times \mathbb{R}^3}\langle u_\mp\rangle^{2\omega}\langle u_\pm\rangle^{\omega}|\nabla z_\mp^{(s)}|^2|\nabla z_\pm^{(l-s)}|^2+\sum_{l=N_*-1}^{k-1}\sum_{s=0}^l\int_{\mathbb{R}\times \mathbb{R}^3}\langle u_\mp\rangle^{2\omega}\langle u_\pm\rangle^{\omega}|\nabla z_\mp^{(s)}|^2|\nabla z_\pm^{(l-s)}|^2\\
		&\stackrel{\eqref{eq:I1case1}}{\lesssim}\big(C_1\big)^4\varepsilon^4+\Big(\underbrace{\sum_{l=N_*-1}^{k-1}\sum_{s=0}^{N_*-2}}_{\mathbf{I_{11}}}+\underbrace{\sum_{l=N_*-1}^{k-1}\sum_{s=N_*-1}^l}_{\mathbf{I_{12}}}\Big)\int_{\mathbb{R}\times \mathbb{R}^3}\langle u_\mp\rangle^{2\omega}\langle u_\pm\rangle^{\omega}|\nabla z_\mp^{(s)}|^2|\nabla z_\pm^{(l-s)}|^2.\stepcounter{equation}\tag{\theequation}\label{eq:I1case2-1}
	\end{align*}
For $\mathbf{I_{11}}$, one can still afford the order of derivatives by using an argument similar to \eqref{eq:I1case1}, which shows that 
\begin{align*}
\mathbf{I_{11}}
	&\lesssim \sum_{l=N_*-1}^{k-1}\sum_{s=0}^{N_*-2}\int_{\mathbb{R}}\big\|\langle u_\pm\rangle^{\omega}\nabla z_\mp^{(s)}\big\|^2_{L^\infty(\mathbb{R}^3)}\int_{\mathbb{R}^3}\frac{\langle u_\mp\rangle^{2\omega}}{\langle u_\pm\rangle^{\omega}}\big|\nabla z_\pm^{(l-s)}\big|^2\\
	&\stackrel{\eqref{eq:Sobolev},\eqref{eq:flux}}{\lesssim}\big(C_1\big)^4\varepsilon^4.\stepcounter{equation}\tag{\theequation}\label{eq:I11}
	\end{align*}
For $\mathbf{I_{12}}$, the argument above will lose control (since the order of derivative term $\nabla z_\mp^{(s)}$ with $s=N_*-1$ goes beyond the highest order of derivatives allowed in Lemma \ref{lemma:Sobolev}). 
Luckily, using H\"older inequality leads us to 
\begin{align*}	
	\mathbf{I_{12}}
	&=\sum_{l=N_*-1}^{k-1}\int_{\mathbb{R}}\int_{\mathbb{R}^3}\langle u_\pm\rangle^{2\omega}|\nabla z_\mp^{(N_*-1)}|^2\cdot \frac{\langle u_\mp\rangle^{2\omega}}{\langle u_\pm\rangle^{\omega}}|\nabla z_\pm^{(l+1-N_*)}|^2\\
	&\stackrel{\text{H\"older}}{\lesssim}\sum_{l=N_*-1}^{k-1}\big\|\langle u_\pm\rangle^{\omega}\nabla z_\mp^{(N_*-1)}\big\|^2_{L^\infty_\tau L^2_x}\Big\|\frac{\langle u_\mp\rangle^{\omega}}{\langle u_\pm\rangle^{\frac{\omega}{2}}}\nabla z_\pm^{(l+1-N_*)}\Big\|^2_{L^2_\tau L^\infty_x}.
\end{align*}
On the one hand, by virtue of Remark \ref{rmk:divcurl} and \eqref{Bootstrap on energy}, it holds that  
\[\big\|\langle u_\pm\rangle^{\omega}\nabla z_\mp^{(N_*-1)}\big\|^2_{L^\infty_\tau L^2_x}=\sup_{t\in[0,t^*]}\big\|\langle u_\pm\rangle^{\omega}\nabla z_\mp^{(N_*-1)}\big\|_{L^2_x}\stackrel{\eqref{eq:z to j inequality},\eqref{Bootstrap on energy}}{\lesssim}C_1\varepsilon.\]
On the other hand, for $l= N_*-1$ or $l=N_*$, the standard Sobolev inequality on $\mathbb{R}^3$ (for $x$) gives rise to
\begin{align*}
	\Big\|\frac{\langle u_{\mp}\rangle^{\omega}}{\langle u_{\pm}\rangle^{\frac{\omega}{2}}}
	\nabla z^{(l+1-N_*)}_{\pm} \Big\|_{L^2_\tau L^\infty_x}
	&\lesssim\Big\|\frac{\langle u_{\mp}\rangle^{\omega}}{\langle u_{\pm}\rangle^{\frac{\omega}{2}}}
	\nabla z^{(l+1-N_*)}_{\pm} \Big\|_{L^2_\tau H^2_x}\\
	& \lesssim
	\Big\|\frac{\langle u_{\mp}\rangle^{\omega}}{\langle u_{\pm}\rangle^{\frac{\omega}{2}}}
	\nabla z^{(l+1-N_*)}_{\pm} \Big\|_{L^2_\tau L^2_x}
	+\Big\|\nabla^2\Big(\frac{\langle u_{\mp}\rangle^{\omega}}{\langle u_{\pm}\rangle^{\frac{\omega}{2}}}\nabla z^{(l+1-N_*)}_{\pm} \Big)
	\Big\|_{L^2_\tau L^2_x}\\
	&\stackrel{\eqref{differentiate weights coro}}{\lesssim}
	\sum_{m=l+1-N_*}^{l+3-N_*}\Big\|\frac{\langle u_{\mp}\rangle^{\omega}}{\langle u_{\pm}\rangle^{\frac{\omega}{2}}}
	\nabla z^{(m)}_{\pm}\Big\|_{L^2_\tau L^2_x}\\
	&\stackrel{\eqref{eq:flux}}{\lesssim}C_1\varepsilon.
\end{align*}
Therefore we have
\begin{equation}\label{eq:I12}
	\mathbf{I_{12}}\lesssim\big(C_1\big)^4\varepsilon^4.
\end{equation}
Plugging \eqref{eq:I11} and \eqref{eq:I12} back into \eqref{eq:I1case2-1}, we obtain 
\begin{equation}\label{eq:I1case2-2}
\mathbf{I_{1}}\lesssim\big(C_1\big)^4\varepsilon^4.
\end{equation}

\smallskip

According to \eqref{eq:I}, \eqref{eq:I2}, \eqref{eq:I1case1} and \eqref{eq:I1case2-2}, we can conclude for $k=1,2,...,N_*+1$ that 
\[\mathbf{I}\lesssim\big(C_1\big)^4\varepsilon^4.\]
The proof of this lemma is now complete. 
\end{proof}

\medskip

Indeed, what we have proved in Lemma \ref{lemma:pressureL2L2} and Lemma \ref{lemma:pressureL2L2-H} can be summarized as follows:
\begin{lemma}\label{lemma:pressureL2L2-all} For $l=1,2,...,N_*+2$, we have the following space-time estimates on pressure:
\begin{equation}\label{eq:pressureL2L2-all}
	\int_{\mathbb{R}\times \mathbb{R}^3} \langle u_{\mp}\rangle^{2\omega}\langle u_{\pm}\rangle^\omega |\nabla^l p |^2\lesssim\big(C_1\big)^4\varepsilon^4.
\end{equation}
\end{lemma}
Together with Lemma \ref{lemma:pressureL2L2-all}, the analysis which gave rise to Lemma \ref{lemma:pressureL2Linfty} can be repeated to yield: 
\begin{lemma}\label{lemma:pressureL2Linfty-all} For $l=1,2,...,N_*$, we have the following $L_\tau^2L_x^\infty$ estimates on pressure:
	\begin{equation}\label{eq:pressureL2Linfty-all}
		\int_{\mathbb{R}}\big\|\langle u_\mp\rangle^{\omega}\langle u_\pm\rangle^{\frac{\omega}{2}}\nabla^l p\big\|^2_{L^\infty_{x}}d\tau\lesssim\big(C_1\big)^4\varepsilon^4.
	\end{equation}
\end{lemma}

\medskip

With these preliminary results to keep in mind, we are now ready to pursue the three goals \eqref{Bootstrap on fluctuation improved}-\eqref{Bootstrap on energy improved}. 

\bigskip

\section{Closure of bootstrap argument and Main energy estimates}\label{sec:energy}

Let us achieve these three goals \eqref{Bootstrap on fluctuation improved}-\eqref{Bootstrap on energy improved} one by one in this section. 

\medskip

\subsection{Bootstrap on the amplitude of fluctuation}

Based on Lemma \ref{lemma:Sobolev} (Weighted Sobolev inequality), it is now trivial to improve the amplitude ansatz \eqref{Bootstrap on fluctuation} as follows: 
\[\|z_\pm\|_{L^\infty}\stackrel{\eqref{eq:Sobolev}}{\leqslant}\frac{C_1'\varepsilon}{\langle u_{\pm}\rangle^\omega}\stackrel{\eqref{weight}}{\leqslant}\frac{C_1'\varepsilon}{R^{1+\delta}}\stackrel{\varepsilon\ll 1}{\leqslant}\frac{1}{4},\]
where $C_1'$ is a universal constant and in the last step  $\varepsilon$ is taken sufficiently small. This improves the ansatz \eqref{Bootstrap on fluctuation} to the goal \eqref{Bootstrap on fluctuation improved}. 

\medskip

\subsection{Bootstrap on the underlying geometry}
Now we improve the ansatz \eqref{Bootstrap on geometry} to the goal \eqref{Bootstrap on geometry improved} by using the language of flow.
Recall that the flow $\psi_{\pm}(t,y)$ can be given by \eqref{flow in integration}. 
Denote the differential of $\psi_{\pm}(t,y)$ at $y$ by $\frac{\partial\psi_{\pm}(t,y)}{\partial y}$. 
Hence we obtain 
\begin{equation}\label{flow differential}
	\frac{\partial \psi_\pm(t,y)}{\partial y}=\mathrm{I}+\int_0^t(\nabla z_\pm)(\tau,\psi_\pm(\tau,y))\frac{\partial\psi_\pm(\tau,y)}{\partial y}d\tau.
\end{equation}

\bigskip
For one thing, a simple manipulation from \eqref{flow differential} yields
\begin{equation*}
	\Big|\frac{\partial\psi_\pm(t,y)}{\partial y}-\mathrm{I}\Big|\leqslant\int_0^t\Big(\big|(\nabla z_\pm)(\tau,\psi_\pm(\tau,y))\big|\Big|\frac{\partial\psi_\pm(\tau,y)}{\partial y}-\mathrm{I}\Big|+\big|(\nabla z_\pm)(\tau,\psi_\pm(\tau,y))\big|\Big)d\tau,
\end{equation*}
and moreover differentiating both sides of this inequality gives
\begin{equation*}
\frac{d}{dt}\Big|\frac{\partial\psi_\pm(t,y)}{\partial y}-\mathrm{I}\Big|\leqslant\big|(\nabla z_\pm)(t,\psi_\pm(t,y))\big|\Big|\frac{\partial\psi_\pm(t,y)}{\partial y}-\mathrm{I}\Big|+\big|(\nabla z_\pm)(t,\psi_\pm(t,y))\big|.
\end{equation*}
Thanks to Gronwall's inequality, we acquire
\begin{align*}
	\Big|\frac{\partial\psi_\pm(t,y)}{\partial y}-\mathrm{I}\Big|&\leqslant\exp\Big(\int_0^t\big|(\nabla z_\pm)(\tau,\psi_\pm(\tau,y))\big|d\tau\Big)\underbrace{\int_0^t\big|(\nabla z_\pm)(\tau,\psi_\pm(\tau,y))\big|d\tau}_{A_\pm}=e^{A_\pm}A_\pm.\stepcounter{equation}\tag{\theequation}\label{Apm definition}
\end{align*}
To bound the integration $A_\pm$, we first notice that \eqref{eq:Sobolev} leads to 
\begin{align*}
	|\nabla z_\pm(\tau,\psi_\pm(\tau,y))|\lesssim\frac{\varepsilon}{\langle u_\mp\rangle^\omega\big|_{x=\psi_\pm(\tau,y)}}.
\end{align*}
Then, noting that $L_\pm u_\pm=0$ and $\partial_t\psi_\pm(\tau,y)=Z_\pm(\tau,\psi_\pm(\tau,y))$, we can calculate the Jacobian as follows:
\begin{align*}
	\frac{d}{d\tau}u_\mp(\tau,\psi_\pm(\tau,y))
	&=(\partial_t u_\mp)(\tau,\psi_\pm(\tau,y))+\partial_t\psi_\pm(\tau,y)\cdot(\nabla u_\mp)(\tau,\psi_\pm(\tau,y))\\
	&=(Z_\pm-Z_\mp)(\tau,\psi_\pm(\tau,y))\cdot(\nabla u_\mp)(\tau,\psi_\pm(\tau,y))\\
	&=(z_\pm-z_\mp\pm 2B_0)(\tau,\psi_\pm(\tau,y))\cdot(\nabla u_\mp)(\tau,\psi_\pm(\tau,y))\\
	&=\pm2(\partial_3 u_\mp)\big|_{x=\psi_\pm(\tau,y)}+\big((z_\pm-z_\mp)\cdot(\nabla u_\mp)\big)\big|_{x=\psi_\pm(\tau,y)}.
\end{align*}
Together with $u_\mp=x_3^\mp$,  \eqref{Bootstrap on fluctuation} and \eqref{Bootstrap on geometry}, it is clear to infer that 
\begin{align*}
	\Big|\frac{d}{d\tau}u_\mp(\tau,\psi_\pm(\tau,y))\Big|
	&\geqslant\Big|\pm2(\partial_3 u_\mp)\big|_{x=\psi_\pm(\tau,y)}\Big|_{\min}+\Big|z_\pm\cdot(\nabla u_\mp)\big|_{x=\psi_\pm(\tau,y)}\Big|_{\min}
	+\Big|-z_\mp\cdot(\nabla u_\mp)\big|_{x=\psi_\pm(\tau,y)}\Big|_{\min}\\
	&=2(1-2C_0\varepsilon)-\frac{1}{2}(1+2C_0\varepsilon)-\frac{1}{2}(1+2C_0\varepsilon)=1-6C_0\varepsilon.
\end{align*}
%Thus, 
By taking $\varepsilon$ sufficiently small, we further obtain
\[\Big|	\frac{d}{d\tau}u_\mp(\tau,\psi_\pm(\tau,y))\Big|\geqslant\frac{1}{2}.\]
Thus we can switch the variable $\tau$ to $u_\mp$ in the integration $A_\pm$ to get 
\begin{align*}
	A_\pm
	&\lesssim\int_0^t\frac{\varepsilon}{\langle u_\mp\rangle^\omega\big|_{x=\psi_\pm(\tau,y)}}d\tau\\
	&\lesssim \varepsilon\sup_{\tau}\frac{1}{\big|\frac{d}{d\tau}u_\mp(\tau,\psi_\pm(\tau,y))\big|}\int_{\mathbb{R}} \frac{1}{\langle u_\mp\rangle^\omega}du_\mp\\
	&\lesssim \varepsilon.\stepcounter{equation}\tag{\theequation}\label{Apm bound}
\end{align*}
In the last step, we have used the fact that $\displaystyle \int_{\mathbb{R}}\frac{1}{\langle u_\mp\rangle^\omega}du_\mp$ is bounded by a universal constant. According to \eqref{Apm definition} and \eqref{Apm bound}, we can derive 
\begin{equation}\label{part 1}
	\Big|\frac{\partial \psi_{\pm}(t,y)}{\partial y}-\mathrm{I}\Big|\leqslant e^{A_\pm}A_\pm\lesssim \varepsilon.
\end{equation}

\bigskip

For another thing, we apply $\partial_k$ (with $k=1,2,3$) to \eqref{flow differential} and then use the chain rule to get 
\begin{equation*}
	\partial_k\Big(\frac{\partial\psi_\pm(t,y)}{\partial x}\Big)=\int_0^t\bigg((\nabla z_\pm)(\tau,\psi_\pm(\tau,y))\partial_k\Big(\frac{\partial\psi_\pm(\tau,y)}{\partial y}\Big)+\partial_k\big((\nabla z_\pm)(\tau,\psi_\pm(\tau,y))\big)\Big(\frac{\partial\psi_\pm(\tau,y)}{\partial y}\Big)\bigg)d\tau,
\end{equation*}
which implies
\begin{align*}
	\frac{d}{dt}\Big|\partial_k\Big(\frac{\partial\psi_\pm(t,y)}{\partial x}\Big)\Big|&\leqslant\big|(\nabla z_\pm)(\tau,\psi_\pm(\tau,y))\big|\Big|\partial_k\Big(\frac{\partial\psi_\pm(\tau,y)}{\partial y}\Big)\Big|+\big|\partial_k\big((\nabla z_\pm)(\tau,\psi_\pm(\tau,y))\big)\big|\Big|\frac{\partial\psi_\pm(\tau,y)}{\partial y}\Big|.
\end{align*}
Collecting this inequality and Gronwall's inequality, we see that
\begin{align*}
	\Big|\partial_k\Big(\frac{\partial\psi_\pm(t,y)}{\partial y}\Big)\Big|
	&\leqslant\exp\Big(\int_0^t\big|(\nabla z_\pm)(\tau,\psi_\pm(\tau,y))\big|d\tau\Big)\int_0^t\big|\partial_k\big((\nabla z_\pm)(\tau,\psi_\pm(\tau,y))\big)\big|\Big|\frac{\partial\psi_\pm(\tau,y)}{\partial y}\Big|d\tau\\
	&\leqslant\exp\Big(\int_0^t\big|(\nabla z_\pm)(\tau,\psi_\pm(\tau,y))\big|d\tau\Big)\int_0^t\big|(\nabla^2 z_\pm)(\tau,\psi_\pm(\tau,y))\big|\Big|\frac{\partial\psi_\pm(\tau,y)}{\partial y}\Big|^2d\tau\\
	&\stackrel{\eqref{part 1}}{\lesssim}\exp\Big(\int_0^t\big|(\nabla z_\pm)(\tau,\psi_\pm(\tau,y))\big|d\tau\Big)
	\underbrace{\int_0^t\big|(\nabla^2 z_\pm)(\tau,\psi_\pm(\tau,y))\big|d\tau}_{B_\pm}=e^{A_\pm}B_\pm.\stepcounter{equation}\tag{\theequation}\label{Bpm definition}
\end{align*}
By virtue of \eqref{eq:Sobolev},  we also have
\begin{equation*}
	|\nabla^2 z_\pm(\tau,\psi_+(\tau,y))|\lesssim\frac{\varepsilon}{\langle u_\mp\rangle^\omega\big|_{x=\psi_\pm(\tau,y)}}.
\end{equation*}
We then repeat the procedure of the estimates on $A_\pm$ to obtain 
\begin{equation}\label{Bpm bound}
	B_\pm\lesssim\varepsilon.
\end{equation}
According to \eqref{Apm bound}, \eqref{Bpm definition} and \eqref{Bpm bound}, by taking $\varepsilon$ sufficiently small, we can derive
\begin{equation}\label{part 2}
	\Big|\nabla_y\Big(\frac{\partial \psi_{\pm}(t,y)}{\partial y}\Big)\Big|\lesssim  \varepsilon.
\end{equation}

\bigskip

Up to now, we have improved the geometry ansatz \eqref{Bootstrap on geometry} to \eqref{part 1} and \eqref{part 2} in the language of flow, i.e. 
\begin{equation}\label{part}
	\Big|\frac{\partial \psi_{\pm}(t,y)}{\partial y}-\mathrm{I}_{3\times3}\Big|\leqslant  C_0' \varepsilon,\ \ \ \ 	\Big|\nabla_y\Big(\frac{\partial \psi_{\pm}(t,y)}{\partial y}\Big)\Big|\leqslant C_0' \varepsilon,
\end{equation}
where $C_0'$ is a universal constant. By definition,
we know that $x^\pm(t,\psi_\pm(t,y))=y$, which leads us to 
\begin{equation*}
	\frac{\partial x^\pm}{\partial x}\Big|_{x=\psi_\pm(t,y)}=\Big(\frac{\partial\psi_{\pm}(t,y)}{\partial y}\Big)^{-1}, \ \ \ \   \nabla_x\Big(\frac{\partial x^\pm}{\partial x}\Big)\Big|_{x=\psi_\pm(t,y)}=\nabla_y\Big\{\Big(\frac{\partial\psi_{\pm}(t,y)}{\partial y}\Big)^{-1}\Big\}\Big(\frac{\partial\psi_{\pm}(t,y)}{\partial y}\Big)^{-1}.
\end{equation*}
Now we turn to rephrase \eqref{part}. 
For the first part, in view of series expansion, it follows that 
\begin{align*}
	\frac{\partial x^\pm}{\partial x}\Big|_{x=\psi_\pm(t,y)}
	&=\Big(\mathrm{I}-\Big(\mathrm{I}-\Big(\frac{\partial x^\pm}{\partial x}\Big|_{x=\psi_\pm(t,y)}\Big)^{-1}\Big)\Big)^{-1}\\
	&=\sum_{k=0}^\infty\Big(\mathrm{I}-\Big(\frac{\partial x^\pm}{\partial x}\Big|_{x=\psi_\pm(t,y)}\Big)^{-1}\Big)^k\\
	&=\mathrm{I}+\sum_{k=1}^\infty\Big(\mathrm{I}-\frac{\partial\psi_{\pm}(t,y)}{\partial y}\Big)^k,
\end{align*}
and hence we obtain
\begin{align*}
	\Big|\frac{\partial x^\pm}{\partial x}\Big|_{x=\psi_\pm(t,y)}-\mathrm{I}\Big|
	&\leqslant\sum_{k=1}^\infty\Big|\mathrm{I}-\frac{\partial\psi_{\pm}(t,y)}{\partial y}\Big|^k\\
	&=\Big|\mathrm{I}-\frac{\partial\psi_{\pm}(t,y)}{\partial y}\Big|\cdot \sum_{k=1}^\infty\Big|\mathrm{I}-\frac{\partial\psi_{\pm}(t,y)}{\partial y}\Big|^{k-1}\\
	&=\Big|\mathrm{I}-\frac{\partial\psi_{\pm}(t,y)}{\partial y}\Big|\cdot\frac{1}{1-\Big|\mathrm{I}-\frac{\partial\psi_{\pm}(t,y)}{\partial y}\Big|}\\
	&\stackrel{\eqref{part}}{\leqslant} \frac{C_0'\varepsilon}{1-C_0'\varepsilon}\\
	&\leqslant C_0'\varepsilon.\stepcounter{equation}\tag{\theequation}\label{part 1'}
\end{align*}
For the second part,  combining the chain rule and \eqref{part} as well as \eqref{part 1'} gives the estimate
\begin{align*}
\Big|\nabla_x\Big(\frac{\partial x^\pm}{\partial x}\Big)\Big|_{x=\psi_\pm(t,y)}\Big|
&=\Big|\nabla_y\Big\{\Big(\frac{\partial\psi_{\pm}(t,y)}{\partial y}\Big)^{-1}\Big\}\Big(\frac{\partial\psi_{\pm}(t,y)}{\partial y}\Big)^{-1}\Big|\\
&\leqslant\Big|\nabla_y\Big(\frac{\partial\psi_{\pm}(t,y)}{\partial y}\Big)\Big| \Big|\Big(\frac{\partial\psi_{\pm}(t,y)}{\partial y}\Big)^{-3}\Big|\\
&\leqslant\Big|\nabla_y\Big(\frac{\partial\psi_{\pm}(t,y)}{\partial y}\Big)\Big| \Big|\frac{\partial x^\pm}{\partial x}\Big|_{x=\psi_\pm(t,y)}\Big|^{3}\\
&\leqslant C_0'\varepsilon(1+C_0'\varepsilon)^3\\
&\leqslant C_0''\varepsilon,\stepcounter{equation}\tag{\theequation}\label{part 2'}
\end{align*}
where $C_0''$ is a universal constant. 
Taking $C_0\geqslant\max\{C_0',C_0''\}$, we can  summarize \eqref{part 1'} and \eqref{part 2'} as 
\[	\Big|\frac{\partial x^\pm}{\partial x}\Big|_{x=\psi_\pm(t,y)}-\mathrm{I}\Big|\leqslant  C_0 \varepsilon,\ \ \ \  \Big|\nabla_x\Big(\frac{\partial x^\pm}{\partial x}\Big)\Big|_{x=\psi_\pm(t,y)}\Big|\leqslant  C_0 \varepsilon,\]
which yields \eqref{Bootstrap on geometry improved} immediately.

\medskip

\subsection{Bootstrap on the total energy bound}

The current subsection is devoted to deriving energy estimates and pursuing the energy goal \eqref{Bootstrap on energy improved}.  
The proof is divided into three steps. 

\bigskip

The first step is to derive the lowest order energy estimates, i.e. control the size of $E_\pm$ and $F_\pm$.
Taking
\[f_{\pm}=z_{\pm},\ \ \ \  \rho_{\pm}=-\nabla p,\ \ \ \  \lambda_{\pm}=\langle u_{\mp}\rangle^{2\omega},\]
we can apply \eqref{eq:linear EE} to \eqref{eq:MHD} and then get 
\begin{equation}\label{eq:LOW}
	\sum_{+,-}E_{\pm}(t)+\sum_{+,-}F_{\pm}(t)\lesssim \sum_{+,-}E_{\pm}(0)+\sum_{+,-}\underbrace{\int_{0}^{t}\int_{\Sigma_{\tau}}\langle u_{\mp}\rangle^{2\omega}|z_{\pm}| |\nabla p|}_{\mathbf{J}}.
\end{equation}
Thanks to H\"older inequality, it follows that
\begin{align*}
	\mathbf{J} &\leqslant \Big(\int_{0}^{t}\int_{\Sigma_{\tau}}\frac{\langle u_{\mp}\rangle^{2\omega}}{\langle u_{\pm}\rangle^{\omega}}|z_{\pm}|^2\Big)^{\frac{1}{2}}
	\Big(\int_{0}^{t}\int_{\Sigma_{\tau}}\langle u_{\mp}\rangle^{2\omega}\langle u_{\pm}\rangle^\omega |\nabla p|^2 \Big)^{\frac{1}{2}}\\
	&\stackrel{\eqref{eq:flux},\eqref{eq:pressureL2L2-all}}{\lesssim}\big(C_1\big)^3 \varepsilon^3.
\end{align*}
Putting the above equality into \eqref{eq:LOW}, we can take supremum over all $t\in [0,t^*]$ to acquire
\begin{equation}\label{eq:LOWEE}
	\sum_{+,-}E_{\pm}+\sum_{+,-}F_{\pm}\lesssim  \sum_{+,-}E_{\pm}(0)+\big(C_1\big)^3\varepsilon^3.
\end{equation}

\bigskip
The second step is devoted to the higher order energy estimates.
For a given multi-index $\beta$ with $0\leqslant|\beta|\leqslant N_*$, we first commute $\partial^{\beta}$ derivatives with the vorticity equations in \eqref{eq:curlMHD-2} to obtain
\begin{equation}\begin{cases}\label{eq:DcurlMHD-2}
		\partial_t j^{(\beta)}_{+}+Z_-\cdot\nabla j^{(\beta)}_{+}&=\rho^{(\beta)}_{+},\\
		\partial_t j^{(\beta)}_{-}+Z_+\cdot\nabla j^{(\beta)}_{-}&=\rho^{(\beta)}_{-},
\end{cases}\end{equation}
where source terms $\rho_\pm^{(\beta)}$ are given by
\begin{equation*}%\label{definition of rhobeta}
	\rho_\pm^{(\beta)}=-\partial^\beta(\nabla z_\mp\wedge\nabla z_\pm)-[\partial^\beta,z_\mp\cdot\nabla] j_\pm.
\end{equation*}
Subsequently, applying \eqref{eq:linear EE} to \eqref{eq:DcurlMHD-2} with the weight functions $\lambda_{\pm}=\langle u_{\mp}\rangle^{2\omega}$ leads us to 
\begin{equation}\label{eq:HIGH}
\sum_{+,-}	E^{(\beta)}_{\pm}(t)+\sum_{+,-}F^{(\beta)}_{\pm}(t)\lesssim \sum_{+,-}E^{(\beta)}_{\pm}(0) +\sum_{+,-}\underbrace{\int_{0}^{t}\int_{\Sigma_{\tau}}\langle u_{\mp}\rangle^{2\omega}\big|j^{(\beta)}_{\pm}\big|\big|\rho^{(\beta)}_{\pm}}_{\mathbf{K}}\big|.
\end{equation}
By virtue of H\"older inequality, we have
\begin{align*}
	\mathbf{K} &\leqslant \Big(\int_{0}^{t}\int_{\Sigma_{\tau}}\frac{\langle u_{\mp}\rangle^{2\omega}}{\langle u_{\pm}\rangle^{\omega}}\big|j^{(\beta)}_{\pm}\big|^2\Big)^{\frac{1}{2}}
	\Big(\int_{0}^{t}\int_{\Sigma_{\tau}}\langle u_{\mp}\rangle^{2\omega}\langle u_{\pm}\rangle^\omega \big|\rho^{(\beta)}_\pm\big|^2 \Big)^{\frac{1}{2}}\\
	&\stackrel{\eqref{eq:flux}}{\lesssim}C_1 \varepsilon\Big(\underbrace{\int_{0}^{t}\int_{\Sigma_{\tau}}\langle u_{\mp}\rangle^{2\omega}\langle u_{\pm}\rangle^\omega \big|\rho^{(\beta)}_\pm\big|^2 }_{\mathbf{K'}}\Big)^{\frac{1}{2}}.
\end{align*}
We note that the source term $\rho_\pm^{(\beta)}$ in $\eqref{eq:DcurlMHD-2}$ can be bounded by
\begin{equation*}
	\big|\rho_\pm^{(\beta)}\big|\leqslant\sum_{\gamma\leqslant\beta} C_\beta^\gamma\big|\nabla z_\mp^{(\gamma)}\big|\big|\nabla z_\pm^{(\beta-\gamma)}\big|
	+\sum_{0\neq\gamma\leqslant\beta} C_\beta^\gamma\big|z_\mp^{(\gamma)}\big|\big|\nabla j_\pm^{(\beta-\gamma)}\big|
	\lesssim\sum_{s\leqslant|\beta|} \big|\nabla z_\mp^{(s)}\big|\big|\nabla z_\pm^{(|\beta|-s)}\big|.
\end{equation*}
As a consequence, we can use the bound on $\mathbf{I_1}$ in \eqref{eq:I1sumI2} to derive 
\begin{equation*}
	\mathbf{K'}\lesssim\sum_{s\leqslant|\beta|}\int_{\mathbb{R}}\int_{\Sigma_{\tau}}\langle u_{\mp}\rangle^{2\omega}\langle u_{\pm}\rangle^\omega\big|\nabla z^{(s)}_{\mp} \big|^2\big|\nabla z^{(|\beta|-s)}_{\pm} \big|^2\lesssim\mathbf{I_1}\lesssim\big(C_1\big)^4\varepsilon^4.
\end{equation*}
Hence we obtain 
\begin{equation*}
	\mathbf{K}\lesssim \big(C_1\big)^3\varepsilon^3.
\end{equation*}
Together with \eqref{eq:HIGH},  for all $t\in [0,t^*]$ and for all $\beta$ with $0\leqslant|\beta|\leqslant N_*$, it then follows that 
\begin{equation}\label{eq:H}
	\sum_{+,-}E_{\pm}^{(\beta)}(t)+\sum_{+,-}F^{(\beta)}_{\pm}(t)\lesssim\sum_{+,-} E_{\pm}^{(\beta)}(0)+\big(C_1\big)^3\varepsilon^3.
\end{equation}
Summing up \eqref{eq:H} for all $0\leqslant |\beta|\leqslant N_{*}$ and taking  supremum over all  $t\in [0,t^*]$, we can summarize that 
\begin{equation}\label{eq:HIGHEE}
	\sum_{+,-}\sum_{k=0}^{ N_{*}}E_{\pm}^k+\sum_{+,-}\sum_{k=0}^{ N_{*}} F_{\pm}^k\lesssim
  \sum_{+,-}\sum_{k=0}^{ N_{*}}E_{\pm}^k(0)+\big(C_1\big)^3\varepsilon^3.
\end{equation}

\medskip

The last step is to complete the total energy estimates. Putting the estimates \eqref{eq:LOWEE} and \eqref{eq:HIGHEE} together, we can infer from the initial energy \eqref{eq:initial energy} that 
\begin{equation*}
	\sum_{+,-}\bigg(E_{\pm}+F_{\pm}+\sum_{k=0}^{N_{*}}E_{\pm}^k+\sum_{k=0}^{N_{*}}F_{\pm}^k\bigg)
	\leqslant
	C_2\varepsilon^2+C_2\big(C_1\big)^3\varepsilon^3,
\end{equation*}
where $C_2$ is a universal constant.  
We then take $C_1=(2 C_2)^{\frac{1}{2}}$, $\varepsilon_0=\dfrac{1}{(2C_2)^{\frac{3}{2}}}$. Thus, for all $\varepsilon< \varepsilon_0$, the above inequality implies
\begin{equation*}
	\sum_{+,-}\bigg(E_{\pm}+F_{\pm}+\sum_{k=0}^{N_{*}}E_{\pm}^k+\sum_{k=0}^{N_{*}}F_{\pm}^k\bigg)\leqslant \big(C_1\big)^2\varepsilon^2,
\end{equation*}
which is the improved energy estimate \eqref{Bootstrap on energy improved}. 

\bigskip

Up to now, we have derived \eqref{Bootstrap on fluctuation improved}-\eqref{Bootstrap on energy improved} and closed the bootstrap argument. As a result, we can reverse and shift the time to complete the proof of Theorem \ref{thm:global existence}, and the constants $C_0$ and $C_1$ from now on can be treated as universal constants.

\smallskip

\begin{remark}
All the above estimates still hold under \eqref{Bootstrap on fluctuation improved}-\eqref{Bootstrap on energy improved}. This is because we have derived these estimates up to universal constants by using the notation $\lesssim$. Moreover, the exactly numerical constants here are irrelevant to the rest proof in this paper.  
\end{remark}
\smallskip
\begin{remark}
With the global-in-time solution $\big(z_+(t,x),z_-(t,x)\big)$ constructed in Theorem \ref{thm:global existence}, we can extend the time region from
$[0, t^*]$ to $[0,+\infty]$ or $[0,-\infty]$ in all the above estimates from now on. 
\end{remark}

\smallskip

\section{Passage to the scattering fields of Alfv\'en waves}\label{sec:proof}

We are now in a position to construct scattering fields for Alfv\'en waves and provide proofs for Theorem \ref{thm:existence}, Theorem \ref{thm:weightSobolev} and Theorem \ref{thm:deviation}. 
By the symmetry of time, it suffices to consider the future scattering fields $z_+(+\infty;x_1^-,x_2^-,u_-)$ on $\mathcal{F}_+$ and $z_-(+\infty;x_1^+,x_2^+,u_+)$ on $\mathcal{F}_-$. 

\medskip

\subsection{Construction of the scattering fields at infinities}

Given a point $(x_1^-,x_2^-,u_-) \in \mathcal{F}_+$ and
a point $(x_1^+,x_2^+,u_+) \in \mathcal{F}_-$, for the solution $\big(z_+(t,x),z_-(t,x)\big)$ constructed in the previous section, we  show that the future scattering fields $\big(z_+(+\infty;x_1^-,x_2^-,u_-), z_-(+\infty;x_1^+,x_2^+,u_+)\big)$ given by 
\begin{equation}\label{eq:def scattering}
z_\pm(+\infty;x_1^\mp,x_2^\mp,u_\mp):=z_\pm(0,x_1^\mp,x_2^\mp,u_\mp)-\int_0^{+\infty} \big(\nabla p\cdot\sqrt{1+|z_\mp\mp B_0|^2}\big)(\tau,x_1^\mp,x_2^\mp,u_\mp)d\tau
\end{equation}
are well-defined. 

\smallskip

In fact, according to Lemma \ref{lemma:pressure derivative bound}, we have 
\begin{equation*}
	|\nabla p(\tau,x_1^\mp,x_2^\mp,u_\mp)|\lesssim \frac{\varepsilon^2}{(1+\tau)^\omega}.
\end{equation*}
By \eqref{Bootstrap on fluctuation}, for sufficiently small $\varepsilon$, there exists some constant $C$ such that  $|z_\mp|\leqslant C\varepsilon$. It follows that
\[\sqrt{2}-C\varepsilon\leqslant\sqrt{1+|z_\mp\mp B_0|^2}\leqslant\sqrt{2}+C\varepsilon.\]
We remark here that the constant $C$ appeared in the last line might be taken as different value from above. 
In other words, there holds 
\begin{equation}\label{eq:measure}
	\sqrt{1+|z_\mp\mp B_0|^2}(\tau,x_1^\mp,x_2^\mp,u_\mp)=\sqrt{2}+O(\varepsilon).
\end{equation}
Hence we obtain 
\begin{equation*}
	|\big(
	\nabla p\cdot\sqrt{1+|z_\mp\mp B_0|^2}\big)(\tau,x_1^\mp,x_2^\mp,u_\mp)|\lesssim\frac{\varepsilon^2}{(1+\tau)^\omega}\in L_\tau^1(\mathbb{R}).
\end{equation*}
Together with the Lebesgue dominated convergence theorem, it implies that the integrals in the definition of $z_+(+\infty;x_1^\mp,x_2^\mp,u_\mp)$ in \eqref{eq:def scattering} converge.  Therefore, the scattering fields $z_+(+\infty;x_1^-,x_2^-,u_-)$ on $\mathcal{F}_+$ and $z_-(+\infty;x_1^+,x_2^+,u_+)$ on $\mathcal{F}_-$ are point-wisely well-defined by \eqref{eq:def scattering}. 
The proof of Theorem \ref{thm:existence} is completed.

\medskip

\subsection{Construction of the weighted energy spaces at infinities}

To begin with, we define on $\mathcal{F}_\pm$ their corresponding weighted Sobolev norms and weighted Sobolev spaces as follows: for any vector field $f$ on $\mathcal{F}_\pm$, setting the weighted measure 
$\langle u_\mp \rangle^{2\omega} d\mu_\mp = \langle u_\mp \rangle^{2\omega} dx_1 dx_2 d u_\mp$
on $\mathcal{F}_\pm$ leads us to the definition of the $L^2$-type space $L^2(\mathcal{F}_\pm,\langle u_\mp \rangle^\omega d\mu_\mp)$.  

\smallskip

Let us also recall that we can use the following five coordinate systems on $\mathbb{R}\times \mathbb{R}^3$: the Cartesian coordinates $(t,x_1,x_2,x_3)$, two characteristic coordinates $(t,x_1^-,x_2^-,u_-)$ and $(t,x_1^+,x_2^+,u_+)$, the double characteristic coordinates $(x_1^-,x_2^-,u_-,u_+)$ and $(x_1^+,x_2^+,u_+,u_-)$. 
We shall use this observation repeatedly in the following proof without further comment.

\medskip

We are now turning to the proof of Theorem \ref{thm:weightSobolev}, namely the key property of the scattering fields that they live in the Sobolev spaces based on the above measures. It suffices to show the following proposition. 

\begin{proposition}\label{prop:Sobolev}For all multi-indices $\beta$ with $0\leqslant\left|\beta\right|\leqslant N_*+1$, we have
\[\nabla^\beta z_{\pm}(+\infty;x_1^\mp,x_2^\mp,u_\mp)\in L^{2}(\mathcal{F}_\pm,\langle u_\mp \rangle^{2\omega} d\mu_{\mp}).\]
\end{proposition}

\begin{proof}
The proof is divided into two steps. The first step deals with the case where $|\beta|=0$. The second step deals with the cases with $1\leqslant|\beta|\leqslant N_*+1$. 

\medskip
\noindent
\textbf{Step 1: } We show that $z_\pm(+\infty;x_1^\mp,x_2^\mp,u_\mp)\in L^2(\mathcal{F}_\pm,\langle u_- \rangle^{2\omega} d\mu_\mp)$.

\medskip
\noindent
By definition, we have
\begin{align*}
	& \ \ \ \int_{\mathcal{F}_+}|z_\pm(+\infty;x_1^\mp,x_2^\mp,u_\mp)|^2\langle u_\mp\rangle^{2\omega}d\mu_\mp\\
	&=\int_{\mathbb{R}^3}\Big|z_\pm(0,x_1^\mp,x_2^\mp,u_\mp)-\int_0^{+\infty}\big(\nabla p\cdot\sqrt{1+|z_\mp\mp B_0|^2}\big)(\tau,x_1^\mp,x_2^\mp,u_\mp)d\tau\Big|^2 \langle u_\mp\rangle^{2\omega}dx_1^\mp dx_2^\mp du_\mp\\
	&\lesssim\underbrace{\int_{\mathbb{R}^3}|z_\pm(0,x_1^\mp,x_2^\mp,u_\mp)|^2\langle u_\mp\rangle^{2\omega}dx_1^\mp dx_2^\mp du_\mp}_{\mathbf{P_1}}\\
	&\ \ \ \ 
	+\underbrace{\int_{\mathbb{R}^3} \Big|\int_{0}^{+\infty}\big(\nabla p\cdot\sqrt{1+|z_\mp\mp B_0|^2}\big)(\tau,x_1^\mp,x_2^\mp,u_\mp)d\tau\Big|^2\langle u_\mp\rangle^{2\omega}dx_1^\mp dx_2^\mp du_\mp}_{\mathbf{P_2}}.
\end{align*}
\smallskip
For $\mathbf{P_1}$, according to the initial energy \eqref{eq:initial energy}, we have
\[\mathbf{P_1}= \int_{\mathbb{R}^{3}}\langle u_{\mp}\rangle^{2\omega}|z_{\pm}(0,x)|^2dx=E_\pm(0)\leqslant \varepsilon^2.\]
\smallskip
For  $\mathbf{P_2}$, using H\"older inequality and Lemma \ref{lemma:pressureL2L2-all}, we can derive
\begin{align*}
	\mathbf{P_2}
	&\leqslant\int_{\mathbb{R}^3}\Big|\int_{\mathbb{R}}\big|\big(\nabla p\cdot\sqrt{1+|z_\mp\mp B_0|^2}\big)(x_1^\mp,x_2^\mp,u_\mp,u_\pm)\big|du_\pm\Big|^2 \langle u_\mp\rangle^{2\omega}dx_1^\mp dx_2^\mp du_\mp\\
	&\stackrel{\text{H\"older}}{\lesssim}\!\!\!\int_{\mathbb{R}^3} \Big(\int_{\mathbb{R}}\frac{1}{\langle u_{\pm}\rangle^{\omega}}du_\pm\Big)\Big(\int_{\mathbb{R}}\langle u_+\rangle^\omega\big|\big(\nabla p\cdot\sqrt{1+|z_\mp\mp B_0|^2}\big)(x_1^\mp,x_2^\mp,u_\mp,u_\pm)\big|^2du_\pm\Big)\langle u_\mp\rangle^{2\omega}dx_1^\mp dx_2^\mp du_\mp\\
    &\lesssim\int_{\mathbb{R}\times\mathbb{R}^3}\langle u_{\mp}\rangle^{2\omega}\langle u_\pm\rangle^\omega\big|\big(\nabla p\cdot\sqrt{1+|z_\mp\mp B_0|^2}\big)(x_1^\mp,x_2^\mp,u_\mp,u_\pm)\big|^2dx_1^\mp dx_2^\mp du_\mp du_\pm\\
    &\lesssim\int_{\mathbb{R}\times \mathbb{R}^3}\langle u_\mp\rangle^{2\omega}\langle u_\pm\rangle^\omega\big|\nabla p\cdot\sqrt{1+|z_\mp\mp B_0|^2}\big|^2 dx_1dx_2dx_3d\tau\\
    &\stackrel{\eqref{eq:measure}}{\lesssim}\int_{\mathbb{R}\times \mathbb{R}^3}\langle u_\mp\rangle^{2\omega}\langle u_\pm\rangle^\omega|\nabla p|^2 dx_1dx_2dx_3d\tau\\
    &\stackrel{\eqref{eq:pressureL2L2-all}}{\lesssim}\varepsilon^4.
\end{align*}
Putting the estimates on $\mathbf{P_1}$ and $\mathbf{P_2}$ together, we have proved that $z_\pm(+\infty;x_1^\mp,x_2^\mp,u_\mp)\in L^2(\mathcal{F}_\pm,\langle u_\mp \rangle^\omega d\mu_\mp)$. 

\bigskip

\noindent
\textbf{Step 2: } We show that  
$(\nabla^{\beta}z_\pm)(+\infty;x_1^\mp,x_2^\mp,u_\mp)\in L^2(\mathcal{F}_\pm,\langle u_\mp \rangle^{2\omega} d\mu_\mp)$ for all multi-indices $\beta$ with $1\leqslant|\beta|\leqslant N_*+1$.

By definition, we have
\begin{align*}
&\ \ \ \
\int_{\mathcal{F}_\pm} \big|\nabla^{\beta}z_{\pm}(+\infty;x_1^\mp,x_2^\mp,u_\mp)\big|^2\langle u_{\mp}\rangle^{2\omega}d\mu_\mp\\
&=\int_{\mathcal{F}_\pm}\Big|\nabla^{\beta}z_\pm(0,x_1^\mp,x_2^\mp,u_\mp)-\nabla^{\beta}\Big(\int_0^{+\infty}\big(\nabla p\cdot\sqrt{1+|z_\mp\mp B_0|^2}\big)(\tau,x_1^\mp,x_2^\mp,u_\mp)d\tau\Big) \Big|^2\langle u_{\mp}\rangle^{2\omega}d\mu_\mp\\
&\lesssim \underbrace{\int_{\mathbb{R}^3} \big|\nabla^{\beta}z_\pm(0,x_1^\mp,x_2^\mp,u_\mp)\big|^2\langle u_{\mp}\rangle^{2\omega}dx_1^\mp dx_2^\mp du_\mp}_{\mathbf{Q_1}}\\
& \ \ \ +\underbrace{\int_{\mathbb{R}^3} \Big|\nabla^{\beta}\Big(\int_0^{+\infty}\big(\nabla p\cdot\sqrt{1+|z_\mp\mp B_0|^2}\big)(\tau,x_1^\mp,x_2^\mp,u_\mp)d\tau\Big)\Big|^2\langle u_{\mp}\rangle^{2\omega}dx_1^\mp dx_2^\mp du_\mp}_{\mathbf{Q_2}}.
\end{align*}
\smallskip
On the one hand, we can bound $\mathbf{Q_1}$ by the initial energy \eqref{eq:initial energy}:
\begin{align*}
\mathbf{Q_1}&\lesssim \int_{\mathbb{R}^{3}}\langle u_{\mp}\rangle^{2\omega}\big|\nabla^{\beta} z_{\pm}(0,x_1,x_2,x_3)\big|^2dx_1 dx_2 dx_3\\
&\stackrel{\text{Remark \ref{rmk:divcurl}}}{\lesssim} \int_{\mathbb{R}^{3}}\langle u_{\mp}\rangle^{2\omega}\big| z_{\pm}(0,x_1,x_2,x_3)\big|^2dx_1 dx_2 dx_3+\sum_{l=0}^{|\beta|-1}\int_{\mathbb{R}^{3}}\langle u_{\mp}\rangle^{2\omega}\big| j^{(l)}_{\pm}(0,x_1,x_2,x_3)\big|^2dx_1 dx_2 dx_3\\
&\leqslant E_\pm(0)+\sum_{l=0}^{|\beta|-1}E^{(l)}_\pm(0)\leqslant \varepsilon^2.\stepcounter{equation}\tag{\theequation}\label{eq:Q1}
\end{align*}
\smallskip
On the other hand, we are able to bound $\mathbf{Q_2}$ with the following claim in hand:
\begin{claim}\label{lemma:claim}
	For all multi-indices $\beta$ with $1\leqslant\left|\beta\right|\leqslant N_*+1$, we have the following formal expression:
	\begin{equation}\label{eq:Lambda}
		\nabla^{\beta}\Big(\int_0^{+\infty}\big(\nabla p\cdot\sqrt{1+|z_\mp\mp B_0|^2}\big)(\tau,x_1^\mp,x_2^\mp,u_\mp)d\tau\Big)
		\stackrel{L^2(\mathcal{F}_\pm,\langle u_{\mp}\rangle^{2\omega} d\mu_\mp)}{=\joinrel=\joinrel=\joinrel=}
		\int_0^{+\infty}\mathbf{R}_{\mp}^{(\beta)}(\tau,x_1^\mp,x_2^\mp,u_\mp)d\tau,
	\end{equation}
	where
	\begin{align*}
		\mathbf{R}_{\mp}^{(\beta)}(\tau,x_1^\mp,x_2^\mp,u_\mp)&:=\sum_{l=0}^{\beta-1}C_{\beta-1}^l\Big(\sum_{k=0}^{\beta-l}C_{\beta-l}^k\nabla^{k+1} p\cdot\nabla^{\beta-l-k} \sqrt{1+|z_\mp\mp B_0|^2}\Big)(\tau,x_1^\mp,x_2^\mp,u_\mp)\\
		&\ \ \ \ \ \ \times\nabla^l\Big(\mathrm{I}+ \int_0^{+\infty}\nabla z_\mp(\tau,x_1^\mp,x_2^\mp,u_\mp)d\tau\Big).\stepcounter{equation}\tag{\theequation}\label{eq:Lambda-def}
	\end{align*}
We remark here and in the sequel that on the left hand side of the equation such as \eqref{eq:Lambda}, $\nabla$ is defined with respect to the coordinate system $(x_1^\mp,x_2^\mp,u_\mp)$ on $\mathcal{F}_\pm$, while on the right hand side of the equation such as \eqref{eq:Lambda}, $\nabla$ is defined with respect to the coordinate system $(t, x_1^\mp,x_2^\mp,u_\mp)$ on $\mathbb{R}\times \mathbb{R}^3$ (except for the last term in \eqref{eq:Lambda-def}, where $\nabla$ is defined with respect to the coordinate system $(x_1^\mp,x_2^\mp,u_\mp)$ on $\mathcal{F}_\pm$ if it exists, i.e. $l\neq 0$).
\end{claim}

\smallskip
\noindent The proof of Claim \ref{lemma:claim} may be of independent interest and importance to the whole proof and will be given with a full discussion in the next subsection. 
For clarity in the presentation, we assume the validity of this claim first and proceed to estimate $\mathbf{Q_2}$.

\bigskip

In a spiritual sense, the initial idea of the estimate on $\mathbf{Q_2}$ is similar to that used for $\mathbf{P_2}$ above. 
Considering the technical difficulties in this part, we prefer to retain the details of our proof as follows:
\begin{align*}
	\mathbf{Q_2} &\stackrel{\eqref{eq:Lambda}}{=}\int_{\mathbb{R}^3}\bigg|\int_0^{+\infty}\mathbf{R}_{\mp}^{(\beta)}(\tau,x_1^\mp,x_2^\mp,u_\mp) d\tau\bigg|^2\langle u_\mp\rangle^{2\omega}dx_1^\mp dx_2^\mp du_\mp\\
	&\stackrel{\eqref{eq:Lambda-def}}{\leqslant}\int_{\mathbb{R}^3}\bigg|\int_{\mathbb{R}}\sum_{l=0}^{\beta-1}C_{\beta-1}^l\Big(\sum_{k=0}^{\beta-l}C_{\beta-l}^k\nabla^{k+1} p\cdot\nabla^{\beta-l-k} \sqrt{1+|z_\mp\mp B_0|^2}\Big)(x_1^\mp,x_2^\mp,u_\mp,u_\pm)\\
	&\ \ \ \ \ \ \ \ \ \ \ \ \ \ \ \ \ \ \ \ \ \ \ \ \ \ \ \ \ \ \ \ \  \times\nabla^l\Big(\mathrm{I}+ \int_0^{+\infty}\nabla z_\mp(x_1^\mp,x_2^\mp,u_\mp,u_\pm)du_\pm\Big)du_\pm\bigg|^2\langle u_\mp\rangle^{2\omega}dx_1^\mp dx_2^\mp du_\mp\\
	&\stackrel{\text{H\"older}}{\lesssim}\!\!\!\int_{\mathbb{R}^3} \Big(\int_{\mathbb{R}}\frac{1}{\langle u_\pm\rangle^{\omega}}du_\pm\Big)\Big(\int_{\mathbb{R}}\langle u_\pm\rangle^\omega\Big|\sum_{l=0}^{\beta-1}C_{\beta-1}^l\Big(\sum_{k=0}^{\beta-l}C_{\beta-l}^k\nabla^{k+1} p\cdot\nabla^{\beta-l-k} \sqrt{1+|z_\mp\mp B_0|^2}\Big)(x_1^\mp,x_2^\mp,u_\mp,u_\pm)\\
	&\ \ \ \ \ \ \ \ \ \ \ \ \ \ \ \ \ \ \ \ \ \ \ \ \ \ \ \ \ \ \ \ \  \times \nabla^l\Big(\mathrm{I}+ \int_{\mathbb{R}}\nabla z_\mp(x_1^\mp,x_2^\mp,u_\mp,u_\pm)du_\pm\Big)\Big|^2 du_\pm\Big) \langle u_\mp\rangle^{2\omega}dx_1^\mp dx_2^\mp du_\mp\\
	&\lesssim
	\int_{\mathbb{R}	\times\mathbb{R}^3}\langle u_\mp\rangle^{2\omega}\langle u_\pm\rangle^\omega\Big|\sum_{l=0}^{\beta-1}C_{\beta-1}^l\Big(\sum_{k=0}^{\beta-l}C_{\beta-l}^k\nabla^{k+1} p\cdot\nabla^{\beta-l-k} \sqrt{1+|z_\mp\mp B_0|^2}\Big)(x_1^\mp,x_2^\mp,u_\mp,u_\pm)\\
	&\ \ \ \ \ \ \ \ \ \ \ \ \ \ \ \ \ \ \ \ \ \ \ \ \ \ \ \ \ \ \ \ \ \times  \nabla^l\Big(\mathrm{I}+\int_{\mathbb{R}}\nabla z_\mp(x_1^\mp,x_2^\mp,u_\mp,u_\pm)du_\pm\Big)\Big|^2 dx_1^\mp dx_2^\mp du_\mp du_\pm\\
	&\lesssim
	\int_{\mathbb{R}\times\mathbb{R}^3}\langle u_\mp\rangle^{2\omega}\langle u_\pm\rangle^\omega\Big|\sum_{l=0}^{\beta-1}C_{\beta-1}^l\Big(\sum_{k=0}^{\beta-l}C_{\beta-l}^k\nabla^{k+1} p\cdot\nabla^{\beta-l-k} \sqrt{1+|z_\mp\mp B_0|^2}\Big)(\tau,x_1,x_2,x_3)\\
	&\ \ \ \ \ \ \ \ \ \ \ \ \ \ \ \ \ \ \ \ \ \ \ \ \ \ \ \ \ \ \ \ \ \times  \nabla^l\Big(\mathrm{I}+ \int_{\mathbb{R}}\nabla z_\mp(\tau,x_1,x_2,x_3)d\tau\Big)\Big|^2 dx_1dx_2dx_3d\tau.\stepcounter{equation}\tag{\theequation}\label{eq:Q2-intro}
\end{align*}
We observe here in passing that several constants in \eqref{eq:Q2-intro} can be omitted. Therefore it follows that 
\begin{align*}
	\mathbf{Q_2}
	&\lesssim
	\int_{\mathbb{R}\times\mathbb{R}^3}\langle u_\mp\rangle^{2\omega}\langle u_\pm\rangle^\omega\Big|\sum_{l=0}^{\beta-1}\Big(\sum_{k=0}^{\beta-l}\nabla^{k+1}p\cdot\nabla^{\beta-l-k}z_\mp\Big)(\tau,x_1,x_2,x_3)\cdot \nabla^l\Big(\mathrm{I}+ \int_{\mathbb{R}}\nabla z_\mp(\tau,x_1,x_2,x_3)d\tau\Big)\Big|^2\!\! dxd\tau\\
	&=
	\int_{\mathbb{R}	\times\mathbb{R}^3}\langle u_\mp\rangle^{2\omega}\langle u_\pm\rangle^\omega\Big|\Big(\sum_{k=0}^{\beta}\nabla^{k+1} p\cdot\nabla^{\beta-k}z_\mp\Big)(\tau,x_1,x_2,x_3)\cdot \Big(\mathrm{I}+ \int_{\mathbb{R}}\nabla z_\mp(\tau,x_1,x_2,x_3)d\tau\Big)\Big|^2\! dxd\tau\\
	&\ \ \ \ + \int_{\mathbb{R}\times\mathbb{R}^3}\langle u_\mp\rangle^{2\omega}\langle u_\pm\rangle^\omega\Big|\sum_{l=1}^{\beta-1}\Big(\sum_{k=0}^{\beta-l}\nabla^{k+1}p\cdot\nabla^{\beta-l-k}z_\mp\Big)(\tau,x_1,x_2,x_3)\!\cdot  \Big(\int_{\mathbb{R}}\nabla^{l+1} z_\mp(\tau,x_1,x_2,x_3)d\tau\Big)\Big|^2\! dxd\tau\\
	&= \int_{\mathbb{R}	\times\mathbb{R}^3}\langle u_\mp\rangle^{2\omega}\langle u_\pm\rangle^\omega\Big|\Big(\sum_{k=0}^{\beta}\nabla^{k+1} p\cdot\nabla^{\beta-k}z_\mp\Big)(\tau,x_1,x_2,x_3)\Big|^2\!dxd\tau\\
	&\ \ \ \  +\int_{\mathbb{R}\times\mathbb{R}^3}\langle u_\mp\rangle^{2\omega}\langle u_\pm\rangle^\omega\Big|\Big(\sum_{k=0}^{\beta}\nabla^{k+1} p\cdot\nabla^{\beta-k}z_\mp\Big)(\tau,x_1,x_2,x_3)\cdot \Big( \int_{\mathbb{R}}\nabla z_\mp(\tau,x_1,x_2,x_3)d\tau\Big)\Big|^2\! dxd\tau\\
	&\ \ \ \ + \int_{\mathbb{R}\times\mathbb{R}^3}\langle u_\mp\rangle^{2\omega}\langle u_\pm\rangle^\omega\Big|\sum_{l=1}^{\beta-1}\Big(\sum_{k=0}^{\beta-l}\nabla^{k+1}p\cdot\nabla^{\beta-l-k}z_\mp\Big)(\tau,x_1,x_2,x_3)\!\cdot \Big(\int_{\mathbb{R}}\nabla^{l+1} z_\mp(\tau,x_1,x_2,x_3)d\tau\!\Big)\Big|^2\! dxd\tau\\
	&= \int_{\mathbb{R}	\times\mathbb{R}^3}\langle u_\mp\rangle^{2\omega}\langle u_\pm\rangle^\omega\Big|\Big(\sum_{k=0}^{\beta}\nabla^{k+1}p\cdot\nabla^{\beta-k}z_\mp\Big)(\tau,x_1,x_2,x_3)\Big|^2\! dxd\tau\\
	&\ \ \ \ + \int_{\mathbb{R}\times\mathbb{R}^3}\langle u_\mp\rangle^{2\omega}\langle u_\pm\rangle^\omega\Big|\sum_{l=0}^{\beta-1}\Big(\sum_{k=0}^{\beta-l}\nabla^{k+1}p\cdot\nabla^{\beta-l-k}z_\mp\Big)(\tau,x_1,x_2,x_3)\cdot  \Big(\int_{\mathbb{R}}\nabla^{l+1} z_\mp(\tau,x_1,x_2,x_3)d\tau\!\Big)\Big|^2\! dxd\tau\\
	&\lesssim\underbrace{\sum_{k=0}^{\beta}\int_{\mathbb{R}\times\mathbb{R}^3}\langle u_\mp\rangle^{2\omega}\langle u_\pm\rangle^\omega|\nabla^{k+1} p|^2\cdot|\nabla^{\beta-k} z_\mp|^2dxd\tau}_{\mathbf{Q_{21}}}\\
	&\ \ \ \  +\underbrace{\sum_{l=0}^{\beta-1}\sum_{k=0}^{\beta-l}\int_{\mathbb{R}	\times\mathbb{R}^3}\langle u_\mp\rangle^{2\omega}\langle u_\pm\rangle^\omega|\nabla^{k+1} p|^2\cdot|\nabla^{\beta-l-k} z_\mp|^2\cdot\Big(\int_{\mathbb{R}}\nabla^{l+1} z_\mp(\tau,x_1,x_2,x_3)d\tau\Big)^2 dxd\tau}_{\mathbf{Q_{22}}}.\stepcounter{equation}\tag{\theequation}\label{eq:Q2-def}
\end{align*}	
Hence the proof of this step can be completed by giving the bounds on  $\mathbf{Q_{21}}$ and $\mathbf{Q_{22}}$ respectively.

\medskip

For $\mathbf{Q_{21}}$, according to the order of derivatives, we distinguish the following two cases:\\
\smallskip
\textbf{Case 1}: $1\leqslant |\beta|\leqslant N_*-1$. By virtue of $\langle u_\pm\rangle^{2\omega}\geqslant 1$, Lemma \ref{lemma:Sobolev} and Lemma \ref{lemma:pressureL2L2-all}, we can  derive 
\begin{align*}
	\mathbf{Q_{21}}
	&\lesssim\sum_{k=0}^{\beta}\int_{\mathbb{R}	\times\mathbb{R}^3}\langle u_\mp\rangle^{2\omega}\langle u_\pm\rangle^\omega|\nabla^{k+1} p|^2\cdot \langle u_\pm\rangle^{2\omega}|\nabla^{\beta-k} z_\mp|^2dxd\tau\\
	&\lesssim\sum_{k=0}^{\beta}\Big\|\langle u_\pm\rangle^{\omega}\nabla^{\beta-k} z_\mp\Big\|^2_{L^\infty_\tau L^\infty_x}\int_{\mathbb{R}\times\mathbb{R}^3}\langle u_\mp\rangle^{2\omega}\langle u_\pm\rangle^\omega|\nabla^{k+1} p|^2dxd\tau\\
	&\stackrel{\eqref{eq:Sobolev},\eqref{eq:pressureL2L2-all}}{\lesssim} \varepsilon^6.\stepcounter{equation}\tag{\theequation}\label{eq:Q21case1}
\end{align*}
\smallskip
\textbf{Case 2}: $|\beta|=N_*\text{ or }N_*+1$. We notice that 
\begin{align*}
	\mathbf{Q_{21}}
	&=\sum_{k=0}^{N_*-1}\!\int_{\mathbb{R}\times\mathbb{R}^3}\!\!\langle u_\mp\rangle^{2\omega}\langle u_\pm\rangle^\omega|\nabla^{k+1} p|^2\cdot|\nabla^{\beta-k} z_\mp|^2 dxd\tau +\!\!\sum_{k=N_*}^{\beta}\!\int_{\mathbb{R}	\times\mathbb{R}^3}\!\!\langle u_\mp\rangle^{2\omega}\langle u_\pm\rangle^\omega|\nabla^{k+1} p|^2\cdot|\nabla^{\beta-k} z_\mp|^2dxd\tau\\
	&\stackrel{\eqref{eq:Q21case1}}{\lesssim}\varepsilon^6+\underbrace{\sum_{k=N_*}^{\beta}\int_{\mathbb{R}	\times\mathbb{R}^3}\langle u_\mp\rangle^{2\omega}\langle u_\pm\rangle^\omega|\nabla^{k+1} p|^2\cdot|\nabla^{\beta-k} z_\mp|^2dxd\tau}_{\mathbf{Q_{21}'}}.
\end{align*}
An entirely analogous argument as \eqref{eq:Q21case1} shows that 
\begin{align*}
	\mathbf{Q_{21}'}
	&\lesssim\sum_{k=N_*}^{\beta}\int_{\mathbb{R}	\times\mathbb{R}^3}\langle u_\mp\rangle^{2\omega}\langle u_\pm\rangle^\omega|\nabla^{k+1} p|^2\cdot \langle u_\pm\rangle^{2\omega}|\nabla^{\beta-k} z_\mp|^2dxd\tau\\
	&\lesssim\sum_{k=N_*}^{\beta}\Big\|\langle u_\pm\rangle^{\omega}\nabla^{\beta-k} z_\mp\Big\|^2_{L^\infty_\tau L^\infty_x}\int_{\mathbb{R}\times\mathbb{R}^3}\langle u_\mp\rangle^{2\omega}\langle u_\pm\rangle^\omega|\nabla^{k+1} p|^2dxd\tau\\
	&\stackrel{\eqref{eq:Sobolev},\eqref{eq:pressureL2L2-all}}{\lesssim} \varepsilon^6,
\end{align*}
and hence $\mathbf{Q_{21}}$ in this case can also be bounded by $\varepsilon^6$ up to a universal constant. In this way, the resulting estimate from the previous two cases is
\begin{equation}\label{eq:Q21}	
\mathbf{Q_{21}}\lesssim\varepsilon^6.
\end{equation}

\medskip

For $\mathbf{Q_{22}}$, as argued above, we also need to consider the following two cases:\\
\smallskip
\textbf{Case 1}: $1\leqslant |\beta|\leqslant N_*-1$. Using $\langle u_\pm\rangle^{2\omega}\geqslant 1$, Lemma \ref{lemma:Sobolev} and Lemma \ref{lemma:pressureL2L2-all} again, we infer that 
\begin{align*}
	\mathbf{Q_{22}}\!
	&\lesssim\sum_{l=0}^{\beta-1}\sum_{k=0}^{\beta-l}\int_{\mathbb{R}	\times\mathbb{R}^3}\langle u_\mp\rangle^{2\omega}\langle u_\pm\rangle^\omega|\nabla^{k+1} p|^2\cdot \langle u_\pm\rangle^{2\omega}|\nabla^{\beta-l-k} z_\mp|^2\cdot\langle u_\pm\rangle^{2\omega} \Big(\int_{\mathbb{R}}\nabla^{l+1} z_\mp(\tau,x_1,x_2,x_3)d\tau\Big)^2dxd\tau\\
	&\lesssim\sum_{l=0}^{\beta-1}\sum_{k=0}^{\beta-l}\int_{\mathbb{R}}\!\bigg(\Big\|\langle u_\pm\rangle^{\omega}\nabla^{\beta-l-k} z_\mp\Big\|^2_{L^\infty_{x}}\Big\|\langle u_\pm\rangle^{\omega}\!\!\int_{\mathbb{R}}\!\nabla^{l+1} z_\mp(\tau,x_1,x_2,x_3)d\tau\Big\|^2_{L^\infty_{x}}\!\int_{\mathbb{R}^3}\!\langle u_\mp\rangle^{2\omega}\langle u_\pm\rangle^\omega|\nabla^{k+1} p|^2dx\bigg)d\tau\\
	&\lesssim\sum_{l=0}^{\beta-1}\sum_{k=0}^{\beta-l}\!\Big\|\langle u_\pm\rangle^{\omega}\nabla^{\beta-l-k} z_\mp\Big\|^2_{L^\infty_\tau L^\infty_x}\!\Big\|\langle u_\pm\rangle^{\omega}\!\!\! \int_{\mathbb{R}}\!\!\nabla^{l+1} z_\mp(\tau,x_1,x_2,x_3)d\tau\Big\|^2_{L^\infty_\tau L^\infty_{x}}\!\!\int_{\mathbb{R}\times\mathbb{R}^3}\!\!\!\langle u_\mp\rangle^{2\omega}\langle u_\pm\rangle^\omega|\nabla^{k+1} p|^2dxd\tau\\
	&\stackrel{\eqref{eq:Sobolev},\eqref{eq:pressureL2L2-all}}{\lesssim} \varepsilon^8.\stepcounter{equation}\tag{\theequation}\label{eq:Q22case1}
\end{align*}
\smallskip
\textbf{Case 2}: $|\beta|=N_*\text{ or }N_*+1$. The same observation as before leaves us with 
\begin{align*}
	\mathbf{Q_{22}}
	&=\sum_{l=0}^{N_*-2}\sum_{k=0}^{N_*-1-l}\int_{\mathbb{R}	\times\mathbb{R}^3}\langle u_\mp\rangle^{2\omega}\langle u_\pm\rangle^\omega|\nabla^{k+1} p|^2\cdot|\nabla^{\beta-l-k} z_\mp|^2\cdot\Big(\int_{\mathbb{R}}\nabla^{l+1} z_\mp(\tau,x_1,x_2,x_3)d\tau\Big)^2 dxd\tau\\
	&\ \ \ \ +\sum_{l=N_*-1}^{\beta-1}\sum_{k=0}^{\beta-l}\int_{\mathbb{R}	\times\mathbb{R}^3}\langle u_\mp\rangle^{2\omega}\langle u_\pm\rangle^\omega|\nabla^{k+1} p|^2\cdot|\nabla^{\beta-l-k} z_\mp|^2\cdot\Big(\int_{\mathbb{R}}\nabla^{l+1} z_\mp(\tau,x_1,x_2,x_3)d\tau\Big)^2 dxd\tau\\
	&\stackrel{\eqref{eq:Q22case1}}{\lesssim}\varepsilon^8+\underbrace{\sum_{l=N_*-1}^{\beta-1}\sum_{k=0}^{\beta-l}\int_{\mathbb{R}	\times\mathbb{R}^3}\langle u_\mp\rangle^{2\omega}\langle u_\pm\rangle^\omega|\nabla^{k+1} p|^2\cdot|\nabla^{\beta-l-k} z_\mp|^2\cdot\Big(\int_{\mathbb{R}}\nabla^{l+1} z_\mp(\tau,x_1,x_2,x_3)d\tau\Big)^2 dxd\tau}_{\mathbf{Q_{22}'}}.
\end{align*}
To deal with the order of derivatives, we can carry out the following estimate on $\mathbf{Q_{22}'}$ by means of the energy ansatz \eqref{Bootstrap on energy} and Lemma \ref{lemma:Sobolev} as well as Lemma \ref{lemma:pressureL2Linfty-all}:
\begin{align*}
	\mathbf{Q_{22}'}\!
	&\lesssim\!\sum_{l=N_*-1}^{\beta-1}\sum_{k=0}^{\beta-l}\int_{\mathbb{R}\times\mathbb{R}^3}\!\langle u_\mp\rangle^{2\omega}\langle u_\pm\rangle^{\omega}|\nabla^{k+1} p|^2\cdot\langle u_\pm\rangle^{2\omega}|\nabla^{\beta-l-k} z_\mp|^2\cdot \langle u_\pm\rangle^{2\omega} \Big(\int_{\mathbb{R}}\!\nabla^{l+1} z_\mp(\tau,x_1,x_2,x_3)d\tau\!\Big)^2\!dxd\tau\\		&\lesssim\!\sum_{l=N_*-1}^{\beta-1}\sum_{k=0}^{\beta-l}\int_{\mathbb{R}}\!\bigg(\Big\|\langle u_\mp\rangle^{\omega}\langle u_\pm\rangle^{\frac{\omega}{2}}\!\nabla^{k+1}p\Big\|^2_{L^\infty_{x}}\!\Big\|\langle u_\pm\rangle^{\omega} \nabla^{\beta-l-k} z_\mp\Big\|^2_{L^\infty_{x}}\!\! \int_{\mathbb{R}^3}\!\!\langle u_\pm\rangle^{2\omega}\!\Big|\!\int_{\mathbb{R}}\!\!\nabla^{l+1} z_\mp(\tau,x_1,x_2,x_3)d\tau\!\Big|^2\!dx\!\bigg)\!d\tau\\
	&\lesssim\!\sum_{l=N_*-1}^{\beta-1}\sum_{k=0}^{\beta-l}\underbrace{\Big\|\langle u_\pm\rangle^{\omega} \nabla^{\beta-l-k} z_\mp\Big\|^2_{L^\infty_\tau L^\infty_{x}}}_{\stackrel{\eqref{eq:Sobolev}}{\lesssim}\varepsilon^2} \underbrace{\bigg\|\int_{\mathbb{R}^3}\langle u_\pm\rangle^{2\omega}\Big|\int_{\mathbb{R}}\nabla^{l+1} z_\mp(\tau,x_1,x_2,x_3)d\tau\Big|^2dx\bigg\|_{L^\infty_\tau}}_{\stackrel{\eqref{Bootstrap on energy}}{\lesssim}\varepsilon^2}\\&\ \ \ \ \ \ \ \ \ \ \ \ \ \ \ \ \ \  \times\underbrace{\int_{\mathbb{R}}\big\|\langle u_\mp\rangle^{\omega}\langle u_\pm\rangle^{\frac{\omega}{2}}\nabla^{k+1} p\big\|^2_{L^\infty_{x}}d\tau}_{\stackrel{\eqref{eq:pressureL2Linfty-all}}{\lesssim}\varepsilon^4}\\
	&\lesssim \varepsilon^8.
\end{align*}
Therefore $\mathbf{Q_{22}}$ is bounded by $\varepsilon^8$ up to a universal constant. To sum up, we always have 
\begin{equation}\label{eq:Q22}	\mathbf{Q_{22}}\lesssim\varepsilon^8.
\end{equation}

\medskip

Finally, adding \eqref{eq:Q21} and \eqref{eq:Q22} up gives the right hand side of \eqref{eq:Q2-def}, leaving 
\begin{equation}\label{eq:Q2}
	\mathbf{Q_2}\lesssim\varepsilon^6.
\end{equation}
Combining \eqref{eq:Q1} and \eqref{eq:Q2}, we can summarize that 
$(\nabla^{\beta}z_\pm)(+\infty;x_1^\mp,x_2^\mp,u_\mp)\in L^2(\mathcal{F}_\pm,\langle u_\mp \rangle^{2\omega} d\mu_\mp)$ for all $1\leqslant|\beta|\leqslant N_*+1$, 
which concludes the proof of \textbf{Step 2}. We have thus proved the proposition. 
\end{proof}

\medskip
\subsection{An appendix to Theorem \ref{thm:weightSobolev}: The proof of Claim \ref{lemma:claim}}

As a missing part in the construction of energy spaces, the proof of Claim \ref{lemma:claim} serves a useful purpose to ensure the commutations between integral and $\nabla^{N_*+1}$ derivatives. The overall strategy for this proof relies on a standard induction on $|\beta|$, which originates from the work \cite{Li-Yu} on commuting the integral and lower order derivatives, and is well adapted to this highest order (also optimal) case by exploiting the pressure estimates. 
As we will see, our method particularly combines both delicate techniques from real analysis and the above energy estimates.

\smallskip

To carry out the induction, we first need to establish \eqref{eq:Lambda} for $|\beta|=1$. We are now in a position to differentiate the integral in the definition \eqref{eq:def scattering} of $z_\pm(+\infty;x_1^\mp,x_2^\mp,u_\mp)$.
\begin{lemma}\label{lemma:claim-proof1}
	For all partial derivatives $D\in \big\{\partial_{x_1^\mp},\partial_{x_2^\mp}, \partial_{u_\mp}\big\}$, we have the following formal expression:
	\begin{align*}
	&\ \ \ \ D\Big(\int_0^{+\infty}\big(\nabla p\cdot\sqrt{1+|z_\mp\mp B_0|^2}\big)(\tau,x_1^\mp,x_2^\mp,u_\mp)d\tau\Big)\\
	&=\int_0^{+\infty}\Big(D(\nabla p)\cdot \sqrt{1+|z_\mp\mp B_0|^2} +\nabla p\cdot D \sqrt{1+|z_\mp\mp B_0|^2}\Big)(\tau,x_1^\mp,x_2^\mp,u_\mp)\\
	&\ \ \ \ \ \ \ \ \times\Big(1+ \int_0^{+\infty} D z_\mp(\tau,x_1^\mp,x_2^\mp,u_\mp)d\tau\Big)d\tau.
	\end{align*}
\end{lemma}
\begin{proof}
It will be the first time in this paper that the coordinate system varying for different situations has really mattered. In addition, it is worth presenting a complete and rigorous proof for this lemma based on the language of flow. Unnecessarily complicated process though it may seem, this proof will help to underline the importance of our treatment on various coordinates and simplify the proofs (such as without using flow maps) for the subsequent lemmas.

\smallskip

For any initial point $y=(y_1,y_2,y_3)\in\mathbb{R}^3$, we recall the flow $\psi_\pm(t,y)=(\psi_\pm^1(t,y),\psi_\pm^2(t,y),\psi_\pm^3(t,y))$ generated by $Z_\pm$ as in \eqref{flow definition}, which can be rewritten in the form analogous to \eqref{definitionforxi}, i.e.
\begin{equation*}
\begin{cases}
		&L_+ \psi_+^i(t,y_1,y_2,y_3)= 0,\\
		&\psi_+^i(t,y_1,y_2,y_3)\big|_{\Sigma_0} = y_i,
	\end{cases}\ \ \ \ \ \ 
	\begin{cases}
		&L_- \psi_-^i(t,y_1,y_2,y_3)= 0,\\
		&\psi_-^i(t,y_1,y_2,y_3)\big|_{\Sigma_0} = y_i.
	\end{cases}
\end{equation*}
Without loss of generality, let us assume $D=\partial_{x_1^\mp}$ for simplicity; actually all the cases $D\in \big\{\partial_{x_1^\mp},\partial_{x_2^\mp}, \partial_{u_\mp}\big\}$ can be discussed in a similar fashion. In order to convert the perturbation in derivatives to that in coordinates more naturally, we introduce another flow $\psi_{\pm,h}(t,y)=(\psi_{\pm,h}^1(t,y),\psi_{\pm,h}^2(t,y),\psi_{\pm,h}^3(t,y))$ given by 
\begin{equation*}
	\begin{cases}
		&L_+ \psi_{+,h}^i(t,y_1,y_2,y_3)= 0,\\
		&\psi_{+,h}^i(t,y_1,y_2,y_3)\big|_{\Sigma_0} = y_i+\delta_{1i}h,
	\end{cases}\ \ \ \ \ \ 
	\begin{cases}
		&L_-\psi_{-,h}^i(t,y_1,y_2,y_3)= 0,\\
		&\psi_{-,h}^i(t,y_1,y_2,y_3)\big|_{\Sigma_0} = y_i+\delta_{1i}h,
	\end{cases}
\end{equation*}
where $\delta_{1i}=1$ if $i=1$ and $\delta_{1i}=0$ if $i=2,3$. Hence, by using flow, it suffices to use the derivative $\partial_{y_1}$ in the differentiation, i.e. for a fixed $\tau$, we have 
\begin{align*}
&\partial_{y_1}\big(\nabla p\cdot\sqrt{1+|z_\mp\mp B_0|^2}\big)(\tau,\psi_\mp(\tau,y_1,y_2,y_3))\\
=&\lim_{h\to0}\frac{\big(\nabla p\cdot\sqrt{1+|z_\mp\mp B_0|^2}\big)(\tau,\psi_{\mp,h}(\tau,y_1,y_2,y_3))-\big(\nabla p\cdot\sqrt{1+|z_\mp\mp B_0|^2}\big)(\tau,\psi_\mp(\tau,y_1,y_2,y_3))}{h}.
\end{align*}

\smallskip

Explicit computations now lead us to
\begin{align*}
		&\ \ \ \ \Big|\frac{\big(\nabla p\cdot\sqrt{1+|z_\mp\mp B_0|^2}\big)(\tau,\psi_{\mp,h}(\tau,y_1,y_2,y_3))-\big(\nabla p\cdot\sqrt{1+|z_\mp\mp B_0|^2}\big)(\tau,\psi_\mp(\tau,y_1,y_2,y_3))}{h}\Big|\\
		&=\big|\partial_{y_1}(\nabla p\cdot\sqrt{1+|z_\mp\mp B_0|^2})(\tau,\psi_\mp(\tau,y_1,y_2,y_3))\cdot \partial_{y_1} \psi^1_\mp(\tau,y_1,y_2,y_3)\big|\\
		&\stackrel{\eqref{flow in integration}}{=}\Big|\partial_{x_1}(\nabla p\cdot\sqrt{1+|z_\mp\mp B_0|^2})(\tau,\psi_\mp(\tau,y_1,y_2,y_3))\cdot \big(1+ \partial_{y_1}\int_0^{+\infty} z^1_\mp(\tau,\psi_\mp(\tau,y_1,y_2,y_3))d\tau\big)\Big|\\
		&\leqslant \big(|\nabla^2 p|\cdot |z_\mp|+|\nabla p|\cdot |\nabla z_\mp|\big)\Big(1+ \int_0^{+\infty}|\nabla z_\mp|d\tau\Big)\\
		&\stackrel{\eqref{eq:pressure derivative bound}, \eqref{eq:Sobolev}}{\lesssim}\frac{\varepsilon^3}{(R+\tau)^{2\omega}}.
\end{align*}
Here in the last step we have also used a pivotal observation that 
\[\int_0^{+\infty}|\nabla z_\mp|d\tau=O(\varepsilon).\]
Reasoning herein is as follows: $z_\mp$ itself and the corresponding derivatives propagate along the characteristic vector field $L_\mp$, while the integration in time is carried out along the flow map $\psi_{\mp}(\tau,y)$ (which is normal to $L_\mp$ and also parallel to $L_\pm$) or in other words  along the characteristic vector field $L_\pm$. Consequently, the increasing spatial distance caused by these two different characteristic directions yields the decay of this integration in time. 
We remark here that this observation can be viewed as an analogue of the null form structure, and especially in this setting has the same flavor as the separation estimates for different family of waves (see Lemma \ref{lemma:separation estimate}) as well as the flux through characteristic hypersurfaces (see Lemma \ref{lemma:flux}).
	
\smallskip

Noticing that the dominant function $\dfrac{\varepsilon^3}{(R+\tau)^{2\omega}}$ is integrable in time, we can further apply the Lebesgue's dominated convergence theorem to commute the limit and the integral, i.e. 
\begin{align*}
	&\ \ \ \ \ \partial_{y_1}\Big(\int_0^{+\infty}\big(\nabla p\cdot\sqrt{1+|z_\mp\mp B_0|^2}\big)(\tau,\psi_{\mp}(\tau,y_1,y_2,y_3))d\tau\Big)\\
	&=\lim_{h\to 0}\!\frac{\int_0^{+\infty}\!\!\big(\nabla p\!\cdot\!\sqrt{1+|z_\mp\mp B_0|^2}\big)(\tau,\psi_{\mp,h}(\tau,y_1,y_2,y_3))d\tau\!\!-\!\!\int_0^{+\infty}\!\!\big(\nabla p\!\cdot\!\sqrt{1+|z_\mp\mp B_0|^2}\big)(\tau,\psi_{\mp}(\tau,y_1,y_2,y_3))d\tau}{h}\\
	&=\lim_{h\to 0}\frac{\int_0^{+\infty}\big(\big(\nabla p\cdot\sqrt{1+|z_\mp\mp B_0|^2}\big)(\tau,\psi_{\mp,h}(\tau,y_1,y_2,y_3))-\big(\nabla p\cdot\sqrt{1+|z_\mp\mp B_0|^2}\big)(\tau,\psi_{\mp}(\tau,y_1,y_2,y_3))\big)d\tau}{h}\\
	&=\int_0^{+\infty}\partial_{y_1}\big(\nabla p\cdot\sqrt{1+|z_\mp\mp B_0|^2}\big)(\tau,\psi_{\mp}(\tau,y_1,y_2,y_3))\cdot \partial_{y_1} \psi^1_{\mp}(\tau,y_1,y_2,y_3)	d\tau\\
	&=\int_0^{+\infty}\partial_{y_1}\big(\nabla p\cdot\sqrt{1+|z_\mp\mp B_0|^2}\big)(\tau,\psi_{\mp}(\tau,y_1,y_2,y_3))\cdot \Big(1+ \partial_{y_1}\int_0^\infty z^1_\mp(\tau,\psi_{\mp}(\tau,y_1,y_2,y_3))d\tau\Big)d\tau\\
	&=\int_0^{+\infty}\Big(\partial_{y_1}\nabla p\cdot\sqrt{1+|z_\mp\mp B_0|^2}+\nabla p\cdot\partial_{y_1}\sqrt{1+|z_\mp\mp B_0|^2}\Big)(\tau,\psi_{\mp}(\tau,y_1,y_2,y_3))\\
	&\ \ \ \ \ \ \ \ \times\Big(1+ \int_0^{+\infty}\partial_{y_1} z^1_\mp(\tau,\psi_{\mp}(\tau,y_1,y_2,y_3))d\tau\Big)d\tau.
\end{align*}
We then return to the usual characteristic coordinates to derive
\begin{align*}
	&\ \ \ \ \partial_{x_1^\mp}\Big(\int_0^{+\infty}\big(\nabla p\cdot\sqrt{1+|z_\mp\mp B_0|^2}\big)(\tau,x_1^\mp,x_2^\mp,u_\mp)d\tau\Big)\\
	&=\int_0^{+\infty}\Big(\partial_{x_1^\mp}(\nabla p)\cdot \sqrt{1+|z_\mp\mp B_0|^2} +\nabla p\cdot \partial_{x_1^\mp}\sqrt{1+|z_\mp\mp B_0|^2}\Big)(\tau,x_1^\mp,x_2^\mp,u_\mp)\\
	&\ \ \ \ \ \ \ \ \times\Big(1+ \int_0^{+\infty} \partial_{x_1^\mp} z^1_\mp(\tau,x_1^\mp,x_2^\mp,u_\mp)d\tau\Big)d\tau,
\end{align*}
which is the formal expression for $D=\partial_{x_1^\mp}$ as desired. 	
This completes the proof of the lemma.
\end{proof}

\smallskip

As a direct consequence of Lemma \ref{lemma:claim-proof1}, we obtain the following formal expression point-wisely: 
	\begin{align*}
	&\ \ \ \ \nabla\Big(\int_0^{+\infty}\big(\nabla p\cdot\sqrt{1+|z_\mp\mp B_0|^2}\big)(\tau,x_1^\mp,x_2^\mp,u_\mp)d\tau\Big)\\
	&=\int_0^{+\infty}\Big(\nabla^2 p\cdot \sqrt{1+|z_\mp\mp B_0|^2} +\nabla p\cdot \nabla \sqrt{1+|z_\mp\mp B_0|^2}\Big)(\tau,x_1^\mp,x_2^\mp,u_\mp)\\
	&\ \ \ \ \ \ \ \ \times\Big(\mathrm{I}+ \int_0^{+\infty} \nabla z_\mp(\tau,x_1^\mp,x_2^\mp,u_\mp)d\tau\Big)d\tau\\
	&=\int_0^{+\infty}\mathbf{R}_{\mp}^{(1)}(\tau,x_1^\mp,x_2^\mp,u_\mp)d\tau,
\end{align*}
which implies \eqref{eq:Lambda} for $|\beta|=1$.  
This is what we wanted to prove in the first step. 

\medskip

By using a classical induction scheme, it is time to establish \eqref{eq:Lambda} for $2\leqslant|\beta|\leqslant N_*+1$ in the rest steps. We now prepare a sequence of lemmas to ensure each step of induction.

\begin{lemma}\label{lemma:inL2-1}
	For all multi-indices $\beta$ with $1\leqslant\left|\beta\right|\leqslant N_*+1$,  we have 
\[\int_0^{+\infty}\mathbf{R}_{\mp}^{(\beta)}(\tau,x_1^\mp,x_2^\mp,u_\mp) d\tau\in L^2\big(\mathcal{F}_\pm,\langle u_\mp\rangle^{2\omega}d\mu_\mp\big).\]
\end{lemma}
\begin{proof} 
Repetition of the arguments leading \eqref{eq:Q2-intro} to \eqref{eq:Q2} directly yields this result.
\end{proof}

\smallskip

\begin{lemma}\label{lemma:inL2-2}
	For all multi-indices $\beta$ with $2\leqslant\left|\beta\right|\leqslant N_*+1$,  for all partial derivatives $D\in \big\{\partial_{x_1^\mp},\partial_{x_2^\mp}, \partial_{u_\mp}\big\}$,we have 
	\begin{equation*}
	D\Big(\int_0^{+\infty}\mathbf{R}_{\mp}^{(\beta-1)}(\tau,x_1^\mp,x_2^\mp,u_\mp)d\tau\Big)\in L^2\big(\mathcal{F}_\pm,\langle u_\mp\rangle^{2\omega}d\mu_\mp\big).
	\end{equation*}
\end{lemma}
\begin{proof}
	It suffices to consider the case where $D=\partial_{x_1^\mp}$, since the other  cases can be treated in the same way.  
Let us begin by using the definition of partial derivative and the product law of limit:
	\begin{align*}
		\mathbf{S}&:= \int_{\mathcal{F}_+}\Big|\partial_{x_1^\mp}\Big(\int_0^{+\infty}\mathbf{R}_{\mp}^{(\beta-1)}(\tau,x_1^\mp,x_2^\mp,u_\mp)d\tau\Big)\Big|^2\langle u_{\mp}\rangle^{2\omega}d\mu_\mp\\
		&=\int_{\mathbb{R}^3}\Big|\lim_{h\to 0}\int_0^{+\infty}\frac{\mathbf{R}_{\mp}^{(\beta-1)}(\tau,x_1^\mp+h,x_2^\mp,u_\mp)-\mathbf{R}_{\mp}^{(\beta-1)}(\tau,x_1^\mp,x_2^\mp,u_\mp)}{h}d\tau\Big|^2\langle u_{\mp}\rangle^{2\omega}d\mu_\mp\\
		&=\int_{\mathbb{R}^3}\lim_{h\to 0}\Big|\int_0^{+\infty}\frac{\mathbf{R}_{\mp}^{(\beta-1)}(\tau,x_1^\mp+h,x_2^\mp,u_\mp)-\mathbf{R}_{\mp}^{(\beta-1)}(\tau,x_1^\mp,x_2^\mp,u_\mp)}{h}d\tau\Big|^2\langle u_{\mp}\rangle^{2\omega}d\mu_\mp.
	\end{align*}
At this juncture, we take advantage of
Fatou's lemma as well as Newton-Leibniz formula to get
	\begin{align*}
		\mathbf{S} &\leqslant\liminf_{h\to 0}\int_{\mathbb{R}^3}\Big|\int_0^{+\infty}\frac{\mathbf{R}_{\mp}^{(\beta-1)}(\tau,x_1^\mp+h,x_2^\mp,u_\mp)-\mathbf{R}_{\mp}^{(\beta-1)}(\tau,x_1^\mp,x_2^\mp,u_\mp)}{h}d\tau\Big|^{2}\langle u_{\mp}\rangle^{2\omega}d\mu_\mp\\
		&\leqslant\liminf_{h\to 0}\underbrace{\int_{\mathbb{R}^3}\Big|\int_0^{+\infty}\int_{0}^{1}\partial_{x_1^\mp}\mathbf{R}_{\mp}^{(\beta-1)}(\tau,x_1^\mp+\theta h,x_2^\mp,u_\mp)d\theta d\tau\Big|^2 \langle u_{\mp}\rangle^{2\omega}d\mu_\mp}_{\mathbf{S}_h}.
	\end{align*}
	For all $h$, we have
	\begin{align*}
		\mathbf{S}_h&\leqslant\int_{\mathbb{R}^3}\Big|\int_{0}^{1}\int_0^{+\infty}\big|\partial_{x_1^\mp}\mathbf{R}_{\mp}^{(\beta-1)}(\tau,x_1^\mp+\theta h,x_2^\mp,u_\mp)\big|d\tau d\theta\Big|^2 \langle u_{\mp}\rangle^{2\omega}dx_1^\mp dx_2^\mp du_\mp\\
		&\leqslant\int_{0}^{1}\underbrace{\int_{\mathbb{R}^3}\Big|\int_0^{+\infty}\big|\partial_{x_1^\mp}\mathbf{R}_{\mp}^{(\beta-1)}(\tau,x_1^\mp+\theta h,x_2^\mp,u_\mp)\big|d\tau\Big|^2 \langle u_{\mp}\rangle^{2\omega}dx_1^\mp dx_2^\mp du_\mp}_{\mathbf{S}_{h,\theta}}d\theta.
	\end{align*}
	In $\mathbf{S}_{h,\theta}$, we observe that both $\theta$ and $h$ do not depend on $x_1^\mp$. Thus, by change of variable $x_1^\mp\rightarrow X_1^\mp:=x_1^\mp+\theta h$, we are able to derive
	\begin{align*}
		\mathbf{S}_{h,\theta}&\lesssim\int_{\mathbb{R}^3}\Big|\int_0^{+\infty} \big|\partial_{x_1^\mp}\mathbf{R}_{\mp}^{(\beta-1)}(\tau,X_1^\mp,x_2^\mp,u_\mp)\big| d\tau\Big|^2 \langle u_\mp\rangle^{2\omega}dX_1^\mp dx_2^\mp du_\mp\\
	&\lesssim\int_{\mathbb{R}^3}\Big|\int_0^{+\infty} \big|\mathbf{R}_{\mp}^{(\beta)}(\tau,X_1^\mp,x_2^\mp,u_\mp)\big| d\tau\Big|^2 \langle u_\mp\rangle^{2\omega}dX_1^\mp dx_2^\mp du_\mp\\
	&\stackrel{\eqref{eq:Q2}}{\lesssim} \varepsilon^6,
	\end{align*}
where in the last step we have used the proof from \eqref{eq:Q2-intro} to \eqref{eq:Q2}. 
Coming back to the estimates on $	\mathbf{S}_h$ and finally on $\mathbf{S}$, we obtain 
\begin{equation*}
	\mathbf{S}\lesssim\varepsilon^6.
\end{equation*}
It is now obvious that the lemma holds.
\end{proof}

\smallskip
 
\begin{lemma}\label{lemma:claim-proof2}
 For all multi-indices $\beta$ with $2\leqslant\left|\beta\right|\leqslant N_*+1$,  for all partial derivatives $D\in \big\{\partial_{x_1^\mp},\partial_{x_2^\mp}, \partial_{u_\mp}\big\}$, as vector fields in $L^2(\mathcal{F}_\pm,\langle u_{\mp}\rangle^{2\omega} d\mu_\mp)$, there holds
	\begin{align*}
		D\Big(\int_0^{+\infty}\mathbf{R}_{\mp}^{(\beta-1)}(\tau,x_1^\mp,x_2^\mp,u_\mp)d\tau\Big) \stackrel{L^2(\mathcal{F}_\pm,\langle u_{\mp}\rangle^{2\omega} d\mu_\mp)}{=\joinrel=\joinrel=\joinrel=} \int_0^{+\infty}D\mathbf{R}_{\mp}^{(\beta-1)}(\tau,x_1^\mp,x_2^\mp,u_\mp)d\tau,
	\end{align*}
	and therefore 
	\begin{align*}
		\nabla\Big(\int_0^{+\infty}\mathbf{R}_{\mp}^{(\beta-1)}(\tau,x_1^\mp,x_2^\mp,u_\mp)d\tau\Big) \stackrel{L^2(\mathcal{F}_\pm,\langle u_{\mp}\rangle^{2\omega} d\mu_\mp)}{=\joinrel=\joinrel=\joinrel=} \int_0^{+\infty}\mathbf{R}_{\mp}^{(\beta)}(\tau,x_1^\mp,x_2^\mp,u_\mp)d\tau.
\end{align*}
\end{lemma}
\begin{proof}
	By virtue of Lemma \ref{lemma:inL2-1} and Lemma \ref{lemma:inL2-2}, it suffices to show the following equality in the sense of distributions (without loss of generality, we may also assume $D=\partial_{x_1^\mp}$ as before):
		\begin{align*}
		D\Big(\int_0^{+\infty}\mathbf{R}_{\mp}^{(\beta-1)}(\tau,x_1^\mp,x_2^\mp,u_\mp)d\tau\Big) \stackrel{\mathcal{D}'(\mathcal{F}_\pm)}{=\joinrel=\joinrel=}\int_0^{+\infty}D\mathbf{R}_{\mp}^{(\beta-1)}(\tau,x_1^\mp,x_2^\mp,u_\mp)d\tau.
	\end{align*}
According to the distribution theory, this problem is thus reduced to showing that the following two pairings are equal for any vector field $\varphi\in \mathcal{D}(\mathbb{R}^3)$:
 \begin{align*}
 	\underbrace{\bigg\langle D\Big(\int_0^{+\infty}\mathbf{R}_{\mp}^{(\beta-1)}(\tau,x_1^\mp,x_2^\mp,u_\mp)d\tau\Big),\varphi(x_1^\mp,x_2^\mp,u_\mp)\bigg\rangle}_{\mathbf{U}}=\underbrace{\bigg\langle\int_0^{+\infty} D \mathbf{R}_{\mp}^{(\beta-1)}(\tau,x_1^\mp,x_2^\mp,u_\mp) d\tau,\varphi(x_1^\mp,x_2^\mp,u_\mp)\bigg\rangle}_{\mathbf{V}}.
 \end{align*}
 
 \smallskip
 
Based on Lemma \ref{lemma:inL2-1}, we begin by observing that the aforementioned integrals are locally integrable functions in $\mathbb{R}^3$ (and hence can be used to define distributions), i.e.
\begin{align*}
&\int_0^{+\infty}\mathbf{R}_{\mp}^{(\beta-1)}(\tau,x_1^\mp,x_2^\mp,u_\mp)d\tau\in L^2\big(\mathcal{F}_\pm,\langle u_\mp\rangle^{2\omega}d\mu_\mp\big)\subset L^2(\mathbb{R}^3,d\mu_\mp)\subset L^1_{\text{loc}}(\mathbb{R}^3,d\mu_\mp),\\
&\int_0^{+\infty}D\mathbf{R}_{\mp}^{(\beta-1)}(\tau,x_1^\mp,x_2^\mp,u_\mp)d\tau\in L^2\big(\mathcal{F}_\pm,\langle u_\mp\rangle^{2\omega}d\mu_\mp\big)\subset L^2(\mathbb{R}^3,d\mu_\mp)\subset L^1_{\text{loc}}(\mathbb{R}^3,d\mu_\mp).
\end{align*}
These two facts together with integrating by parts then lead us to a series of computations as follows:
\begin{align*}
	\mathbf{U}&=-\bigg\langle \int_0^{+\infty}\mathbf{R}_{\mp}^{(\beta-1)}(\tau,x_1^\mp,x_2^\mp,u_\mp)d\tau, D \varphi(x_1^\mp,x_2^\mp,u_\mp)\bigg\rangle\\
	&=-\int_{\mathbb{R}^3}\Big(\int_0^{+\infty}\mathbf{R}_{\mp}^{(\beta-1)}(\tau,x_1^\mp,x_2^\mp,u_\mp)d\tau\Big)\cdot D\varphi(x_1^\mp,x_2^\mp,u_\mp)d\mu_\mp\\
	&\stackrel{\text{Fubini}}{=}-\int_{[0,+\infty) \times \mathbb{R}^3}\mathbf{R}_{\mp}^{(\beta-1)}(\tau,x_1^\mp,x_2^\mp,u_\mp)\cdot D\varphi(x_1^\mp,x_2^\mp,u_\mp)d\tau d\mu_\mp\\
	&=-\int_{[0,+\infty)} \bigg(\int_{\mathbb{R}^3}\mathbf{R}_{\mp}^{(\beta-1)}(\tau,x_1^\mp,x_2^\mp,u_\mp)\cdot D\varphi(x_1^\mp,x_2^\mp,u_\mp) d\mu_\mp\bigg)d\tau \\
	&=\int_{[0,+\infty)} \bigg(\int_{\mathbb{R}^3}D\mathbf{R}_{\mp}^{(\beta-1)}(\tau,x_1^\mp,x_2^\mp,u_\mp)\cdot \varphi(x_1^\mp,x_2^\mp,u_\mp)d\mu_\mp\bigg)d\tau \\
	&=\int_{[0,+\infty)\times \mathbb{R}^3}D\mathbf{R}_{\mp}^{(\beta-1)}(\tau,x_1^\mp,x_2^\mp,u_\mp)\cdot \varphi(x_1^\mp,x_2^\mp,u_\mp)d\tau d\mu_\mp \\
	&\stackrel{\text{Fubini}}{=}\int_{\mathbb{R}^3}\Big(\int_{0}^{+\infty} D\mathbf{R}_{\mp}^{(\beta-1)}(\tau,x_1^\mp,x_2^\mp,u_\mp)d\tau\Big)\cdot \varphi(x_1^\mp,x_2^\mp,u_\mp)d\mu_\mp\\
	&=\mathbf{V},
\end{align*}	
where in particular we have applied Fubini's theorem twice to commute the order of integrals as noted. Evidently, the proof of this lemma will be complete once we show that these two applications of Fubini's theorem are both legitimate.
	
\smallskip	
	
In fact, for the first application, it is natural to consider the following spacetime integral:
\begin{align*}
		\mathbf{U'}&:=
		\int_{[0,+\infty)\times \mathbb{R}^3}\big|\mathbf{R}_{\mp}^{(\beta-1)}(\tau,x_1^\mp,x_2^\mp,u_\mp)\cdot D\varphi(x_1^\mp,x_2^\mp,u_\mp)\big|d\tau d\mu_\mp\\
		&\lesssim \int_{\mathbb{R}^3}\int_{\mathbb{R}}\big|\mathbf{R}_{\mp}^{(\beta-1)}(\tau,x_1^\mp,x_2^\mp,u_\mp)\big|\cdot|D\varphi(x_1^\mp,x_2^\mp,u_\mp)|dx_1^\mp dx_2^\mp du_\mp du_\pm\\
		&\lesssim
		\bigg(\underbrace{\int_{\mathbb{R}\times \mathbb{R}^3}\langle  u_\pm\rangle^{\omega}\big|\mathbf{R}_{\mp}^{(\beta-1)}(\tau,x_1^\mp,x_2^\mp,u_\mp)\big|^2 }_{\mathbf{U_1}}\bigg)^{\frac{1}{2}}
		\Bigg(\underbrace{\int_{\mathbb{R}\times \mathbb{R}^3}\frac{|D\varphi(x_1^\mp,x_2^\mp,u_\mp)|^2}{\langle u_\pm\rangle^{\omega}}}_{\mathbf{U_2}}\Bigg)^{\frac{1}{2}}.
	\end{align*}
Let us first bound $\mathbf{U_1}$:
	\begin{equation*}
		\mathbf{U_1}
		\lesssim
		\int_{\mathbb{R}\times \mathbb{R}^3}\langle u_\mp\rangle^{2\omega}\langle u_\pm\rangle^{\omega}\big|\mathbf{R}_{\mp}^{(\beta-1)}(\tau,x_1^\mp,x_2^\mp,u_\mp)\big|^2\lesssim \varepsilon^6.
	\end{equation*}
where we have used the fact   $\langle u_\mp\rangle^{2\omega}\geqslant 1$ and the bound of the right hand side of \eqref{eq:Q2-intro} in Lemma \ref{lemma:inL2-1}.
For $\mathbf{U_2}$, since $\varphi\in \mathcal{D}(\mathbb{R}^3)$, then $D\varphi\in \mathcal{D}(\mathbb{R}^3)\subset L^2(\mathbb{R}^3)$ and we can use the flux to bound
	\begin{equation*}
		\mathbf{U_2}= \int_{\mathbb{R}}\frac{1}{\langle u_\pm\rangle^{\omega}}\Big(\int_{C^{\pm,t}_{u_\pm}} |D \varphi(x_1^\mp,x_2^\mp,u_\mp)|^2  d\sigma_\pm \Big)du_\pm\lesssim \|\nabla \varphi\|_{L^1}.
	\end{equation*}
	Thus we arrive at the conclusion that 
	\begin{equation*}
		\mathbf{U'}<+\infty.
	\end{equation*}
In this way, we are able to employ Fubini's theorem for the first time. 

\smallskip
 
We turn to focus on the second application of Fubini's theorem. The  argument above can be repeated and we obtain  
\begin{align*}
	\mathbf{V'}&:=
	\int_{[0,+\infty)\times \mathbb{R}^3}\big|D\mathbf{R}_{\mp}^{(\beta-1)}(\tau,x_1^\mp,x_2^\mp,u_\mp)\cdot \varphi(x_1^\mp,x_2^\mp,u_\mp)\big|d\tau d\mu_\mp\\
	&\lesssim \int_{\mathbb{R}^3}\int_{\mathbb{R}}\big|D\mathbf{R}_{\mp}^{(\beta-1)}(\tau,x_1^\mp,x_2^\mp,u_\mp)\big|\cdot|\varphi(x_1^\mp,x_2^\mp,u_\mp)|dx_1^\mp dx_2^\mp du_\mp du_\pm\\
	&\lesssim
	\bigg(\underbrace{\int_{\mathbb{R}\times \mathbb{R}^3}\langle  u_\pm\rangle^{\omega}\big|D\mathbf{R}_{\mp}^{(\beta-1)}(\tau,x_1^\mp,x_2^\mp,u_\mp)\big|^2 }_{\mathbf{V_1}}\bigg)^{\frac{1}{2}}
	\Bigg(\underbrace{\int_{\mathbb{R}\times \mathbb{R}^3}\frac{|\varphi(x_1^\mp,x_2^\mp,u_\mp)|^2}{\langle u_\pm\rangle^{\omega}}}_{\mathbf{V_2}}\Bigg)^{\frac{1}{2}}.
\end{align*}
We can also derive
\begin{align*}
	\mathbf{V_1}
	&\lesssim
	\int_{\mathbb{R}\times \mathbb{R}^3}\langle u_\mp\rangle^{2\omega}\langle u_\pm\rangle^{\omega}\big|D\mathbf{R}_{\mp}^{(\beta-1)}(\tau,x_1^\mp,x_2^\mp,u_\mp)\big|^2\lesssim \varepsilon^6,\\
	\mathbf{V_2}&= \int_{\mathbb{R}}\frac{1}{\langle u_\pm\rangle^{\omega}}\Big(\int_{C^{\pm,t}_{u_\pm}} | \varphi(x_1^\mp,x_2^\mp,u_\mp)|^2  d\sigma_\pm \Big)du_\pm\lesssim \|\nabla \varphi\|_{L^1}.
\end{align*}
A further step shows  
\begin{equation*}
	\mathbf{V'}<+\infty,
\end{equation*}
which enables us to use Fubini's theorem again. 
Consequently, the lemma is recovered.
\end{proof}

\medskip

Based on Lemma \ref{lemma:inL2-1} and Lemma \ref{lemma:inL2-2}, we can start from \eqref{eq:Lambda} for $|\beta|=1$ and apply Lemma \ref{lemma:claim-proof2} for $k$ (with $1\leqslant k \leqslant N_*$) times to derive \eqref{eq:Lambda} for $2\leqslant|\beta|\leqslant N_*+1$. In this process, the $k$-th induction depends only on the conclusion for the $(k-1)$-th induction, and therefore \eqref{eq:Lambda} for $2\leqslant|\beta|\leqslant N_*+1$ can be obtained in a trivial manner. 
Up to now, we have completed the proof of Claim \ref{lemma:claim}. 

\smallskip

Thus the proof of Proposition \ref{prop:Sobolev} is now complete and Theorem \ref{thm:weightSobolev} follows immediately.

\bigskip

\subsection{Asymptotic property of scattering fields}\label{differential}
We are now ready to prove Theorem \ref{thm:deviation}. Also by symmetry considerations, it suffices to establish the asymptotic property of future scattering fields.  In other words, we only need to study the deviation of scattering fields $z_\pm(+\infty; x_1^\mp,x_2^\mp,u_\mp)$ from their corresponding linear propagating fields $z_\pm(0, x_1^\mp,x_2^\mp,u_\mp)$.

\smallskip

%Clearly,
Performing the same procedure as in the proof of Proposition \ref{prop:Sobolev} evidently gives rise to the following:
when $|\beta|=0$, it holds that 
\begin{align*}
	& \ \ \ \ 
	\int_{\mathcal{F}_+}|z_\pm(+\infty;x_1^\mp,x_2^\mp,u_\mp)-z_\pm(0,x_1^\mp,x_2^\mp,u_\mp)|^2\langle u_\mp\rangle^{2\omega}d\mu_\mp\\
	&\lesssim	\int_{\mathbb{R}^3} \Big|\int_{0}^{+\infty}\big(\nabla p\cdot\sqrt{1+|z_\mp\mp B_0|^2}\big)(\tau,x_1^\mp,x_2^\mp,u_\mp)d\tau\Big|^2\langle u_\mp\rangle^{2\omega}dx_1^\mp dx_2^\mp du_\mp\\
	&=\mathbf{P_2}\lesssim\varepsilon^4,
\end{align*}
and when $1\leqslant|\beta|\leqslant N_*+1$, it holds that 
\begin{align*}
&\ \ \ \
\int_{\mathcal{F}_\pm} \big|\nabla^{\beta}z_{\pm}(+\infty;x_1^\mp,x_2^\mp,u_\mp)-\nabla^{\beta}z_\pm(0,x_1^\mp,x_2^\mp,u_\mp)\big|^2\langle u_{\mp}\rangle^{2\omega}d\mu_\mp\\
&\lesssim \int_{\mathbb{R}^3} \Big|\nabla^{\beta}\Big(\int_0^{+\infty}\big(\nabla p\cdot\sqrt{1+|z_\mp\mp B_0|^2}\big)(\tau,x_1^\mp,x_2^\mp,u_\mp)d\tau\Big)\Big|^2\langle u_{\mp}\rangle^{2\omega}dx_1^\mp dx_2^\mp du_\mp\\
&=\mathbf{Q_2}\lesssim \varepsilon^6.
\end{align*}
The resulting estimates for all multi-indices $\beta$ with $0\leqslant |\beta|\leqslant N_*+1$ then become
\[\int_{\mathcal{F}_\pm} \big|\nabla^{\beta}z_{\pm}(+\infty;x_1^\mp,x_2^\mp,u_\mp)-\nabla^{\beta}z_\pm(0,x_1^\mp,x_2^\mp,u_\mp)\big|^2\langle u_{\mp}\rangle^{2\omega}d\mu_\mp\lesssim\varepsilon^4,\]
and therefore we are led to the conclusion that 
\[\big\|z_\pm(+\infty;x_1^\mp,x_2^\mp,u_\mp)-z_\pm(0,x_1^\mp,x_2^\mp,u_\mp)\big\|_{H^{N_*+1}(\mathcal{F}_\pm,\langle u_\mp\rangle^{2\omega}d\mu_\mp)}=O(\varepsilon^2).\]

\smallskip

Theorem \ref{thm:deviation} is now a direct consequence of what we have proved. 

\bigskip

\section{The proof of the inverse scattering theorem}\label{sec:inverse}
In this section, we come to address the ultimate conjecture \eqref{conjecture} of this paper by proving Theorem \ref{thm:inversetheorem}.

\medskip
\subsection{Differential of nonlinear scattering operators} 
We note that all the derivatives of scattering fields at infinities have been recovered in Section \ref{sec:proof}. Moreover, for $\#=a,b,c,d$, the continuity of $\mathcal{N}^{(\#)}$ at $\mathbf{0} \in H^{N_*+1}(\Sigma_0,\langle u_-\rangle^{2\omega}d\mu_-)\times H^{N_*+1}(\Sigma_0,\langle u_+\rangle^{2\omega}d\mu_+)$ follows immediately from Theorem \ref{thm:weightSobolev} and Theorem \ref{thm:deviation}. Therefore the key element of our inverse scattering theorem now is to verify the following relationship between the differentials of nonlinear scattering operators and linear propagating operators:
\begin{equation}\label{eq:differential}
d \,\mathcal{N}^{(\#)}	\big|_{\mathbf{0}}=\mathcal{L}^{(\#)},\ \ \text{where }  \#=a,b,c,d.
\end{equation}
For clarity, let us study these four cases and compute the differentials of their corresponding nonlinear scattering operators one by one. 
The following Figure \ref{fig:abcd} will help us to structure all cases in the proof. 
\begin{figure}[ht]
	\centering
	\includegraphics[width=1\textwidth]{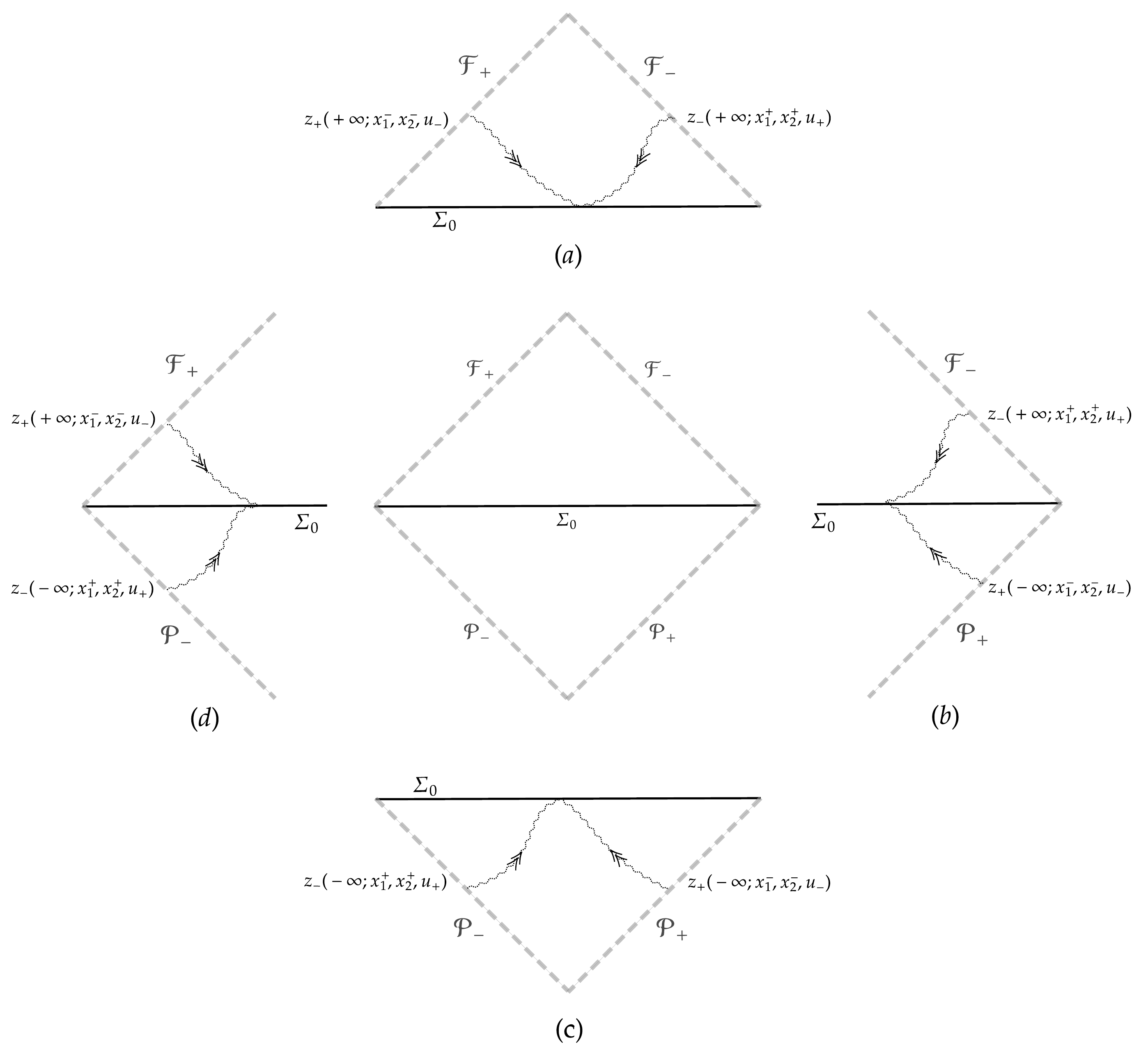}
	\caption{Structure of the inverse scattering theorem}
	\label{fig:abcd}
\end{figure}

\medskip

\subsubsection*{\bf{Proof for Case (a)}}
Notice that regarded as operators between 
\[H^{N_*+1}(\Sigma_0,\langle u_-\rangle^{2\omega}d\mu_-)\times H^{N_*+1}(\Sigma_0,\langle u_+\rangle^{2\omega}d\mu_+)\rightarrow H^{N_*+1}(\mathcal{F}_+,\langle u_-\rangle^{2\omega}d\mu_-) \times  H^{N_*+1}(\mathcal{F}_-,\langle u_+\rangle^{2\omega}d\mu_+),\]
the linear propagating operator $\mathcal{L}^{(a)}$ and the nonlinear scattering operator $\mathcal{N}^{(a)}$ have been established by  
\begin{align*}
\ \ \mathcal{L}^{(a)}:\  \big(z_+(0,x_1^-,x_2^-,u_-),z_-(0,x_1^+,x_2^+,u_+)\big)&\mapsto \big(z_+(0,x_1^-,x_2^-,u_-),z_-(0,x_1^+,x_2^+,u_+)\big),\\
\ \ \mathcal{N}^{(a)}:\  \big(z_+(0,x_1^-,x_2^-,u_-),z_-(0,x_1^+,x_2^+,u_+)\big)&\mapsto \big(z_+(+\infty;x_1^-,x_2^-,u_-),z_-(+\infty;x_1^+,x_2^+,u_+)\big).
\end{align*}
Thanks to \eqref{eq:initial energy}, we have 
\begin{equation}\label{eq:initialsobolev}
\big\|\big(z_+(0,x_1^-,x_2^-,u_-),z_-(0,x_1^+,x_2^+,u_+)\big)\big\|_{H^{N_*+1}(\Sigma_0,\langle u_-\rangle^{2\omega}d\mu_-)\times H^{N_*+1}(\Sigma_0,\langle u_+\rangle^{2\omega}d\mu_+)}=O(\varepsilon).
\end{equation}
This allows us to use Theorem \ref{thm:deviation} to infer that
\begin{align*}
\big\|\big(z_+(+\infty;x_1^-,x_2^-,u_-),z_-(+\infty;x_1^+,x_2^+,u_+)\big)&\\
-\big(z_+(0,x_1^-,x_2^-,u_-),z_-(0,x_1^+,x_2^+,u_+)\big)&\big\|_{H^{N_*+1}(\mathcal{F}_+,\langle u_-\rangle^{2\omega}d\mu_-) \times  H^{N_*+1}(\mathcal{F}_-,\langle u_+\rangle^{2\omega}d\mu_+)}=O(\varepsilon^2).
\end{align*}
Since $\varepsilon$ is arbitrary, by letting $\varepsilon\to 0$, we conclude that 
\begin{align*}
&\ \ \ \ \mathcal{N}^{(a)}\big(z_+(0,x_1^-,x_2^-,u_-),z_-(0,x_1^+,x_2^+,u_+)\big)\\
&=\mathcal{L}^{(a)}\big(z_+(0,x_1^-,x_2^-,u_-),z_-(0,x_1^+,x_2^+,u_+)\big)+o\big(z_+(0,x_1^-,x_2^-,u_-),z_-(0,x_1^+,x_2^+,u_+)\big).
\end{align*}
Then we can compute the differential of $\mathcal{N}^{(a)}$ at $\mathbf{0} \in H^{N_*+1}(\Sigma_0,\langle u_-\rangle^{2\omega}d\mu_-)\times H^{N_*+1}(\Sigma_0,\langle u_+\rangle^{2\omega}d\mu_+)$ as
\[d \,\mathcal{N}^{(a)}\big|_{\mathbf{0}}=\mathcal{L}^{(a)}.\]

\medskip
\subsubsection*{\bf{Proof for Case (b)}}
Note that regarded as operators between 
\[H^{N_*+1}(\Sigma_0,\langle u_-\rangle^{2\omega}d\mu_-)\times H^{N_*+1}(\Sigma_0,\langle u_+\rangle^{2\omega}d\mu_+)\rightarrow H^{N_*+1}(\mathcal{P}_+,\langle u_-\rangle^{2\omega}d\mu_-) \times  H^{N_*+1}(\mathcal{F}_-,\langle u_+\rangle^{2\omega}d\mu_+),\]
the linear propagating operator $\mathcal{L}^{(b)}$ and the nonlinear scattering operator $\mathcal{N}^{(b)}$ have been constructed by 
\begin{align*}
	\ \ \mathcal{L}^{(b)}:\  \big(z_+(0,x_1^-,x_2^-,u_-),z_-(0,x_1^+,x_2^+,u_+)\big)&\mapsto \big(z_+(0,x_1^-,x_2^-,u_-),z_-(0,x_1^+,x_2^+,u_+)\big),\\
	\ \ \mathcal{N}^{(b)}:\  \big(z_+(0,x_1^-,x_2^-,u_-),z_-(0,x_1^+,x_2^+,u_+)\big)&\mapsto \big(z_+(-\infty;x_1^-,x_2^-,u_-),z_-(+\infty;x_1^+,x_2^+,u_+)\big).
\end{align*}
Combined with \eqref{eq:initialsobolev}, Theorem \ref{thm:deviation} then leads us to
\begin{align*}
\big\|\big(z_+(-\infty;x_1^-,x_2^-,u_-),z_-(+\infty;x_1^+,x_2^+,u_+)\big)&\\
-\big(z_+(0,x_1^-,x_2^-,u_-),z_-(0,x_1^+,x_2^+,u_+)\big)&\big\|_{H^{N_*+1}(\mathcal{P}_+,\langle u_-\rangle^{2\omega}d\mu_-) \times  H^{N_*+1}(\mathcal{F}_-,\langle u_+\rangle^{2\omega}d\mu_+)}=O(\varepsilon^2).
\end{align*}
Letting $\varepsilon\to 0$ as before also gives
\begin{align*}
&\ \ \ \ \mathcal{N}^{(b)}\big(z_+(0,x_1^-,x_2^-,u_-),z_-(0,x_1^+,x_2^+,u_+)\big)\\
&=\mathcal{L}^{(b)}\big(z_+(0,x_1^-,x_2^-,u_-),z_-(0,x_1^+,x_2^+,u_+)\big)+o\big(z_+(0,x_1^-,x_2^-,u_-),z_-(0,x_1^+,x_2^+,u_+)\big).
\end{align*}
Hence we can compute the differential of $\mathcal{N}^{(b)}$ at $\mathbf{0} \in H^{N_*+1}(\Sigma_0,\langle u_-\rangle^{2\omega}d\mu_-)\times H^{N_*+1}(\Sigma_0,\langle u_+\rangle^{2\omega}d\mu_+)$ as
\[d \,\mathcal{N}^{(b)}\big|_{\mathbf{0}}=\mathcal{L}^{(b)}.\]

\medskip
\subsubsection*{\bf{Proof for Case (c)}}
Recall that regarded as operators between 
\[H^{N_*+1}(\Sigma_0,\langle u_-\rangle^{2\omega}d\mu_-)\times H^{N_*+1}(\Sigma_0,\langle u_+\rangle^{2\omega}d\mu_+)\rightarrow H^{N_*+1}(\mathcal{P}_+,\langle u_-\rangle^{2\omega}d\mu_-) \times  H^{N_*+1}(\mathcal{P}_-,\langle u_+\rangle^{2\omega}d\mu_+),\]
the linear propagating operator $\mathcal{L}^{(c)}$ and the nonlinear scattering operator $\mathcal{N}^{(c)}$ have been defined by 
\begin{align*}
	\ \ \mathcal{L}^{(c)}:\  \big(z_+(0,x_1^-,x_2^-,u_-),z_-(0,x_1^+,x_2^+,u_+)\big)&\mapsto \big(z_+(0,x_1^-,x_2^-,u_-),z_-(0,x_1^+,x_2^+,u_+)\big),\\
	\ \ \mathcal{N}^{(c)}:\  \big(z_+(0,x_1^-,x_2^-,u_-),z_-(0,x_1^+,x_2^+,u_+)\big)&\mapsto \big(z_+(-\infty;x_1^-,x_2^-,u_-),z_-(-\infty;x_1^+,x_2^+,u_+)\big).
\end{align*}
It is obvious that \eqref{eq:initialsobolev} and Theorem \ref{thm:deviation} are also applicable to yield
\begin{align*}
\big\|\big(z_+(-\infty;x_1^-,x_2^-,u_-),z_-(-\infty;x_1^+,x_2^+,u_+)\big)&\\
-\big(z_+(0,x_1^-,x_2^-,u_-),z_-(0,x_1^+,x_2^+,u_+)\big)&\big\|_{H^{N_*+1}(\mathcal{P}_+,\langle u_-\rangle^{2\omega}d\mu_-) \times  H^{N_*+1}(\mathcal{P}_-,\langle u_+\rangle^{2\omega}d\mu_+)}=O(\varepsilon^2),
\end{align*}
and it follows that 
\begin{align*}
&\ \ \ \ \mathcal{N}^{(c)}\big(z_+(0,x_1^-,x_2^-,u_-),z_-(0,x_1^+,x_2^+,u_+)\big)\\
&=\mathcal{L}^{(c)}\big(z_+(0,x_1^-,x_2^-,u_-),z_-(0,x_1^+,x_2^+,u_+)\big)+o\big(z_+(0,x_1^-,x_2^-,u_-),z_-(0,x_1^+,x_2^+,u_+)\big).
\end{align*}
Therefore we can derive the differential of $\mathcal{N}^{(c)}$ at $\mathbf{0} \in H^{N_*+1}(\Sigma_0,\langle u_-\rangle^{2\omega}d\mu_-)\times H^{N_*+1}(\Sigma_0,\langle u_+\rangle^{2\omega}d\mu_+)$:
\[d \,\mathcal{N}^{(c)}\big|_{\mathbf{0}}=\mathcal{L}^{(c)}.\]

\medskip
\subsubsection*{\bf{Proof for Case (d)}}
We also recall that regarded as operators between 
\[H^{N_*+1}(\Sigma_0,\langle u_-\rangle^{2\omega}d\mu_-)\times H^{N_*+1}(\Sigma_0,\langle u_+\rangle^{2\omega}d\mu_+)\rightarrow H^{N_*+1}(\mathcal{F}_+,\langle u_-\rangle^{2\omega}d\mu_-) \times  H^{N_*+1}(\mathcal{P}_-,\langle u_+\rangle^{2\omega}d\mu_+),\]
the linear propagating operator $\mathcal{L}^{(d)}$ and the nonlinear scattering operator $\mathcal{N}^{(d)}$ have been defined by 
\begin{align*}
	\ \ \mathcal{L}^{(d)}:\  \big(z_+(0,x_1^-,x_2^-,u_-),z_-(0,x_1^+,x_2^+,u_+)\big)&\mapsto \big(z_+(0,x_1^-,x_2^-,u_-),z_-(0,x_1^+,x_2^+,u_+)\big),\\
	\ \ \mathcal{N}^{(d)}:\  \big(z_+(0,x_1^-,x_2^-,u_-),z_-(0,x_1^+,x_2^+,u_+)\big)&\mapsto \big(z_+(+\infty;x_1^-,x_2^-,u_-),z_-(-\infty;x_1^+,x_2^+,u_+)\big).
\end{align*}
An entirely similar argument as the above three cases shows that 
\begin{align*}
\big\|\big(z_+(+\infty;x_1^-,x_2^-,u_-),z_-(-\infty;x_1^+,x_2^+,u_+)\big)&\\
-\big(z_+(0,x_1^-,x_2^-,u_-),z_-(0,x_1^+,x_2^+,u_+)\big)&\big\|_{H^{N_*+1}(\mathcal{F}_+,\langle u_-\rangle^{2\omega}d\mu_-) \times  H^{N_*+1}(\mathcal{P}_-,\langle u_+\rangle^{2\omega}d\mu_+)}=O(\varepsilon^2),
\end{align*} 
and hence 
\begin{align*}
&\ \ \ \ \mathcal{N}^{(a)}\big(z_+(0,x_1^-,x_2^-,u_-),z_-(0,x_1^+,x_2^+,u_+)\big)\\
&=\mathcal{L}^{(a)}\big(z_+(0,x_1^-,x_2^-,u_-),z_-(0,x_1^+,x_2^+,u_+)\big)+o\big(z_+(0,x_1^-,x_2^-,u_-),z_-(0,x_1^+,x_2^+,u_+)\big).
\end{align*}
Finally we are able to get the differential of $\mathcal{N}^{(d)}$ at $\mathbf{0} \in H^{N_*+1}(\Sigma_0,\langle u_-\rangle^{2\omega}d\mu_-)\times H^{N_*+1}(\Sigma_0,\langle u_+\rangle^{2\omega}d\mu_+)$:
\[d \,\mathcal{N}^{(d)}\big|_{\mathbf{0}}=\mathcal{L}^{(d)}.\]

\bigskip

Up to now, we can summarize what have proved in these four cases to be \eqref{eq:differential} as asserted. Consequently, for $\#=a,b,c,d$, it always holds that $d \,\mathcal{N}^{(\#)}\big|_{\mathbf{0}}$ is  invertible. Then according to the Inverse Function Theorem, $\mathcal{N}^{(\#)}$ is indeed a local diffeomorphism at $\mathbf{0} \in H^{N_*+1}(\Sigma_0,\langle u_-\rangle^{2\omega}d\mu_-)\times H^{N_*+1}(\Sigma_0,\langle u_+\rangle^{2\omega}d\mu_+)$. We have thus proved Theorem \ref{thm:inversetheorem} (Inverse scattering theorem for Alfv\'en waves).

\bigskip
\subsection{Concluding remarks}
We conclude this work by a brief discussion on our main contributions. 

\smallskip

Regarding innovation in the conclusion, this paper is devoted to addressing the ultimate conjecture \eqref{conjecture} on the inverse scattering problem of Alfv\'en waves in three dimensional ideal magnetohydrodynamics. The main statement of our inverse scattering theorem is consistent with the physical phenomenon that the Alfv\'en waves behave exactly in the same manner as their scattering fields detected by the faraway observers. We has provided rigorous mathematical explanations for this physical phenomenon, and in particular the scattering isomorphisms in our constructions can uniquely determine the Alfv\'en waves emanating from the plasma. In this way we are able to establish the relationship between initial Alfv\'en waves and their scattering fields, and hence to recover the Alfv\'en waves from knowledge of their scattering fields on infinities. 

\smallskip

Regarding innovation in the strategy, our main tools are the pressure estimates and the energy estimates. They arise naturally from the physical background of Alfv\'en waves and will be useful for further studies on the MHD system. The main inverse scattering theorem is obtained by reasoning along the following lines: 
we first introduce the scattering infinities and the scattering fields of Alfv\'en waves through geometric methods, then take advantage of the null structure inherent in the MHD system under strong magnetic backgrounds to derive the pressure estimates and the weighted energy estimates, and finally develop the energy approach proposed in \cite{He-Xu-Yu,Li-2021,Li-Yu} to study the scattering behavior as well as to construct the scattering isomorphisms of Alfv\'en waves. In this proof, the nonlinear nature of Alfv\'en waves has been fully considered.  
\smallskip

This work also leads to several directions for future research. For example, the inverse scattering theory in three-dimensional thin domain (as well as in the space dimension two scenario), which is of great interest and difficulties, will be treated separately in our companion papers.  Another very pressing direction is that since we have only discussed the ideal case of MHD system in this paper, it is natural to ask whether similar inverse scattering results hold for the viscous cases with viscosity coefficients and/or with resistivity coefficients. Moreover, we hope that our formulation of inverse scattering theory for Alfv\'en waves will serve as a guide for other wave scattering problems of a similar nature.

\end{document}